\documentclass[french]{smfart}
\usepackage[utf8]{inputenc}
\usepackage[english,francais]{babel}
\usepackage{smfthm}
\usepackage{mathrsfs}
\usepackage{graphics}
\usepackage{hyperref} 
\usepackage{amssymb, amsmath, mathabx}
\usepackage{lmodern}
\usepackage[all]{xy}
\usepackage{multicol}
\usepackage[shortlabels]{enumitem}
\usepackage{scalerel}
\usepackage{tikz-cd}
\makeindex
\DeclareUnicodeCharacter{00A0}{\relax}

\author{Antoine Ducros}
\address{Sorbonne Université, Université Paris-Diderot, CNRS, Institut de Mathématiques de Jussieu-Paris
Rive Gauche, IMJ-PRG, Département de mathématiques et applications, École normale supérieure, CNRS, PSL University,
F-75005, Paris, France
}

\email{antoine.ducros\at imj-prg.fr}
\urladdr{http://www.imj-prg.fr/$\sim$antoine.ducros}
\title{Dévisser, découper, éclater
et aplatir les espaces de Berkovich}
\alttitle{Blow-ups flattening in non-archimedean analytic geometry}
\dedicatory{À la mémoire de Laurent Gruson et Michel Raynaud}
\subjclass{14G22, 14G99}
\keywords{Espaces de Berkovich, aplatissement par éclatements}
\altkeywords{Berkovich spaces, Flattening by blowing-up}
\NumberTheoremsAs{subsection}
\SwapTheoremNumbers

\newcommand{\eg}{e.\@g.\@}

\newcommand{\ie}{i.\@e.\@}

\newcommand{\abs}[1]{\mathopen|#1\mathclose|}
\newcommand{\adhp}[2]{\overline{ \{#1\}}^{#2_{\mathrm {Zar}}}}
\newcommand{\adht}[2]{\overline{ #1}^{#2}}
\newcommand{\adhz}[2]{\overline{ #1}^{#2_{\mathrm {Zar}}}}
\newcommand{\akg}[3]{\mathfrak A_{#1,#2}^{\leq #3}}
\newcommand{\an}{^{\mathrm{an}}}
\newcommand{\al}{^{\mathrm{al}}}

\newcommand{\bkg}[3]{\mathfrak D_{#1,#2}^{\leq #3}}

\newcommand{\dimc}[2]{\mathrm{dim}_{\mathrm c}(#1,#2)}
\newcommand{\dimcg}[2]{\mathrm{dim}_{\mathrm c}(#1_{\mathrm G},#2)}

\newcommand{\dex}[2]{\mathsf E_{#2}(#1)}
\newcommand{\ecl}[2]{\mathsf B_{#2}(#1)}
\newcommand{\gpm}{^{\times}}
\newcommand{\grot}{_{\mathrm G}}
\newcommand{\hotimes}{\hat \otimes}
\newcommand{\hr}[1]{\mathscr H(#1)}

\newcommand{\inv}{^{-1}}

\newcommand{\pl}[2]{\mathsf P(#1/#2)}
\newcommand{\ql}[2]{\mathsf Q(#1/#2)}
\newcommand{\qln}[3]{\mathsf Q(#1/#2)_{\geq #3}}
\newcommand{\qlns}[3]{\mathsf Q(#1/#2)_{\geq #3}^{\mathrm {sat}}}

\newcommand{\spec}{\mathrm{Spec}\;}
\newcommand{\supp}[1]{\mathrm{Supp}(\mathscr #1)}

\newcommand{\A}{\mathbf A}

\newcommand{\N}{\mathbf Z_{\geq 0}}
\renewcommand{\P}{\mathbf P}
\newcommand{\Q}{\mathbf Q}
\newcommand{\R}{\mathbf R}
\newcommand{\Z}{\mathbf Z}

\renewcommand{\d}{\mathrm d}
\renewcommand{\phi}{\varphi}
\renewcommand{\epsilon}{\varepsilon}
\renewcommand{\leq}{\leqslant}
\renewcommand{\geq}{\geqslant}

\renewcommand{\labelitemi}{$\bullet$}

\renewcommand{\theequation}{\alph{equation}}
\SetEnumerateShortLabel{1}{\textnormal{(\arabic{enumi})}}
\SetEnumerateShortLabel{s}{\textnormal{(\arabic{enumi}$^\ast$)}}
\SetEnumerateShortLabel{i}{\textnormal{(\roman{enumi})}}
\SetEnumerateShortLabel{a}{\textnormal{(\alph{enumi})}}
\SetEnumerateShortLabel{A}{\textnormal{(\Alph{enumi})}}
\SetEnumerateShortLabel{2}{\textnormal{(\arabic{enumii})}}
\SetEnumerateShortLabel{j}{\textnormal{(\roman{enumii})}}
\SetEnumerateShortLabel{b}{\textnormal{(\arabic{enumi}\alph{enumii})}}
\SetEnumerateShortLabel{c}{\textnormal{(\alph{enumii})}}
\SetEnumerateShortLabel{B}{\textnormal{(\Alph{enumii})}}

\begin{document}

\thanks{Lors
de la rédaction de cet article, l'auteur a bénéficié du soutien de l'ANR à travers les projets {\em Valuations, combinatoire et théorie des modèles} (ANR-13-BS01-0006), 
et  {\em Définissabilité en géométrie non archimédienne}
(ANR-15-CE40-0008), ainsi que de celui de l'IUF dont il était membre junior d'octobre 2012 à octobre 2017. Il a aussi profité en mars 2019 de l'hospitalité de l'université hébraïque de Jérusalem, avec le soutien du projet ERC Consolidator  770922 (BirNonArchGeom) de Michael Temkin}

\begin{abstract}
Nous développons dans cet article des techniques d'aplatissement des faisceaux cohérents en géométrie de Berkovich, en nous
inspirant de la stratégie générale que Raynaud et Gruson ont mise en œuvre pour traiter le problème analogue en théorie des schémas.
Nous donnons ensuite quelques applications à l'étude des morphismes entre espaces analytiques compacts, et
obtenons notamment une description de l'image d'un tel morphisme. 
\end{abstract}

\begin{altabstract}
We develop in this article flattening techniques for coherent sheaves in the realm of Berkovich spaces; we are inspired by the general strategy that Raynaud and Gruson have used for dealing with the analogous problem in scheme theory. We then give some applications to the study of morphisms between compact analytic spaces; among other things, we get a description of the image of such a morphism. 
\end{altabstract}
\maketitle

\tableofcontents

\setcounter{section}{-1}

\section{Introduction}

Dans cet article, nous développons dans le contexte des espaces de Berkovich des techniques d'aplatissement
par éclatements inspirées de celles que Raynaud et Gruson
ont introduites dans leur travail fondateur \cite{raynaud-g1971}, puis
nous en donnons quelques applications, notamment à la description de l'image d'un morphisme arbitraire entre espaces
analytiques compacts ; nous espérons ultérieurement les utiliser
avec Amaury Thuillier pour étudier les images de squelettes. 

\subsection*{Les résultats de \cite{raynaud-g1971}}
Avant de décrire un peu plus précisément nos résultats, nous allons rappeler ceux de Raynaud et Gruson, en en donnant une version simplifiée. Soit $f\colon Y\to X$ un morphisme de type fini entre schémas noethériens,
soit $\mathscr F$ un faisceau cohérent sur $Y$
et soit $U$ le plus grand ouvert de $X$ au dessus duquel $\mathscr F$ est $X$-plat. 
Le théorème 5.2.2 de \cite{raynaud-g1971}
(dont Quentin Guignard a donné récemment une preuve
complètement différente dans \cite{guignard2018})
assure alors qu'il existe un sous-schéma fermé $F$ de $X$ de support
$X\setminus U$ tel que la transformée  stricte
correspondante
$\widetilde {\mathscr F}$ de $\mathscr F$ soit plate sur l'éclaté $\widetilde X$ de $X$ le long de $F$. 
(Rappelons que 
$\widetilde {\mathscr F}$ est
le quotient de l'image réciproque de $\mathscr F$ sur $Y\times_X \widetilde X$ par son sous-faisceau
des sections à support contenu ensemblistement dans $Y\times_X S\subset Y\times_X \widetilde X$, où $S$
est le diviseur de Cartier $\widetilde X\times_X F$
de $\widetilde X$.)

Notons que l'image du morphisme $\widetilde X\to X$ est égale à l'adhérence $\overline U$ de $U$
dans $X$ ; c'est donc lorsque $U$ est dense que ce théorème de Raynaud-Gruson est le plus puissant ; 
dans le cas opposé où $U$
est vide, $\widetilde X$ l'est aussi et le théorème est sans contenu. Ainsi, les techniques 
d'aplatissement ne créent pas de platitude \emph{ex-nihilo} : elles se contentent de propager un peu
la platitude déjà présente.

Posons $Y'=Y\times_X \widetilde X$, notons $\mathsf P$ l'ouvert de $Y$ formé des points en lesquels $\mathscr F$ est $X$-plat, et $\mathsf Q$ son
fermé complémentaire ; remarquons que $F$ est ensemblistement égal à $\overline{f(\mathsf Q)}$. 
Nous allons maintenant
reformuler
quelques propriétés de l'éclatement $\widetilde X\to X$
d'une façon qui peut sembler un peu laborieuse, mais
dont nous verrons l'intérêt plus bas.

\begin{enumerate}[a]
\item Le quotient de l'image réciproque de $\mathscr F$ sur $Y'$ par son
sous-faisceau des sections à support ensemblistement contenu dans $Y'\times_{\widetilde X}S$
est plat sur $\widetilde X$ . 
\item L'ouvert $Y'\setminus (Y'\times_{\widetilde X}S)$ de 
$Y'$ est situé au-dessus de $\mathsf P$. 
\item Le fermé $Y'\times_{\widetilde X} S$ de  $Y'$
est situé ensemblistement
au-dessus de
$f\inv (\overline{f(\mathsf Q)})$. 
\item L'image de $Y'\to Y$
contient $Y\setminus f\inv(\overline{f(\mathsf Q)})$. 
\item L'image de $Y'\times_{\widetilde X} S\to Y$ contient
ensemblistement $f\inv(G)$ pour toute composante irréductible $G$
de $\overline{f(\mathsf Q)}$ qui n'est pas une composante irréductible de $X$.
\end{enumerate}

\subsection*{Notre théorème principal}
Venons-en maintenant aux méthodes d'aplatissement mises au point
dans le présent travail. Précisons d'emblée
que nous ne nous intéresserons qu'à des morphismes
entre espaces compacts, mais  même 
avec cette restriction, 
on ne peut pas espérer disposer du même énoncé \emph{mutatis mutandis}
qu'en théorie des schémas, en raison 
du comportement sauvage de la topologie de Zariski en géométrie analytique. Par exemple, considérons un $3$-disque fermé
$D$ , un $3$-disque fermé $D'\subsetneq D$ et une courbe $C$ tracée sur $D'$ et Zariski-dense dans $D$. Soit $\Delta$ l'éclaté 
de $D'$ le long de $C$. Il est clair que le morphisme induit $\Delta\to D$ ne pourra pas être aplati par un éclatement de $D$ le long d'un de ses sous-espaces analytiques fermés ; il peut par contre être aplati par lui-même, mais il s'agit d'un éclatement de $D'$ et non de $D$.
On doit donc au minimum autoriser dans notre procédure d'aplatissement, en plus des éclatements, des immersions de domaines analytiques
compacts ; mais nous devrons
plus généralement
faire appel aux morphismes quasi-étales (à source compacte), et nous avons de bonnes raisons de
croire qu'on ne peut pas 
s'en passer (voir les commentaires en \ref{comment-final}). 

On fixe un corps ultramétrique
complet $k$. Soit $f\colon Y\to X$ un morphisme entre espaces $k$-analytiques compacts.
Nous appellerons (dans cette introduction uniquement)
 \emph{triplet admissible}
un triplet $(Z,S,V)$ constitué d'un espace analytique compact $Z$ muni d'un morphisme $Z\to X$, 
d'un diviseur de Cartier $S$ de 
$Z$
et d'un domaine analytique
compact $V$ de $Y\times_X Z$
tel qu'il existe une factorisation
\[Z=Z_m\to T_{m-1}\to Z_{m-1}\to\ldots \to Z_1\to T_0\to Z_0=X\]
et pour tout $i$ un diviseur de Cartier $\Sigma_i$ de $Z_i$
ainsi qu'un sous-espace analytique fermé $C_{i-1}$ de $T_{i-1}$ (si $i\geq 1$)
satisfaisant les propriétés suivantes : 
\begin{itemize}[label=$\diamond$] 
\item les $T_i$ et les $Z_i$ sont tous compacts ; 
\item pour tout $i$, le morphisme
$T_i\to Z_i$ est quasi-étale ; 
\item $\Sigma_0=\emptyset$ et $\Sigma_m=S$ ; 
\item pour tout $i$ supérieur ou égal à $1$, le sous-espace analytique fermé
$C_{i-1}$ de $T_{i-1}$ majore $T_{i-1}\times_{Z_{i-1}}\Sigma_{i-1}$, 
le morphisme $Z_i\to T_{i-1}$
est l'éclatement de centre $C_{i-1}$ et 
$\Sigma_i=Z_i\times_{Z_{i-1}}C_{i-1}$. 
\end{itemize}

Soit $\mathscr F$
un faisceau cohérent sur $Y$, soit
$\mathsf P$ le lieu de $X$-platitude de $\mathscr F$ 
et soit $\mathsf Q$ le
fermé complémentaire de $\mathsf P$
dans $Y$.
Notre théorème \ref{flat-eclat} (ou plutôt sa déclinaison dans un cas particulier bien précis, voir le commentaire
en \ref{aplat-strictosensu})
assure qu'il existe une famille finie 
$((Z_i,S_i,Y_i))_i$ de
triplets admissibles telle que :  

\begin{enumerate}[1]
\item pour tout $i$, le quotient de l'image réciproque de $\mathscr F$ sur $Y_i$
par son sous-faisceau des sections à support ensemblistement contenu dans $Y_i\times_{Z_i}S_i$
est plat sur $Z_i$ ; 

\item pour tout $i$, l'ouvert $Y_i\setminus (Y_i\times_{Z_i}S_i)$
de $Y_i$ est situé au-dessus de $\mathsf P$ ; 
\item pour tout $i$, le fermé
$Y_i\times_{Z_i}S_i$ 
de $Y_i$ est contenu dans l'image réciproque de $f\inv(f(\mathsf Q))$ ; 
\item la réunion des images des morphismes $Y_i\to Y$ contient $Y\setminus f\inv(f(\mathsf Q))$ ; 
\item la réunion des images des morphismes $Y_i
\times_{Z_i}S_i\to Y$
contient tout point de $\mathsf Q$
dont l'image sur $X$ n'est pas
adhérente à $f(\mathsf Q)\cap \mathsf A(X)$. 
\end{enumerate}
L'énoncé (5) appelle une explication. La notation $\mathsf A(X)$ désigne l'ensemble des «points
d'Abhyankar de rang maximal de $X$» (\ref{not-ax}), qui jouent en gros en géométrie analytique
le rôle des points génériques de composantes irréductibles en théorie des schémas. En particulier, 
si $T$ est un fermé de Zariski de $X$, l'adhérence de $T\cap \mathsf A(X)$ est la réunion 
des composantes irréductibles de $X$ contenues dans $T$ (\ref{sss-ax-compirr}).

On peut voir les propriétés (1), (2), (3), (4) et (5) comme des avatars respectifs des propriétés (a), (b), (c), (d) et (e)
énoncées plus haut dans le cadre schématique (en
ce qui concerne l'analogie entre (5) et (e), elle se fonde sur la remarque qui précède ; notons également qu'ici $\mathsf Q$ est compact, si bien que $\overline{f(\mathsf Q)}=f(\mathsf Q)$).

\subsection*{À propos de nos méthodes}
Nous suivons dans les grandes lignes la stratégie de Raynaud et Gruson,
fondée sur ce qu'ils appellent les \emph{dévissages} ;  
mais sa mise en œuvre dans le contexte des espaces de Berkovich plutôt que des schémas se heurte à un nombre important de difficultés. 

Une partie d'entre elles ont été résolues par l'auteur dans des
articles antérieurs. Par exemple, c'est au chapitre 5 de \cite{ducros2018}
que sont jetées les bases de la théorie
des dévissages en géométrie analytique (que nous rappelons et complétons à la section \ref{s-complement-devissages}). 

D'autres le sont
dans le présent article. Ainsi, à la section \ref{s-assassin}, nous définissons et étudions les composantes 
immergées d'un espace analytique, et plus généralement d'un faisceau cohérent
$\mathscr F$ sur celui-ci. Plus précisément, nous définissons la notion 
d'{\emph{assassin}} de $\mathscr F$
(la terminologie est empruntée à Raynaud et Gruson)
puis, dans la foulée, celle de \emph{composante assassine} 
(définition \ref{defi-ass}) de $\mathscr F$ ; les composantes irréductibles 
de $\supp F$ sont toujours des composantes assassines de $\mathscr F$, et ce sont les autres composantes assassines de $\mathscr F$ qu'on appelle composantes immergées. Nous décrivons le
comportement
des composantes assassines par restriction à un domaine analytique (prop. \ref{prop-equiv-ass}), et
établissons quelques principes GAGA à leur sujet (prop. \ref{pro-compass-aff} et plus généralement lemme \ref{lem-ass-gaga}). Nous utilisons par ailleurs les composantes assassines pour introduire un analogue analytique de la notion d'adhérence schématique, que nous avons appelé (faute d'avoir trouvé une dénomination moins ambiguë) \emph{adhérence analytique} (lemme-définition \ref{lem-adh-analytique}).

Et à la section \ref{s-ideal}, nous montrons (théorème \ref{theo-id-coeff})
l'existence
d'un «idéal des coefficients» associé à un sous-espace analytique fermé
$Z$ 
de la source d'un morphisme quasi-lisse et compact $Y\to X$
à fibres géométriquement irréductibles (ceci implique notamment que
l'ensemble des points $x$ de $X$ tels que la fibre $Y_x$
soit entièrement contenue dans $Z$ est un fermé de Zariski de
$X$). Ce théorème joue un rôle absolument crucial 
dans notre procédure d'aplatissement, car les centres de nos
éclatements sont définis par des idéaux de coefficients bien choisis,
mais pour l'appliquer nous nous heurtons à un obstacle
de taille :
en effet,
si
nos dévissages font apparaître des morphismes quasi-lisses,
ils ne fournissent aucun contrôle sur les composantes
irréductibles géométriques de
leurs fibres (contrairement à ce qui se passe
chez Raynaud et Gruson). C'est ce problème qui est
à l'origine de la présence de morphismes quasi-étales
dans notre procédure d'aplatissement. En effet, il
peut être contourné, mais \emph{localement pour la topologie
quasi-étale sur la base}. Plus précisément, nous
montrons que
si $Y\to X$ est un morphisme plat et à fibres géométriquement réduites (par exemple, un morphisme quasi-lisse)
entre espaces
$k$-analytiques compacts, il existe un espace $k$-analytique compact $X'$,
une surjection quasi-étale $X'\to X$, et un
\emph{découpage}
de $Y\times_X X'$ au-dessus de $X'$, 
c'est-à-dire un recouvrement
fini $(Y_i)$ de $Y\times_X X'$ par des domaines analytiques
compacts tels que les fibres
de $Y_i\to X'$ soient
géométriquement connexes
pour tout $i$. C'est le théorème 
\ref{exist-decoupages}, qui
repose pour l'essentiel sur le théorème
3.5 de \cite{ducros2019}. Ce dernier est une sorte de substitut
au théorème de la fibre réduite de Bosch, Lütkebohmert et Raynaud
(\cite{frg4}, thm. 2.5), substitut
qui
a l'avantage d'être valable pour les espaces de
Berkovich non nécessairement stricts et sans hypothèse
d'équidimensionalité relative ; il est 
démontré dans \cite{ducros2019}
sans utiliser le théorème original de la fibre réduite
ni recourir à la géométrie formelle.

\subsection*{Quelques applications}
À la section \ref{s-appli}, nous donnons quelques 
applications de notre théorème principal. 

Nous montrons tout d'abord
un théorème d'équidimensionalisation (thm. \ref{thm-equidim}) dont l'énoncé
est le suivant : soit $Y\to  X$ un morphisme entre espaces $k$-analytiques compacts ; supposons
qu'il existe un entier $\delta$ tel que la réunion des fibres 
de $Y\to X$ qui sont purement de dimension $\delta$
soit dense dans $Y$ ; 
il existe alors une famille finie $((Z_i, S_i, Y_i))_i$ de triplets admissibles telle que pour tout $i$, l'adhérence analytique $Y'_i$ de
$Y_i\setminus Y_i\times_{Z_i}S_i$ dans $Y_i$ soit purement de dimension relative $\delta$ sur $Z_i$, et telle que les images 
des morphismes $Y'_i\to Y$ recouvrent $Y$. 

Nous abordons ensuite le problème de la factorisation d'un morphisme 
nulle part plat par un fermé de Zariski du but. Donnons quelques
explications. Soit $f\colon Y\to X$ un morphisme de type fini entre schémas noethériens, avec $X$ intègre ; si $f$ n'est plat en aucun point de $Y$, l'image $f(Y)$ ne contient pas le point générique de $X$ et son adhérence de Zariski est donc un fermé \emph{strict}
de $X$. Un tel résultat n'a aucune chance d'être vrai en géométrie analytique (penser à une courbe tracée sur un petit bidisque et Zariski-dense dans un bidisque plus grand), mais nous montrons
l'assertion suivante (thm. \ref{thm-dominant-factor}) qui
en est un substitut : soit $Y\to  X$ un morphisme entre espaces $k$-analytiques compacts ; supposons que $X$ est réduit et que
le lieu de $X$-platitude de $Y$ est vide ; il existe alors une famille finie $((Z_i, S_i, Y_i))_i$ de triplets admissibles
telle que pour tout $i$, le morphisme $Y_i\to Z_i$ se factorise ensemblistement par $S_i$, et telle que les images
des morphismes $Y_i\to Y$ recouvrent $Y$. 

Grâce à cet énoncé, nous démontrons le dernier
résultat de l'article par récurrence sur la dimension de la base et réduction au cas génériquement plat (qui lui-même se traite par une application directe du théorème principal). Ce dernier résultat s'énonce comme suit (thm. \ref{thm-images-genplat} et thm. \ref{thm-lipshitz} ; voir aussi les commentaires en \ref{ss-comment-images} sur son intérêt et ses limites) : soit $f\colon Y\to X$ un morphisme entre espaces $k$-analytiques compacts, avec $X$ réduit ; il existe une famille finie de morphismes $(f_i\colon V_i\to X)_i$ telle que chacun des $f_i$ s'écrive comme une composée (dans n'importe quel ordre) d'éclatements, de morphismes quasi-étales, et d'immersions fermées, tous à source et but compacts et réduits, et telle que l'image $f(Y)$
soit la réunion des $\bigcup_i f_i(V_i)$ ; si de plus le lieu de platitude de $f$ est dense dans $Y$, on peut faire en sorte que chacun des $f_i$ soit simplement
une composée d'éclatements et de morphismes quasi-étales, tous à source et but compacts et réduits (les immersions fermées ne sont plus nécessaires). 

\subsection*{Liens avec des travaux antérieurs}
Différents auteurs ont déjà mis en œuvre des travaux d'aplatissement
dans
le contexte de la géométrie analytique. 

\begin{itemize}[label=$\diamond$]

\item
Bosch et Lütkebohmert ont utilisé 
la variante \emph{formelle}
des techniques de Raynaud et Gruson pour fabriquer
des modèles formels plats de morphismes plats entre espaces strictement analytiques compacts (\cite{frg2}, thm. 5.2 ; voir aussi 
\cite{abbes2010}, thm. 5.8.1). Mais ces résultats sont en un sens «orthogonaux» aux nôtres, car les éclatements des auteurs ne modifient que la fibre spéciale, et n'ont aucun effet aplatissant au niveau générique ;  et nous ne les utilisons pas dans notre preuve (ni dans nos travaux antérieurs auxquels nous faisons parfois appel). 

\item Hironaka a démontré divers théorèmes d'aplatissement en géométrie analytique \emph{complexe}, par des techniques qui n'ont rien à voir
avec celles de Raynaud et Gruson (ni
par conséquent avec les nôtres).

Dans \cite{hironaka1973}  il démontre un théorème d'aplatissement \emph{local} pour un morphisme entre espaces analytiques, dont notre théorème \ref{flat-eclat} peut être plus ou moins vu comme un analogue ultramétrique. Il en donne ensuite une application en géométrie analytique \emph{réelle} (en utilisant une procédure de complexification), et plus précisément à l'étude des images de parties semi-analytiques réelles sous un morphisme analytique réel propre. 

Dans \cite{hironaka1975}, 
il démontre un théorème d'aplatissement \emph{global}
pour un morphisme \emph{propre} $Y\to X$ entre espaces analytiques (avec $X$ réduit) : 
si $\mathscr F$ est un faisceau cohérent sur $Y$, il prouve l'existence d'un
morphisme propre et biméromorphe $X'\to X$ tel que
la transformée stricte correspondante de $\mathscr F$
soit plate sur $X'$. Il
serait très intéressant
d'établir un résultat analogue dans le contexte analytique ultramétrique.  
\end{itemize}

Par ailleurs, les images des morphismes entre espaces analytiques compacts ont fait l'objet de nombreuses études du point de vue de la théorie des modèles des corps valués, et nous pensons que la description de ces images que nous obtenons et avons décrite ci-dessus est une reformulation géométrique d'un théorème obtenu par Cluckers et Lipshitz dans \cite{cluckers-l2017}. Nous discutons ce point plus avant 
en \ref{comment-final} (dont la lecture requiert une certaine familiarité avec la théorie des modèles).

\subsection*{Remerciements}
Je tiens à manifester ma reconnaissance aux rapporteurs anonymes pour leur lecture très attentive du manuscrit et les corrections qu'ils ont suggérées.

\section{Rappels, conventions, notations}

\subsection{La cadre général}
On fixe pour toute la suite du texte un corps ultramétrique complet $k$ (la valuation peut être triviale). Nous travaillerons avec la notion d'espace $k$-analytique \emph{au sens de Berkovich} et considérerons comme connues les bases de la théorie, exposées par exemple dans les textes fondateurs 
\cite{berkovich1990} et \cite{berkovich1993} (notre définition d'espace $k$-analytique sera celle de \cite{berkovich1993} ; 
avec ce choix, les espaces définis dans \cite{berkovich1990}
sont les \emph{bons} espaces $k$-analytiques, c'est-à-dire ceux dont tout point possède un voisinage affinoïde).

On fixe également un sous-groupe $\Gamma$ de $\R_{>0}$ tel que $\abs {k\gpm}\cdot \Gamma\neq \{1\}$ (autrement dit
$\Gamma$ est non trivial si la valeur absolue de $k$ est triviale).
Nous utiliserons la notion d'espace $k$-analytique 
\emph{$\Gamma$-strict}, pour laquelle nous renvoyons au paragraphe 3.1
de \cite{ducros2018}. Intuitivement, un espace $k$-analytique est $\Gamma$-strict s'il peut être défini en utilisant uniquement des paramètres réels appartenant à $\Gamma$. Ainsi, si $\Gamma=\{1\}$ (ce qui n'est possible que lorsque $k$ n'est pas trivialement valué) les espaces $k$-analytiques $\Gamma$-stricts sont les espaces strictement $k$-analytiques ; et si $\Gamma=\R_{>0}$, tout espace $k$-analytique est $\Gamma$-strict. 

\subsection{Notations de base}
Soit $X$ un espace $k$-analytique. 

Si $E$ est une partie de $X$ on notera $\adht EX$ l'adhérence de $E$ dans $X$. 

Si $L$ est une extension complète de $k$, on notera $X_L$
l'espace $L$-analytique obtenu par extension des scalaires de $k$ à $L$.
On dispose d'un $k$-morphisme $X_L\to X$ qui est surjectif (\emph{cf.}
\cite{ducros2007}, 0.5). Lorsque $X=\mathscr M(A)$, ce morphisme est induit
par la flèche $A\to A\hotimes_k L$, qui est fidèlement plate (\cite{berkovich1993}, lemme 2.1.2).

Si $x$ est un point de $X$ on notera $\hr x$ le corps résiduel complété de $x$. 
Si $f\colon Y\to X$ est un morphisme d'espaces $k$-analytiques, on notera $f\inv(x)$ ou $Y_x$ la fibre de $f$ en $x$ ; c'est un 
espace $\hr x$-analytique.

\subsection{Topologie et G-topologie}
Soit $X$ un espace $k$-analytique. Il est fourni avec une topologie au sens usuelle, et une topologie de Grothendieck ensembliste qui la raffine et est appelée la G-topologie ; le site correspondant à cette dernière
est noté $X\grot$ (\cite{berkovich1993}, 1.3). Le site $X\grot$ est muni d'un faisceau de $k$-algèbres que nous noterons $\mathscr O_X$ (nous nous écartons ici des notations de Berkovich, qui le note $\mathscr O_{X\grot}$ et désigne par $\mathscr O_X$ sa restriction au site topologique usuel de $X$). Ce faisceau $\mathscr O_X$ est cohérent (\cite{ducros2009}, Lemma 0.1 ; J. Poineau nous a signalé une erreur dans la preuve de ce lemme, mais elle est aisément réparable -- voir la note de bas de page de
\cite{ducros2018}, 1.3.1).  Nous appellerons \emph{faisceau cohérent}
sur $X$
tout $\mathscr O_X$-module cohérent sur le site $X\grot$. 

\subsection{Notations relatives aux faisceaux cohérents}
Soit $\mathscr F$ un faisceau cohérent sur un 
espace $k$-analytique $X$. 

\subsubsection{Supposons que $X$ est bon} Soit $x$ un point de $X$. 
On notera $\mathscr F_x$ la colimite des $\mathscr F(U)$ où $U$ parcourt l'ensemble
des voisinages \emph{ouverts}
de $x$ dans $X$ (autrement dit, $\mathscr F_x$ est la fibre en $x$ du faisceau $\mathscr F$ restreint au site topologique usuel de $X$) ; 
en particulier, $\mathscr O_{X,x}$ désignera la colimite des $\mathscr O_X(U)$ où $U$ parcourt l'ensemble
des voisinages ouverts 
de $x$ dans $X$ ; c'est un
anneau local excellent hensélien (\cite{berkovich1993}, thm. 2.1.4 et 2.1.5 pour la noethérianité et l'hensélianité, et \cite{ducros2009}, thm.
2.13 pour l'excellence). 

\subsubsection{On ne suppose plus que $X$ est bon}
Soit $Y$ un espace analytique défini sur une extension complète $L$ de $k$, 
et soit $f\colon Y\to X$ un $k$-morphisme. Le faisceau cohérent $f^*\mathscr F$ sera également noté $\mathscr F_Y$ (s'il n'y a pas
d'ambiguïté sur $f$) ; en particulier, si $V$ est un domaine
analytique de $X$, la restriction de $\mathscr F$ à $V$
sera notée $\mathscr F_V$.

\subsection{GAGA}
Soit $X=\mathscr M(A)$ un espace $k$-affinoïde.

\subsubsection{}
Nous désignerons par $X\al$ le schéma $\spec A$
(ici «al» est une abréviation de «algébrique») ; il est muni d'un morphisme d'espaces localement annelés
$X\to X\al$.

\subsubsection{Analytification
d'un schéma}
Soit $\mathscr X$ un $A$-schéma de type fini. La catégorie des diagrammes commutatifs 

\[\begin{tikzcd}
Y\ar[r]\ar[d]&\mathscr X\ar[d]\\
X\ar[r]&X\al
\end{tikzcd}
\]
où $Y$ est un bon espace analytique défini sur une extension complète de $k$, où $Y\to X$ est un $k$-morphisme 
d'espaces analytiques et où $Y\to \mathscr X$ est un morphisme d'espaces localement annelés, admet un objet final 
\[\begin{tikzcd}
\mathscr X\an\ar[r]\ar[d]&\mathscr X\ar[d]\\
X\ar[r]&X\al
\end{tikzcd}
\]
(\cite{berkovich1993}, prop. 2.6.2). L'espace $\mathscr X\an$ est $k$-analytique
et est appelé \emph{l'analytifié}
de $\mathscr X$ (notons que si $\mathscr X=X\al$ alors $\mathscr X\an=X$). Le morphisme
structural $\mathscr X\an \to \mathscr X$ est surjectif, et les applications
qu'il induit entre anneaux locaux sont 
régulières (\cite{ducros2009}, thm.3.3 ; la platitude et la surjectivité sont dues à Berkovich, \cite{berkovich1993}
prop. 2.6.2). L'image sur $\mathscr X$ d'un point $x$ de $\mathscr X\an$ sera notée $x\al$ ; plus généralement, l'image directe sur $\mathscr X$ d'une partie $E$ de $\mathscr X\an$ sera notée $E\al$, et l'image réciproque sur 
$\mathscr X\an$ d'une partie $F$ de $\mathscr X$ sera notée $F\an$. 

\subsubsection{Analytification d'un faisceau cohérent}
Si $\mathscr F$ est un faisceau cohérent sur $\mathscr X$, le faisceau cohérent induit sur $\mathscr X\an$ sera noté $\mathscr F\an$. Lorsque $\mathscr X$ est un $A$-schéma propre, $\mathscr F\mapsto \mathscr F\an$ est une équivalence de catégories dont on notera $\mathscr F\mapsto \mathscr F\al$ un quasi-inverse ; on trouvera une preuve de cette assertion dans l'appendice A de \cite{poineau2010}. Toutefois, le seul cas qui nous servira ici sera celui où $\mathscr X=X\al$,
dans lequel le résultat est essentiellement dû à Kiehl et Tate ; le faisceau $\mathscr F\al$ est alors simplement le faisceau cohérent
sur $X\al$ associé au $A$-module de type fini $\mathscr F(X)$, et $\mathscr F$ est lui-même le faisceautisé de $V\mapsto \mathscr F(X)\otimes_A \mathscr O_X(V)$ (\emph{cf}. \cite{berkovich1993}, 1.2).

\subsubsection{}
Soit $Y$ un espace $k$-analytique quelconque, soit $V$ un domaine affinoïde de $Y$
et soit $x$ un point de $V$. L'image de $x$ par l'application $V\to V\al$ sera parfois notée $x_V\al$ pour bien indiquer qu'on voit $x$ comme un point de $V$ (il pourrait y avoir une ambiguïté si $Y$ était lui-même affinoïde, ou plus généralement
si c'était l'analytifié d'un schéma de type fini sur une algèbre affinoïde). 

\subsection{Validité de certaines propriétés en un point}
On étudie de manière systématique au chapitre 2 de \cite{ducros2018}
(et notamment aux sections 2.2, 2.3 et 2.4)
la notion de validité en un point d'un espace $k$-analytique d'une propriété
$\mathsf P$ d'algèbre commutative sujette à un certain nombre d'axiomes.
Nous allons simplement traiter ici le cas de
quelques
propriétés spécifiques. On fixe un
espace $k$-analytique $X$, 
un point
$x$ de $X$, et un diagramme
$\mathscr F\to \mathscr G\to \mathscr H$ 
dans
la catégorie des faisceaux cohérents sur $X$.

\subsubsection{Validité en un point : la définition}\label{ss-valid-prop}
On dira que $X$ est régulier, (resp. normal, resp. réduit) en $x$ si pour tout bon domaine analytique $V$ de $X$ contenant $x$
l'anneau local $\mathscr O_{V,x}$ est régulier, (resp. \ldots) ;
on dit que $\mathscr F$ est $S_n$ (resp. libre de rang $n$) en $x$
si pour tout bon domaine analytique $V$
de $X$ contenant $x$
le $\mathscr O_{V,x}$-module
$\mathscr F_{V,x}$
est $S_n$ (resp. \ldots) ; on dit
que $\mathscr F\to \mathscr G$ est injectif (resp. surjectif, resp. bijectif) en $x$ si
pour tout bon domaine analytique $V$ de $X$ contenant $x$ le morphisme 
$\mathscr F_{V,x}\to \mathscr G_{V,x}$ est injectif (resp. \ldots) ; on dit que
$\mathscr F\to \mathscr G\to \mathscr H$
est exacte en $x$ si pour tout bon domaine analytique $V$ de $X$ contenant $x$
la suite $\mathscr F_{V,x}\to \mathscr G_{V,x}\to \mathscr H_{V,x}$ est exacte. 

\subsubsection{}Ces définitions sont en fait plus souples qu'il n'y paraît, car il
suffit à chaque fois de vérifier la condition requise
pour \emph{un} bon domaine analytique donné $V$ de $X$ contenant $x$ : 
voir les lemmes 2.4.1 et 2.4.3
de \cite{ducros2018}, qui reposent sur le fait que si $W\subset V$ sont deux bons domaines analytiques de $X$ contenant $x$,
le morphisme $\mathscr O_{V,x}\to \mathscr O_{W,x}$ est régulier (\cite{ducros2009}, thm. 3.3 ; la platitude
est une conséquence directe de \cite{berkovich1990}, prop. 2.2.4). 
Il s'ensuit que si $V$ est un domaine analytique quelconque de $X$ contenant $x$ et si $\mathsf P$
est l'une des propriétés évoquées ci-dessus alors 
$X$ (ou $\mathscr F$, ou $\mathscr F\to \mathscr G$, ou $\mathscr F\to \mathscr G\to \mathscr H$)
satisfait $\mathsf P$
en $x$ si et seulement si $V$ (ou $\mathscr F_V$, ou $\mathscr F_V\to \mathscr G_V$,
ou $\mathscr F_V\to \mathscr G_V \to \mathscr H_V$) satisfait $\mathsf P$ en $x$. 

\subsubsection{Validité globale}
On dira que $X$ (ou $\mathscr F$, ou $\mathscr F\to \mathscr G$, ou $\mathscr F\to \mathscr G\to \mathscr H$)
satisfait $\mathsf P$ si celle-ci est satisfaite en tout point de $X$. Avec
toutefois une exception : nous continuerons à employer
l'expression «libre de rang $n$ en tout point» et réserverons évidemment
«libre de rang $n$» au cas des faisceaux \emph{globalement}
libres. Mentionnons que le faisceau cohérent $\mathscr F$ est libre de rang $n$ en tout point de $X$ si et seulement si il 
est G-localement libre de rang $n$ (cela découle de la définition et de \cite{berkovich1993}, prop. 1.3.4 (iii)). 

\subsubsection{GAGA pour les propriétés d'algèbre commutative}
\label{ss-valid-gaga}
Soit $\mathscr X$ un schéma de type fini sur une algèbre $k$-affinoïde. On dispose
(\cite{ducros2018}, Lemma
2.4.5) d'un principe GAGA
concernant la validité en un point d'une des propriétés ci-dessus ; par exemple, $\mathscr X\an$ est régulier en un point $x$ si et seulement si $\mathscr X$ est régulier en $x\al$ ; si $\mathscr F$ est un faisceau cohérent sur $\mathscr X$
alors $\mathscr F\an$ est $S_n$ en $x$ si et seulement si $\mathscr F$ est $S_n$ en $x\al$, etc. Le point clef
de la preuve est la régularité des morphismes entre anneaux locaux induits par $\mathscr X\an \to \mathscr X$. 

\subsubsection{Espaces réguliers, normaux
et réduits}
On déduit du principe GAGA ci-dessus
qu'un espace $k$-analytique $X$ est régulier (resp. normal, resp. réduit) si et seulement si $\mathscr O_X(V)$ est régulier (resp. normal, resp. réduit) pour tout domaine affinoïde $V$ de $X$ ; il suffit de le vérifier pour un recouvrement affinoïde ensembliste de $X$. 

\begin{rema}
Soit $X$
un espace $k$-analytique.
Les notions d'injectivité, surjectivité et bijectivité d'un morphisme
de faisceaux cohérents sur $X$, ou d'exactitude d'une suite de faisceaux
cohérents sur $X$, sont \emph{a priori}
ambigües : pour chacune d'elles on pourrait en effet se référer ou bien 
la définition de \ref{ss-valid-prop} (par
la validité en tout point de $X$) ou bien à celle de la théorie générale
des faisceaux sur un site. 
Mais il n'y a pas d'inquiétude à avoir : les deux définitions sont en fait
à chaque fois équivalentes
(\cite{ducros2018}, 2.5.5).

\end{rema}

\subsection{Topologie de Zariski}
Soit $X$ un espace $k$-analytique. 

\subsubsection{Définition de la topologie de Zariski}Si $\mathscr I$ est un faisceau cohérent d'idéaux sur $X$, on note $\mathrm V(\mathscr I)$
l'ensemble des $x\in X$ tels que $f(x)=0$ pour toute section $f$ de $\mathscr I$. Les parties de $X$
de la forme $\mathrm V(\mathscr I)$, où $\mathscr I$ est un faisceau cohérent d'idéaux sur $X$, sont les 
fermés d'une topologie
plus grossière que la topologie usuelle et appelée la \emph{topologie de Zariski de $X$}. 
Si $E$ est une partie de $X$, on notera $\adhz EX$ l'adhérence de $E$ dans $X$ pour la topologie de Zariski. 
Lorsque 
$X$ est affinoïde, la topologie de Zariski de $X$ est l'image réciproque de la topologie de Zariski de $X\al$ par l'application naturelle $X\to X\al$. 
Un espace $k$-analytique sera dit irréductible s'il est irréductible
\emph{pour la topologie de Zariski}. 

\subsubsection{Sous-espaces
analytiques fermés}
Soit $\mathscr I$ un faisceau cohérent d'idéaux sur $X$. La catégorie formée des
couples $(Y,f)$ où $Y$ est un espace
analytique défini sur une extension complète de $k$
et où $f\colon Y\to X$ est un $k$-morphisme tel que $f\inv\mathscr I=0$, admet un objet final
$(Z,\iota)$ ; 
l'espace $Z$ est $k$-analytique, son ensemble sous-jacent est le lieu des zéros $\mathrm V(\mathscr I)$ de
l'idéal $\mathscr I$ (c'est donc un fermé de
Zariski de
$X$), et $\iota_*\mathscr O_Z=\mathscr O_X/\mathscr I$
(\emph{cf.} \cite{ducros2018}). 
On dit que $Z$ est le \emph{sous-espace analytique fermé} associé à $\mathscr I$ ; une
\emph{immersion fermée} est un morphisme $S\to X$ induisant un isomorphisme entre $S$ et un sous-espace
analytique fermé de $X$. 

On
munit l'ensemble des sous-espaces analytiques fermés de $X$
de la relation d'ordre opposée à la relation d'inclusion entre les faisceaux cohérents d'idéaux qui les définissent. De cette façon, l'application qui envoie un sous-espace analytique fermé de $X$ sur son ensemble sous-jacent est croissante. 

Si $Z$ est un sous-espace analytique fermé et si $V$ est un domaine analytique de $X$, on notera par abus $Z\cap V$ 
le produit fibré $Z\times_X V$ ; c'est à la fois un sous-espace analytique fermé de $V$ et un domaine analytique de $Z$. 

\subsubsection{Structure réduite et caractère G-local
de la toplogie de Zariski}Si $Y$ est un fermé de Zariski de $X$, il possède une plus petite structure de sous-espace analytique fermé ; le faisceau
cohérent d'idéaux associé est celui des sections de $\mathscr O_X$ qui s'annulent \emph{ponctuellement}
sur $Y$ ; la structure correspondante est aussi l'unique structure de sous-espace analytique fermé sur $Y$ pour laquelle celui-ci est réduit ; 
on l'appelle la \emph{structure réduite} sur $Y$. Lorsque $Y=X$, l'espace correspondant est noté $X_{\mathrm{red}}$ ; le faisceau cohérent d'idéaux de $\mathscr O_X$
qui définit  $X_{\mathrm{red}}$ est celui des fonctions localement nilpotentes. Pour une référence, on pourra 
se reporter à \cite{ducros2009}, prop. 4.1. 

L'existence de cette structure réduite et sa canonicité ont une conséquence importante : le fait pour une partie de $X$
d'être un fermé (ou un ouvert) de Zariski est de nature G-locale (\emph{cf.}\cite{ducros2018}, 1.3.19). 

\subsubsection{Lieux de validité des propriétés d'algèbre
commutative}
Le lieu de validité d'une des propriétés considérée au \ref{ss-valid-prop}
est toujours un ouvert de Zariski de $X$ : en effet on peut par ce qui précède raisonner G-localement, donc
se ramener au cas affinoïde ; l'assertion résulte alors du principe GAGA rappelé en
\ref{ss-valid-gaga}
et de l'excellence des algèbres affinoïdes (\cite{ducros2009}, thm. 2.13).

\subsection{Théorie de la dimension}
On dispose en géométrie analytique ultramétrique
d'une théorie de la dimension avec des variantes globale, locale
et relative. Elle est due à Berkovich, et nous renvoyons le lecteur au chapitre 2 de \cite{berkovich1990}
ainsi qu'à la section 1 de \cite{ducros2007} pour les définitions de base.

\subsubsection{}
Soit $X$ un espace $k$-analytique. La dimension de $X$ sera notée $\dim X$
et sa dimension en un point $x\in X$ sera notée $\dim_x X$.
Il résulte des définitions que si $V$ est un domaine analytique de $X$ 
on a $\dim_x V=\dim_x X$ pour tout $x\in V$
et que $\dim X$
est égal à $\sup_{x\in X}\dim_x X$
(la dimension de l'espace vide est par convention
égale à $-\infty$).
Si $L$ est une extension complète de $k$,
si $x\in X$ et si $y$ est un antécédent de $x$ sur $X_L$
on a $\dim_y X_L=\dim_x X$ (\cite{ducros2007}, prop. 1.2.2).

\subsubsection{Espaces équidimensionnels}
Si $d$ est un entier, on dit que $X$ est \emph{purement de dimension $d$} si $\dim_x X=d$ pour tout $x\in X$. 
Si $X$ est purement de dimension $d$ tout domaine analytique $V$ de $X$ l'est aussi (car on aura $\dim_x V=\dim_x X=d$ pour tout $x\in V$) et si $X$ est non vide, $X$ est de dimension $d$. On dit que $X$ est \emph{équidimensionnel}
s'il est purement de dimension $d$ pour un certain $d$ (uniquement déterminé si $X\neq \emptyset$). 

\subsubsection{À propos de la dimension
relative}
Si $f\colon Y\to X$ est un morphisme entre espaces $k$-analytiques et si $y$ est un point de $Y$ dont on note $x$ l'image sur $X$, la \emph{dimension relative
de $Y$ sur $X$} en $y$ est par définition la dimension de $f\inv(x)$ en $y$, qu'on notera aussi $\dim_y f$. 
La fonction $Y\to \N, y\mapsto \dim_y f$, est semi-continue supérieurement pour la topologie de Zariski de $Y$
(\cite{ducros2007}, thm. 4.9).

\subsubsection{La quantité $d_k(x)$}
Si $x\in X$, on notera $d_k(x)$ la somme du rang rationnel de $\abs{\hr x\gpm}/\abs {k\gpm}$ et du degré de transcendance résiduel de $\hr x$ sur $k$ ; l'intérêt de cette quantité est la formule
$\dim X=\sup_{x\in X}d_k(x)$ (\cite{berkovich1993}, lemme 2.5.2).

\subsection{Les composantes
irréductibles en géométrie analytique}\label{rappel-irred}
La topologie de Zariski d'un espace $k$-analytique $X$ n'a aucune raison d'être noethérienne en général (elle l'est dès que $X$ est compact). On peut néanmoins développer une théorie des composantes irréductibles en géométrie de Berkovich. C'est l'objet du chapitre 4 de \cite{ducros2009} ; nous allons expliquer brièvement ici de quoi il retourne. 

\subsubsection{Le cas affinoïde}
Si $X$ est affinoïde sa topologie est noethérienne, et il possède donc des composantes irréductibles au sens habituel ; un fermé de Zariski $T$ de $X$ en est une composante irréductible si et seulement si $T\al$ est une composante irréductible de $X\al$. 

\subsubsection{Le cas général}
Dans le cas général on définit les composantes irréductibles de $X$ comme les parties de la forme $\adhz TX$ où $T$ est une composante irréductible d'un domaine affinoïde de $X$. On démontre qu'un domaine affinoïde de $X$ ne rencontre qu'un nombre fini de composantes irréductibles de $X$, que les composantes irréductibles de $X$ sont exactement les fermés de Zariski irréductibles maximaux de $X$, et que tout fermé de Zariski irréductible de $X$ est contenu dans une composante irréductible de $X$ (\cite{ducros2009}, thm. 4.20). 

\subsubsection{Composantes irréductibles et domaines
analytiques}\label{compirr-domaines}
On démontre également (\emph{ibid}, prop. 4.22) que si $V$ est un domaine analytique de $X$ alors : 
\begin{enumerate}[1] 
\item pour toute composante irréductible $Z$ de $X$, l'intersection $Z\cap V$ est une union $\bigcup_i T_i$ de composantes irréductibles de $V$ ; pour tout $i$ on a $\dim T_i=\dim Z$ et $\adhz {T_i}X=Z$ ;

\item pour toute composante irréductible $T$ de $V$, l'adhérence $\adhz TX$ est une composante irréductible de $X$, de même dimension que $T$, et $T$ est une composante irréductible de $\adhz TX\cap V$. 
\end{enumerate}

On en déduit que $\adhz VX$ est la réunion des composantes irréductibles de $X$ rencontrant $V$. 

\subsubsection{Composantes irréductibles et théorie
de la dimension}
Si $X$ est un espace $k$-analytique irréductible il est purement de dimension $d$ pour un certain $d\geq 0$, et tout fermé de Zariski strict de $X$ est de dimension $<d$ (\cite{ducros2009}, cor. 4.14). 
Si l'on ne suppose plus $X$ irréductible, l'entier $\dim_x X$ est égal pour tout point $x$ de $X$ au maximum des dimensions des composantes irréductibles de $X$ contenant $x$ (\emph{ibid.}, lemme 4.21). 

\subsubsection{Composantes irréductibles et codimension}
Soit $X$ un espace $k$-analytique, soit $Y$ un fermé de Zariski de $X$.
Si $x$ est un point de $Y$ alors la \emph{codimension de $Y$ dans $X$ en $x$}
est définie comme 
dans le cas des schémas par la formule
\[\inf_{Z\;\text{comp. irr. de }Y, \;\;x\in Z}\;\;\;\;\;
 \sup_{T\;\text{comp. irr. de}\;X, \;\;Z\subset T} (\dim T-\dim Z)\]
(en particulier, elle est égale à $\dim X-\dim Y$ dès que $X$ et $Y$ sont irréductibles). 

On dispose d'un principe GAGA pour la codimension : si $\mathscr X$ est un schéma de type
fini sur une algèbre $k$-affinoïde et si $\mathscr Y$ est un fermé de Zariski de $\mathscr X$, 
alors pour tout point $x$ de $\mathscr  Y\an$
la codimension de $\mathscr Y\an$ dans $\mathscr X\an$
en $x$ est égale à celle de $\mathscr Y$ dans $\mathscr X$ en $x\al$
(\cite{ducros2018}, cor. 2.7.13). 

\subsubsection{Composantes irréductibles d'un espace normal}
Si $X$ est un espace $k$-analytique normal, ses composantes irréductibles sont ses
composantes connexes (\cite{ducros2009}, propr. 5.14).

\subsubsection{Adhérence des ouverts de Zariski}
Si $U$ est un ouvert de Zariski de $X$, son adhérence $\adhz UX$ est égale à la réunion des composantes irréductibles de $X$ que $U$ rencontre (c'est immédiat) ; elle coïncide
avec l'adhérence topologique $\adht UX$ de $U$ dans $X$ (\cite{ducros2009}, Lemma 5.1.12), et sa
formation commute à l'extension des scalaires (\emph{ibid}., Cor. 5.14).  Il découle par ailleurs des propriétés (1) et (2) de \ref{compirr-domaines}
que pour tout domaine analytique $V$ de $X$ on a
$\adhz UX\cap V=\adhz{(U\cap V)}V$. Nous nous servirons à plusieurs reprises du cas particulier
suivant : si $U$ est Zariski-dense dans $X$ alors
$U\cap V$ est Zariski-dense dans $V$ pour tout domaine
analytique $V$ de $X$. 

\subsubsection{Principe GAGA
pour les composantes irréductibles} 
Soit $\mathscr X$ un schéma de type fini sur une algèbre affinoïde et soit
$(\mathscr X_i)_i$
la famille des composantes irréductibles de $\mathscr X$. La famille $(\mathscr X_i\an)_i$ est
alors la famille des composantes irréductibles de $\mathscr X\an$ ; en particulier, si $\mathscr X$
est irréductible alors $\mathscr X\an$ est irréductible (\cite{ducros2018}, prop. 2.7.16). 

\begin{lemm}\label{lem-ouvzar-connexe}
Soit $X$ un espace $k$-analytique irréductible et soit $U$ un ouvert de Zariski non vide de $X$. 
L'espace $U$ est irréductible. 
\end{lemm}

\begin{proof}
Soit $Z$ le complémentaire de $U$ dans $X$. Soit $\widetilde X$ le normalisé de $X$
(\cite{ducros2018}, déf. 5.10 et thm. 5.13) et soit $\widetilde U$
(resp. $\widetilde Z$)  l'image réciproque de $U$ (resp. $Z$) sur $\widetilde X$. 
L'ouvert $\widetilde U$ de $\widetilde X$ est le normalisé de $U$ (\emph{op. cit.}, lemme 5.11)
et il suffit pour conclure de démontrer que $\widetilde U$
est connexe (ceci entraînera l'irréductibilité de $\widetilde U$
puisqu'il est normal, et donc
celle de $U$ par surjectivité du morphisme de normalisation). 
Or l'espace $\widetilde X$ est normal, et il est connexe
par irréductibilité de $X$ (\emph{op. cit.}, prop. 5.16) ; par ailleurs, 
la surjectivité du morphisme de 
normalisation entraîne que $\widetilde Z$ est un fermé de Zariski strict, et partant d'intérieur vide, de $\widetilde X$. 
En utilisant l'avatar en théorie de Berkovich du théorème d'extension de Riemann  (\cite{berkovich1990}, prop. 3.3.14, elle-même
fondée sur la version rigide-analytique de ce résultat due à Lütkebohmert, \cite{lutkebohmert1974}), on en déduit que $\widetilde U$
est connexe. 
\end{proof}

\begin{lemm}\label{zar-transitive}
Soit $X$ un espace $k$-analytique, soit $U$ un ouvert de Zariski de $X$, et soit $V$ un ouvert fermé de $U$. L'ouvert $V$ est alors
un ouvert de Zariski de $X$.
\end{lemm}
\begin{proof}
Soit $W$ un domaine affinoïde de $X$. L'ouvert fermé $V$ de $U$ est réunion de composantes connexes de $U$, et est \emph{a fortiori}
réunion de composantes irréductibles de $U$. Par conséquent, $V\cap W$ est réunion de composantes irréductibles de $U\cap W$. 
Par le principe GAGA pour les composantes irréductibles
appliqué à l'ouvert $(W\cap U)\al$ de $W\al$, les composantes
irréductibles de $W\cap U$ sont précisément les parties de la forme $\mathscr Y\an$ où $\mathscr Y$ est une composante
irréductible de $(W\cap U)\al$ ; elles sont donc en nombre fini, et sont
Zariski-constructibles. L'intersection $W\cap V$ est donc Zariski-constructible et ouverte dans $W$ ; c'est en
conséquence un ouvert de Zariski de $W$
(\cite{berkovich1993}, Cor. 2.6.6). Ceci valant pour tout $W$, l'ouvert $V$ est un ouvert de Zariski de $X$.
\end{proof}

\begin{rema}
Les  lemmes précédents ne sont pas simplement des énoncés triviaux de topologie générale (comme le seraient leurs analogues schématiques). En effet, si $U$ est un domaine analytique d'un espace $k$-analytique $X$, la topologie de Zariski de $U$ est en général plus fine que la topologie induite par la topologie de $X$, et ce même dans le cas où $U$ est lui-même un ouvert de Zariski de $X$. Par exemple, si la valeur absolue de $k$ n'est pas triviale, une suite d'éléments de $k$ tendant vers l'infini en valeur absolue définit un fermé de Zariski de $\A^{1,\mathrm{an}}_k$ qui n'est pas la trace d'un fermé de Zariski de $\P^{1,\mathrm{an}}_k$ (ce phénomène n'a rien de spécifiquement non archimédien : on le rencontre tout aussi bien en géométrie analytique complexe). 
\end{rema}

\subsection{Quelques propriétés géométriques}\label{rapp-plat}
Soit $X$ un espace $k$-analytique. Les assertions suivantes 
sont équivalentes : 
\begin{enumerate}[i]
\item $X_L$ est connexe (resp. irréductible, resp. réduit) pour toute extension
complète $L$ de $k$. 
\item Il existe une extension complète $L$ de $k$ qui est séparablement
close\footnote{Précisions que si la valuation de $L$ n'est pas triviale, $L$ est séparablement close si et seulement si elle est algébriquement close.}
(resp. séparablement close, resp. parfaite) telle que 
$X_L$ soit connexe (resp. irréductible, resp. réduit). 
\end{enumerate}
Si elles sont satisfaites, on dit que $X$
est \emph{géométriquement}
connexe (resp. irréductible, resp. réduit).

Pour les preuves, on pourra
consulter \cite{ducros2009} et plus précisément le thm. 7.14 pour la connexité, le thm. 7.16 pour l'irréductibilité, 
et la prop. 6.3 pour le caractère réduit. 

\subsection{Platitude en géométrie analytique}
Nous allons brièvement présenter la théorie de la platitude en géométrie de Berkovich, qui
est introduite et étudiée systématiquement dans \cite{ducros2018}. 
Soit $f\colon Y\to X$ un morphisme entre espaces $k$-analytiques, soit $y$ un point de $Y$ et soit $x$ son image sur $X$. 
Soit $\mathscr F$ un faisceau cohérent sur $Y$. 

\subsubsection{Le cas bon}
Supposons que $Y$ et $X$ sont bons. On dit que $\mathscr F$ est \emph{naïvement} $X$-plat en $y$ si $\mathscr F_y$ est un $\mathscr O_{X,x}$-module plat. Cette notion présente un grave défaut, qui explique le choix de l'adverbe «naïvement» : elle n'est stable ni par changement de base $k$-analytique, ni par extension des scalaires -- nous décrivons en détail un contre-exemple à la section 4.4 de \cite{ducros2018}. 
Pour y remédier, on \emph{impose}
la stabilité par extension des scalaires et changement de base
$k$-analytique : on dit que $\mathscr F$ est $X$-plat
(ou plat sur $X$) en $y$ si pour tout bon espace analytique
$X'$ défini sur une extension complète quelconque de $k$,
tout $k$-morphisme $X'\to X$ et tout antécédent
$y'$ de $y$ sur $Y':=Y\times_X X'$,
le faisceau cohérent $\mathscr F_{Y'}$ est
naïvement
$X'$-plat en $y'$.
Soit $U$ 
un bon domaine analytique de $X$ contenant $x$ et soit $V$
un bon domaine analytique de $Y\times_X U$ contenant $y$. 
Les anneaux locaux $\mathscr O_{V,y}$ et $\mathscr O_{U,x}$
sont respectivement plats sur $\mathscr O_{Y,y}$
et $\mathscr O_{X,x}$ (c'est une conséquence de la proposition
2.2.4 de \cite{berkovich1990}). On en déduit facilement que
$\mathscr F$ est $X$-plat en $y$
si et seulement si $\mathscr F_V$ est $U$-plat en $y$
(cette équivalence est énoncée au paragraphe 
4.1.7  \cite{ducros2018} mais M. Daylies
a remarqué que la preuve qu'on en donne est erronée ; pour un
raisonnement correct, voir 
les errata mentionnés dans la bibliographie). 

\subsubsection{Le cas général}
On ne suppose plus que $Y$ et $X$ sont bons. On dit que $\mathscr F$ est $X$-plat en $y$ si pour tout bon domaine analytique $U$ de $X$ contenant $x$ et tout bon domaine analytique $V$ de $Y\times_X U$ contenant $y$, le faisceau cohérent $\mathscr F_V$ est $U$-plat en $y$ ; et il suffit en fait de le vérifier pour \emph{un} tel couple $(U,V)$ donné. Si $\mathscr F=\mathscr O_Y$ on dit aussi que $Y$ est $X$-plat en $y$ ou que $f$ est plat en $x$. On dit que $\mathscr F$ est $X$-plat s'il est plat en tout point de $Y$ (si $\mathscr F=\mathscr O_Y$ on dira aussi que $Y$ est $X$-plat, ou que $f$ est plat) 

Soit $U$ un domaine analytique de $X$ contenant $x$ et soit $V$ un domaine analytique de $Y\times_X U$ contenant $y$ (on ne suppose pas $U$ ou $V$ bons). Il est immédiat que $\mathscr F$ est $X$-plat en $y$
si et seulement si $\mathscr F_V$ est $U$-plat en $y$. En particulier $U\hookrightarrow X$ est plat
(appliquer ce qui précède au morphisme $\mathrm{Id}_X\colon X\to X$, qui est évidemment plat). 

\subsubsection{Platitude du morphisme
structural}
Le morphisme structural $X\to \mathscr M(k)$ est plat (\cite{ducros2018}, lemme 4.1.13 ; il y a quelque chose à démontrer !).

\subsubsection{Changement de base et descente}
Soit $X'$ un espace analytique
défini sur une extension complète
$L$ de $k$, soit $X'\to X$ un $k$-morphisme et soit $y'$ un point de $Y':=Y\times_x X'$ situé au-dessus de $y$, dont on note $x'$ l'image sur $X'$. Si $\mathscr F$ est $X$-plat en $y$ alors $\mathscr F_{Y'}$ est $X'$-plat en $y'$ : la définition permet en effet de se ramener au cas où tous les espaces en jeux sont bons, dans lequel cette propriété a été imposée. 
La réciproque est vraie si le morphisme $X'\to X_L$ induit par $X'\to X$ est plat
en $x'$ (combiner les propositions 4.5.5 et 4.5.6
de \cite{ducros2018}). 

\subsubsection{Platitude et suite exacte}
Soit $0\to \mathscr L\to \mathscr M\to \mathscr N\to 0$ une suite de faisceaux cohérents sur $X$. 
Si elle est exacte en $x$ et si $Y$ est $X$-plat en $y$ alors 
$0\to \mathscr L_Y\to \mathscr M_Y\to \mathscr N_Y\to 0$ est exacte en $y$
(\cite{ducros2018}, prop. 4.5.7 (1) pour une preuve directe).

\subsubsection{Platitude et dimension}
Si $\mathscr F$ est $X$-plat en $y$ on a
\[\dim_y \supp F=\dim_y (\supp F)_x+\dim_x X\]
(\cite{ducros2018}, Lemma 4.5.11). 

\subsubsection{Lieu de platitude}
L'ensemble des points de $Y$ en lesquels $\mathscr F$ est $X$-plat est un ouvert de Zariski de $Y$
(\cite{ducros2018}, thm. 10.3.2). 

\subsubsection{Image d'un espace compact par un morphisme plat}
Si $X$ et $Y$
sont compacts et $\Gamma$-stricts et si $\mathscr F$ est $X$-plat, l'image 
$f(\supp F)$ est un domaine analytique compact et $\Gamma$-strict de 
$X$ (\cite{ducros2018}, thm. 9.2.2 ; le cas où $\Gamma=\{1\}$ et où $\mathscr F=\mathscr O_Y$ est dû à Raynaud, \emph{cf.} 
\cite{frg2} Cor. 5.11).

\subsubsection{Platitude automatique au-dessus de certains points «génériques»}
Soit $Y\to X$ un morphisme d'espaces $k$-analytiques, soit $\mathscr F$ un faisceau 
cohérent sur $Y$ et soit $x$ un point de $X$. Supposons que $X$ est réduit en $x$ et que $d_k(x)=\dim_x X$. 
Le faisceau $\mathscr F$ est alors $X$-plat en chaque point de la fibre $Y_x$
(\cite{ducros2018}, thm. 10.3.7). 

\subsubsection{Platitude et composantes irréductibles}
Si $f$ est plat, $\adhz {f(Z)}X$ est une composante irréductible de $X$ pour toute composante
irréductible $Z$
de $Y$ (\cite{ducros2009}, lemme 5.7). 

\subsubsection{Principes GAGA
pour la platitude}
Soient $\mathscr Y$ et $\mathscr X$ deux schémas de type fini sur une même algèbre $k$-affinoïde, et soit $\mathscr G$ un faisceau cohérent sur $\mathscr Y$. Le faisceau $\mathscr G\an$ est alors $\mathscr X\an$-plat en un point $t$ de $\mathscr Y\an$ si et seulement si $\mathscr G$ est $\mathscr X$-plat en $t\al$ (\cite{ducros2018}, lemme 4.2.1 et prop. 4.2.4). 

Par ailleurs, soit $Y\to X$ un morphisme entre espaces $k$-affinoïdes, soit $\mathscr F$ un faisceau cohérent sur $Y$
et soit $y$ un point de $Y$. Si $\mathscr F$ est $X$-plat en $y$ alors $\mathscr F\al$ est $X\al$-plat en $y\al$ ; si 
$\mathscr F\al$ est $X\al$-plat en $y\al$ et si $y$ vit sur un sous-espace
analytique fermé de $Y$ qui est fini sur $X$ (c'est par exemple le cas si $y$ est un point rigide ou si $Y$ est fini sur $X$) 
alors $\mathscr F$ est $X$-plat en $y$ (\cite{ducros2018}, lemme 4.2.1
et thm. 8.3.7).

\subsection{Morphismes
quasi-lisses et quasi-étales}
Soit $f\colon Y\to X$ un morphisme
entre espaces $k$-analytiques, soit $y$ un point de $Y$ et soit $x$ son image sur $X$.

\subsubsection{}
Au chapitre 5 de \cite{ducros2018},
on définit la notion de \emph{quasi-lissité} de $f$ en $y$ par un critère jacobien mettant en jeu le faisceau $\Omega_{Y/X}$
des différentielles de Kähler (\cite{ducros2018}, déf. 5.2.4 ; voir \emph{ibid.}, 5.1 pour les rappels sur $\Omega_{Y/X}$) ;
on dit que $f$ est quasi-lisse s'il est quasi-lisse en tout point de $Y$. Un espace $k$-analytique est dit quasi-lisse (en un point donné
ou globalement) si son morphisme structural vers $\mathscr M(k)$ est quasi-lisse (en le point donné, ou globalement). 
On emploie l'expression «quasi-étale» 
pour «quasi-lisse de dimension relative nulle» (cette définition est compatible avec celle donnée par Berkovich dans \cite{berkovich1994}, \emph{cf.} \cite{ducros2018}, lemme 5.4.11). 

\subsubsection{Propriétés élémentaires}
La quasi-lissité est stable par composition et changement de base ; les immersions de domaines analytiques sont quasi-lisses ; plus généralement si $U$ est un domaine analytique de $X$ contenant $x$ et si $V$ est un domaine analytique de $Y\times_X U$ contenant $y$ alors $Y\to X$ est quasi-lisse en $y$ si et seulement si $V\to U$ est quasi-lisse en $y$ (\cite{ducros2018}, 
5.2.10--5.2.13).

\subsubsection{Quasi-lissité sur le corps de base}
Soit
$Z$ un espace $k$-analytique et soit $z\in Z$.
On dit que $Z$ est géométriquement régulier en $z$  si $Z_L$ est régulier en tout antécédent de $z$ pour toute extension complète $L$ de $k$ ; il suffit de le vérifier pour \emph{une} extension \emph{parfaite} $L$ de $k$ donnée et \emph{un} antécédent
donné de $z$ sur $Z_L$, \emph{cf.}
\cite{ducros2018}, 2.6.9

Les assertions suivantes sont équivalentes (\emph{cf.} \cite{ducros2018} 5.1.9, 5.1.10) :
\begin{enumerate}[i]
\item l'espace $Z$ est quasi-lisse en  $z$ ; 
\item $\Omega_{Z/k}$ est libre de rang $\dim_z Z$ en $z$ ; 
\item $Z$ est \emph{géométriquement régulier} en $z$. 
\end{enumerate}

\subsubsection{Quasi-lissité et platitude}
Le morphisme $f\colon Y\to X$ est quasi-lisse en $y$ si et seulement si $f$ est plat en $y$ et $Y_x$ est quasi-lisse en $y$. Pour une preuve de ces affirmations, voir \cite{ducros2018} 5.1.9, 5.1.10 et thm. 
5.3.4.

\subsubsection{Lieu de quasi-lissité}
Si $d$ est un entier, l'ensemble des points $z$ de $Y$
tels que $f$  soit quasi-lisse en $z$ et tels que $\dim_z f=d$
est un ouvert de Zariski de $Y$ (\cite{ducros2018}, thm. 10.7.2). 

\subsubsection{Quasi-lissité et propriétés
d'algèbre commutative}
Les propriétés évoquées en \ref{ss-valid-prop} «descendent par morphisme
plat et sont préservées par morphisme quasi-lisse». Plus précisément, soit
$\mathscr F\to \mathscr G\to \mathscr H$
un diagramme dans la catégorie des faisceaux cohérents sur $X$
et soit $\mathsf P$ l'une des propriétés
considérées en \ref{ss-valid-prop}.
Si $f$ est plat en $y$ et si $Y$ (ou $\mathscr F_Y$, ou
$\mathscr F_Y\to \mathscr G_Y$, ou $\mathscr F_Y\to \mathscr G_Y\to \mathscr H_Y$) satisfait $\mathsf P$ en $y$ alors
$X$ (ou $\mathscr F$, ou $\mathscr F\to \mathscr G$, ou $\mathscr F\to \mathscr G\to \mathscr H$) satisfait $\mathsf P$ en $x$. La réciproque vaut si $f$ est quasi-lisse en $y$ (\cite{ducros2018}, lemmes 4.5.1 et 4.5.2, prop. 
5.5.4 et 5.5.5).

\subsubsection{Quasi-lissité et principes GAGA}
Sii $\mathscr Y\to \mathscr X$ est un morphisme entre 
schémas de type fini sur une même algèbre affinoïde et si $y$
est un point de $\mathscr Y\an$ alors $\mathscr Y\an \to \mathscr X\an$ est quasi-lisse (resp. quasi-étale)
en $y$ si et seulement si $\mathscr Y\to \mathscr X$ est lisse (resp. étale) en $y\al$
(\cite{ducros2018}, prop. 2.2.7 et cor. 5.3.6). 

\begin{rema}
Le préfixe «quasi» 
dans les expressions «quasi-lisse» et «quasi-étale» indique que les morphismes concernés
peuvent avoir du bord, ce qui n'est pas le cas des morphismes lisses et étales définis
par Berkovich au chapitre 3 de \cite{berkovich1993}. 
Plus précisément, 
soit $Y\to X$ un morphisme d'espaces $k$-analytiques et soit $y\in Y$.
Le morphisme $Y\to X$ est étale en $y$
si et seulement si
$f$ est quasi-étale en $y$ et $y\notin \partial (Y/X)$ ; si $Y\to X$ est lisse en $y$, 
il est quasi-lisse en $y$ et $y\notin \partial (Y/X)$ ; et la réciproque de cette dernière assertion
vaut lorsque $X$ et $Y$ sont bons, mais probablement pas en général (\cite{ducros2018}, cor. 5.4.8 et rem. 5.4.9 ; 
pour la définition du bord $\partial (Y/X)$, voir \cite{berkovich1993},1.5.4). 
\end{rema}

\subsection{À propos du théorème de Gerritzen-Grauert}\label{ss-ger-grau}
Soit $X$ un espace $k$-affinoïde. Rappelons qu'un domaine affinoïde $V$ de $X$ est dit \emph{rationnel}
s'il peut être défini par une conjonction d'inégalités de la forme
\[\abs {f_1}\leq \lambda_1 \abs g, \ldots, \abs{f_n}\leq \lambda_n \abs g\]
où les $f_i$ et $g$ sont des fonctions analytiques sur $X$ sans zéro commun et où les $\lambda_i$ sont des réels strictement positifs. 
Si $X$ est $\Gamma$-strict et si les $\lambda_i$ peuvent être choisis dans $\Gamma$, nous dirons que $V$ est $\Gamma$-rationnel. 

On dispose d'une version $\Gamma$-stricte du théorème de Gerritzen-Grauert : si $X$ est $\Gamma$-strict, tout domaine affinoïde $\Gamma$-strict de $X$ est réunion finie de domaines $\Gamma$-rationnels. Lorsque $\Gamma=\{1\}$ c'est le théorème de Gerritzen-Grauert classique (\emph{cf.} \cite{bosch-g-r1984}, \S 7.3.5, Thm. 1., Cor. 3) ; lorsque 
$\Gamma=\R_{>0}$ c'est établi dans \cite{ducros2003}, lemme 2.4
par réduction au cas $\Gamma=\{1\}$, mais la preuve de ce lemme s'adapte en fait sans
la moindre difficulté au cas
où $\Gamma$ est quelconque.

\begin{lemm}\label{lem-gerr-grau}
Soit $X$ un espace $k$-analytique $\Gamma$-strict et compact et soit $Y$ un sous-espace
analytique fermé de $X$. Soit $V$ un domaine analytique compact et $\Gamma$-strict de $Y$. 
Il existe un domaine analytique compact et $\Gamma$-strict $W$ de $X$ tel que $V=Y\cap W$.
\end{lemm}

\begin{proof}
Soit $(X_i)$ un recouvrement fini de $X$ par des domaines affinoïdes $\Gamma$-stricts. Si l'on construit pour tout $i$ un domaine
analytique compact et $\Gamma$-strict $W_i$ de $X_i$ tel que $W_i\cap Y=V\cap X_i$ il suffira de poser $W=\bigcup_i W_i$ pour conclure. 
On peut donc supposer $X$ affinoïde. Dans ce cas $Y$ est également affinoïde, et la version $\Gamma$-stricte du théorème de Gerritzen-Grauert assure alors que $V$ est une union finie de domaines $\Gamma$-rationnels ; il suffit de démontrer le lemme pour chacun d'eux, ce qui permet de se ramener au cas où $V$ est un domaine $\Gamma$-rationnel
de $Y$. Choisissons un système d'inégalités 
\[\abs {f_1}\leq \lambda \abs g, \ldots, \abs{f_n}\leq \lambda_n \abs g\] décrivant $V$, où les $f_i$ et $g$ sont des fonctions
analytiques sur $Y$ sans zéro commun et où les $\lambda_i$ appartiennent à $\Gamma$. Comme les $f_i$ et $g$ sont sans zéro commun, $g$ ne s'annule pas sur $V$. Il y est alors minoré par un élément $\mu\in \abs{k\gpm}\cdot \Gamma$ ; quitte à multiplier $g$ et les $f_i$ par un scalaire non nul convenable, on peut supposer que $\mu \in \Gamma$, et rajouter l'inégalité $\abs 1\leq \mu\inv \abs g$ au système décrivant $V$. 

Relevons chacune des $f_i$ en une fonction analytique $\phi_i$ sur $X$, et relevons $g$ en une fonction analytique $\psi$ sur $X$. Le domaine $\Gamma$-rationnel $W$ de $X$ défini par les inégalités
\[\abs {\phi_1}\leq \lambda \abs \psi, \ldots, \abs{\phi_n}\leq  \lambda_n \abs \psi, \abs 1 \leq \mu\inv \abs \psi\]
répond alors à nos exigences. 
\end{proof}

\begin{rema}
Lorsque $\Gamma=\R_{>0}$ le lemme \ref{lem-gerr-grau} est
une conséquence de la version du théorème de Gerritzen-Grauert prouvée par Temkin, voir
la proposition 3.5 de
\cite{temkin2005}. Et la preuve de Temkin s'adapte en fait
au cas d'un groupe
$\Gamma$ quelconque, à condition de remplacer sa théorie de la réduction $\R_{>0}$-graduée par sa variante 
$\Gamma$-graduée, \emph{cf.} le chapitre 3
de \cite{ducros2018} et plus particulièrement 
les sections 3.4 et 3.5 ; on obtient ainsi une autre démonstration du lemme \ref{lem-gerr-grau}.
\end{rema}

\subsection{Idéaux de Fitting}\label{ss-fitting}
Nous nous servirons de manière cruciale dans la démonstration
de notre théorème principal de la notion d'idéal de Fitting ; rappelons brièvement
ce en quoi elle consiste.  
Soit $A$ un anneau commutatif et soit $M$ un $A$-module de type fini. Soit $r$ un entier positif ou nul. Donnons-nous une présentation 
\[\begin{tikzcd}
A^{(I)}\ar[r,"u"]&A^n\ar[r]&M\ar[r]&0
\end{tikzcd}
\]
de $M$. 
Le \emph{$r$-ième idéal de Fitting de $M$}
est l'idéal de $A$ engendré par les mineurs de taille $n-r$ 
de la matrice de $u$
(dans les bases canoniques de $A^{(I)}$ et $A^n$) si $r\leq n$, et l'idéal $A$ si $r>n$. La terminologie est justifiée car on montre que
cet idéal ne dépend
pas de la présentation choisie de $M$ (voir par exemple
\cite[\href{https://stacks.math.columbia.edu/tag/07Z8}{Tag 07Z8}]{stacks-project}).
Soit $X$ un schéma (resp. un espace $k$-analytique)
et soit $\mathscr F$ un faisceau quasi-cohérent de type fini 
(resp. un faisceau cohérent) sur $X$.
Par recollement à partir du cas affine (resp. affinoïde) on peut définir le $r$-ième idéal de Fitting $\mathscr I$ de $\mathscr F$ sur $X$. C'est un faisceau quasi-cohérent (resp. cohérent) d'idéaux, dont le lieu des des zéros est exactement l'ensemble des points de $X$ en lesquels le rang ponctuel de $\mathscr F$ est $>r$.

\section{L'assassin analytique}\label{s-assassin}

\begin{enonce}[remark]{Convention}
Si $M$ est un module non nul sur un anneau local noethérien $A$,
la dimension $\dim M$ de $M$ sera par définition la dimension de Krull
de son support sur $\spec A$.
\end{enonce}

\subsection{Platitude et profondeur}\label{ss-plat-prof}
Soit $A\to B$ un morphisme local entre anneaux locaux noethériens et soit $F$ le corps résiduel de $A$. 
Soit $M$ un $A$-module non nul de type fini et soit $N$ un $B$-module non nul de type fini et
plat sur $A$. 
Considérons $N\otimes_A M$, $N\otimes_A F$ et $M$ comme
des modules sur $B, B\otimes_AF$ et $A$ respectivement.
On a alors les égalités : 

\begin{eqnarray}
\dim(N\otimes_A M)&=&\dim(N\otimes_A F)+\dim M\label{eq-dimplat}\\
\mathrm{prof}(N\otimes_A M)&=&\mathrm{prof}(N\otimes_A F)+\mathrm{prof}\;M
\label{eq-profplat}
\end{eqnarray}
(voir \cite{ega42}, Cor. 6.1.2 et Prop. 6.3.1).
Indiquons deux conséquences immédiates de ces égalités qui nous
seront utiles. 

\subsubsection{}\label{sss-transfert-prof0gen}
On déduit de l'égalité (\ref{eq-profplat})
que $N\otimes_A M$ est de profondeur nulle si et seulement si
$N\otimes_AF$ et $M$ sont de profondeur nulle. 

\subsubsection{}\label{sss-transfert-prof0part}
Supposons que $N=B$ (et donc que $B$ est plat sur $A$), et que
$B\otimes_A F$ est de Cohen-Macaulay, c'est-à-dire que
$\mathrm{prof}(B\otimes_AF)=\dim (B\otimes_A F)$. 
Par ce qui précède, $\mathrm{prof}(B\otimes_A M)=0$ 
si et seulement si $\mathrm{prof}(B\otimes_A F)=\mathrm{prof}\;M=0$.  
Au vu de l'égalité (\ref{eq-dimplat}) on a par ailleurs
$\mathrm{prof}(B\otimes_A F)=\dim(B\otimes_A F)=\dim (B)-\dim (A)$,
de sorte que finalement 
\[\mathrm{prof}(B\otimes_A M)=0\iff(\mathrm{prof}\;M=0\;\;\text{et}\;\;\dim B=\dim A).\]

\subsection{Dimension centrale}\label{ss-dimcent}
Soit $X$ un espace $k$-analytique et soit $x$ un point de $X$.

\subsubsection{}\label{sss-defdimcent}
Rappelons (\cite{ducros2018}, déf.  3.2.2)
que la \emph{dimension centrale}
du germe $(X,x)$ est le minimum des dimensions de
$\adhp x V$ où $V$ parcourt l'ensemble des \emph{voisinages}
analytiques de $x$ dans $X$ ; nous la noterons $\dimc Xx$. 

Nous allons avoir besoin dans ce qui suit d'une « G-version » de cette notion : 
nous définirons $\dimcg Xx$ comme le  minimum des dimensions de
$\adhp x V$ où $V$ parcourt l'ensemble des domaines analytiques de $X$ contenant $x$.

L'application $V\mapsto \dim \adhp x V$ 
est une fonction croissante du domaine analytique $V$ contenant $x$. Dans les définitions
de $\dimc Xx$ (resp. $\dimcg Xx$), on peut donc se contenter de faire parcourir à $V$ une base de voisinages
analytiques de $x$ (resp. une famille cofinale de domaines analytiques de $X$ contenant $x$) : dans le premier cas, on pourra
ainsi se
limiter aux ouverts, ou aux voisinages analytiques compacts, ou aux voisinages affinoïdes si $X$ est bon ; dans le second, on pourra se limiter
aux domaines affinoïdes.

\subsubsection{}\label{dimcent-dx}
On a les inégalités 
\[d_k(x)\leq \dimcg Xx\leq \dimc Xx\leq \dim \adhp x X.\] Si $d_k(x)=\dim \adhp x X$ on a donc
\[d_k(x)=\dimcg Xx= \dimc Xx= \dim \adhp x X.\]

\subsubsection{}\label{dimcent-fermzar}
Si $V$ est un voisinage analytique de $x$ on a $\dimc Vx=\dimc Xx$ ; si $V$ est un domaine analytique
de $X$
contenant $x$ on a
$\dimcg Vx=\dimcg Xx$ et $\dimc Vx\leq \dimc Xx$, avec égalité si $\dimcg Xx=\dimc Xx$
car
on a alors $\dimc Xx=\dimcg Xx=\dimcg Vx\leq \dimc Vx$. 

Si $Y$ est un fermé de Zariski de $X$ contenant $x$
on a $\dimc Yx=\dimc Xx$ et $\dimcg Yx=\dimcg Xx$.
La première égalité est évidente ; la seconde vient du fait
que si $W$ est un domaine analytique de $Y$ contenant $x$,
 il existe en vertu
du
lemme \ref{lem-gerr-grau}
(à appliquer ici avec $\Gamma=\R_{>0}$) 
un domaine analytique $V$ de $X$ contenant $x$ tel que $V\cap Y\subset W$.

\subsubsection{}\label{dimcent-dimkrull} Si $X$ est affinoïde, on a 
$\dim \mathscr O_{X,x}=\dim \mathscr O_{X\al, x\al}$ si et seulement si 
$\dimc Xx=\dim \adhp x X$ (\cite{ducros2018}, lemme 3.2.8) ; et si $X$ est bon, 
on a l'égalité $\dim \mathscr O_{X,x}+\dimc Xx=\dim_x X$ (\emph{ibid.}, cor. 3.2.9).

\subsection{}
Soit $X$ un espace
$k$-affinoïde, soit $V$ un domaine 
affinoïde de $X$ et soit $\mathscr F$ un faisceau cohérent sur $X$. 

\subsubsection{}\label{para-diag-prof}
Soit $x$ un point de $V$ appartenant
au support de $\mathscr F$. Dans le diagramme commutatif
\[\xymatrix{
{\spec \mathscr O_{V,x}}\ar[r]\ar[d]\ar[rd]&{\spec \mathscr O_{X,x}}\ar[d]\\
{\spec \mathscr O_{V\al,x_V\al}}\ar[r]&{\spec \mathscr O_{X\al,x\al}}
}\]
toutes les flèches sont régulières (\cite{ducros2009}, Th. 3.3), et en particulier plates et à fibres de Cohen-Macaulay. 

\subsubsection{}\label{sss-equiv1-prof0}
En vertu de \ref{sss-transfert-prof0part}, il s'ensuit que 
$\mathscr F_x$ est de profondeur nulle si et seulement si 
$\mathscr F\al_{x\al}$ est de profondeur nulle et $\dim \mathscr F\al_{x\al}=\dim \mathscr F_x$.
D'après \ref{dimcent-fermzar}
et \ref{dimcent-dimkrull},
cela revient à demander
que 
$\mathscr F\al_{x\al}$ soit de profondeur nulle
et que $\dimc Xx=\dim \adhp x X$. 

\subsubsection{}\label{sss-equiv2-prof0}
De même, $\mathscr F_{V,x}$ est de profondeur nulle si et seulement 
si $\mathscr F_x$ est de profondeur nulle et
$\dim \mathscr F_{V,x}=\dim \mathscr F_x$. D'après
\ref{dimcent-fermzar}
et \ref{dimcent-dimkrull},
cela revient à demander que 
$\mathscr F_x$ soit de profondeur nulle et que 
$\dimc Vx=\dimc Xx=\dim \adhp xX$ (la dernière égalité provient de \ref{sss-equiv1-prof0}).

\subsection{}\label{defi-ass-schema}
Soit $X$ un schéma noethérien et soit 
$\mathscr F$ un faisceau cohérent sur $X$.
Nous suivrons la terminologie de
Raynaud et Gruson en appelant \emph{assassin}
de $\mathscr F$
l'ensemble des points $x$ de $\supp F$
tels que le $\mathscr O_{X,x}$-module
$\mathscr F_x$ soit de profondeur nulle, ce qui revient à demander
que l'idéal maximal $\mathfrak m_x$
de $\mathscr O_{X,x}$ soit l'annulateur d'un élément de $\mathscr F_x$, ou encore que tout élément de $\mathfrak  m_x$ soit un diviseur de zéro du module $\mathscr F_x$ ; 
l'assassin de $\mathscr F$ sera noté
$\mathrm{Ass}(\mathscr F)$ ; c'est un ensemble fini
qui contient l'ensemble des points maximaux de $\supp F$. 
L'ensemble $\mathrm{Ass}(\mathscr O_X)$ sera
également appelé l'assassin de $X$ et noté $\mathrm{Ass}(X)$.

Une \emph{composante assassine}
de $\mathscr F$ (ou de $X$ si $\mathscr F=\mathscr O_X$)
est un fermé irréductible de $X$ dont le point générique  appartient à 
$\mathrm{Ass}(\mathscr F)$. Toute composante irréductible de $\supp F$ est une composante assassine de $\mathscr F$. Les autres composantes assassines de $\mathscr F$ (ou de $X$ si $\mathscr F=\mathscr O_X$)
sont traditionnellement appelées ses composantes \emph{immergées}.

\begin{rema}
Notre définition de l'assassin est
spécifique au cas noethérien, le seul dont nous aurons besoin ici. Sur un schéma quelconque, la bonne définition est celle de Raynaud et Gruson (\cite{raynaud-g1971}, définition 3.2.1). 
\end{rema}

Le but de ce qui suit est de développer une théorie des composantes assassines et immergées en géométrie analytique. 

\begin{defi}\label{defi-ass}
Soit $X$ un espace $k$-analytique et soit $\mathscr F$ un faisceau cohérent sur $X$. On a appelle \emph{assassin}
de $\mathscr F$, et l'on note $\mathrm{Ass}(\mathscr F)$, l'ensemble des points $x$ de $\supp F$ tel que 
que le $\mathscr O_{V,x}$-module $\mathscr F_{V,x}$ soit de profondeur nulle pour tout bon domaine analytique $V$
de $X$ contenant $x$ (il suffit que ce soit le cas pour tout domaine affinoïde contenant $x$). 
L'ensemble $\mathrm{Ass}(\mathscr O_X)$ sera souvent noté plus simplement $\mathrm{Ass}(X)$ et appelé l'assassin de $X$.

Un fermé de Zariski $Y$ de $X$ sera appelé une \emph{composante assassine}
de $\mathscr F$ (ou de $X$ si $\mathscr F=\mathscr O_X$)
s'il est de la forme $\adhp xX$ pour un certain $x\in \mathrm{Ass}(X)$ (remarquons qu'une telle
composante est toujours contenue dans $\supp F$).
\end{defi}

\subsection{}\label{assx-dom}
Soit $X$ un espace
$k$-analytique et soit $\mathscr F$ un faisceau cohérent sur $X$. Soit
$V$ un domaine analytique de $X$. Il découle immédiatement de la définition que
l'intersection $\mathrm{Ass}(\mathscr F)\cap V$
est contenue dans $\mathrm{Ass}(\mathscr F_V)$, mais on a en fait égalité. 
En effet, soit $x\in \mathrm{Ass}(\mathscr F_V)$ et soit $U$ un domaine affinoïde de $X$ contenant $x$.
Soit $W$ un domaine affinoïde de $U\cap V$ contenant $x$. Comme 
$x$ appartient à $\mathrm{Ass}(\mathscr F_V)$, le module $\mathscr F_{W,x}$ est de profondeur nulle, 
et il découle alors de \ref{sss-equiv2-prof0}
que $\mathscr F_{U,x}$ est aussi de profondeur nulle. 
Ainsi, $x\in \mathrm {Ass}(\mathscr F)$.

\begin{prop}\label{pro-compass-aff}
Soit $X$ un espace $k$-affinoïde
et soit $\mathscr F$ un faisceau cohérent sur $X$. 

\begin{enumerate}[1]

\item Soit $x\in \supp F$. Les assertions suivantes sont équivalentes. 

\begin{enumerate}[j]
\item Le point $x$ appartient à $\mathrm{Ass}(\mathscr F)$. 
\item La profondeur de $\mathscr F_x$ est nulle, et $\dimcg Xx=\dim \adhp xX$. 
\item La profondeur de $\mathscr F_{x\al}\al$ est nulle, et $\dimcg Xx=\dim \adhp xX$. 
\end{enumerate}

\item Soit $Z$ un fermé de Zariski de $X$. 
Le fermé $Z$ est une composante assassine de $\mathscr F$ si et seulement si $Z\al$ est une
composante assassine de $\mathscr F\al$. 

\item Soit $Z$ une composante assassine de $\mathscr F$ et soit $z\in Z$. Les assertions suivantes sont équivalentes : 
\begin{enumerate}[B]
\item $z\in \mathrm{Ass}(\mathscr F)$ et $\adhp zX=Z$ ; 
\item $\dimcg Xz=\dim Z$. 
\end{enumerate}

\item Soit $V$ un domaine affinoïde de $X$. 
\begin{enumerate}[b]

\item Soit $Z$ une composante assassine de $\mathscr F$. Toute composante
irréductible de $Z\cap V$ est une composante assassine de $\mathscr F_V$. 
\item Soit $T$ une composante assassine de $\mathscr F_V$. L'adhérence $\adhz TX$ est une composante
assassine de $\mathscr F$ de même dimension que $T$, et $T$ est une composante irréductible
de $\adhz TX\cap V$. 
\end{enumerate}
\end{enumerate}
\end{prop}

\begin{proof}
Prouvons (1). Si (i) est vraie alors par définition, $\mathscr F_{V,x}$ est de profondeur nulle
pour tout domaine affinoïde $V$ de $X$ contenant $x$. On en déduit que $\mathscr F_x$ est de profondeur nulle 
et, grâce à 
\ref{sss-equiv2-prof0}, que 
\[\dimc Vx=\dimc Xx=\dim \adhp xX\] pour tout domaine affinoïde $V$ de $X$
contenant $x$, ce qui entraîne que 
que $\dimcg Xx=\dim \adhp xX$ ; ainsi (ii) est vraie. 
Si (ii) est vraie, alors (iii) est vraie d'après \ref{sss-equiv1-prof0}. Enfin, supposons que (iii)
soit vraie. Soit $V$ un domaine affinoïde de $X$ contenant $x$.
L'hypothèse (iii) entraîne 
que $\dimc Vx=\dimc Xx=\dim \adhp xX$. 
La seconde égalité et le fait que $\mathscr F\al_{x\al}$ soit de profondeur nulle
entraînent que $\mathscr F_x$ est de profondeur nulle (\ref{sss-equiv1-prof0}) ; et ceci implique au vu
de la première égalité que $\mathscr F_{V,x}$ est de profondeur nulle (\ref{sss-equiv2-prof0}), 
d'où (i). 

Prouvons (2). Supposons que $Z$ soit une composante assassine de $\mathscr F$. Il existe alors
un point $x\in \mathrm{Ass}(\mathscr F)$ tel que $Z=\adhp xX$. Le point $x\al$ est le point
générique de $Z\al$. Comme $x\in \mathrm{Ass}(\mathscr F)$, il résulte de (1)
que $\mathscr F\al_{x\al}$ est de profondeur nulle ; par conséquent, $Z\al$ est une composante
assassine de $\mathscr F\al$.
Réciproquement, supposons que $Z\al$ soit une composante assassine de $\mathscr F\al$. 
Dans ce cas, $Z$ est irréductible
et contenu dans le support de $\mathscr F$ ; soit $n$ sa dimension. Choisissons un point $x$ sur $Z$
tel que $d_k(x)=n$. On a alors $\adhp xX=Z$, et $x\al$ est le point générique de $Z\al$.
Comme $Z\al$ est une composante assassine de $\mathscr F\al$, la profondeur
de $\mathscr F\al_{x\al}$ est nulle ; et l'égalité $d_k(x)=\dim \adhp xX$
assure par ailleurs que $\dimcg Xx=\dim \adhp xX$ (\ref{dimcent-dx}). On déduit alors de (1)
que $x\in \mathrm{Ass}(\mathscr F)$ ; par conséquent, 
$Z=\adhp xX$ est une composante assassine de $\mathscr F$.

Montrons (3).
L'implication (A)$\Rightarrow$(B)
est une conséquence immédiate de (1). 
Supposons maintenant que (B) est vérifié. On a
alors clairement $\adhp zX=Z$, et $z\al$ est donc le point
générique de $Z\al$. Puisque $Z$ est une composante assassine de $\mathscr F$, 
le fermé $Z\al$ est une composante assassine de $\mathscr F\al$ d'après (2), et $\mathscr F\al_{z\al}$ est donc
de profondeur nulle. Joint à l'égalité $\dimcg Xz=\dim Z=\dim \adhp z X$, ceci entraîne en vertu de (1)
que $z\in \mathrm{Ass}(\mathscr F)$.

Montrons (4a). Soit $T$ une composante irréductible de $Z\cap V$. 
Soit $t$ un point de $T$ tel que $d_k(t)=\dim T=\dim Z$. On a alors 
$\adhp t V=T$ et $\dimcg Xt=\dim Z$. En vertu de (3), cette dernière
égalité implique que $t$ appartient à $\mathrm{Ass}(\mathscr F)$ et donc
à $\mathrm{Ass}(\mathscr F_V)$. Par conséquent, $T=\adhp tV$ est une composante
assassine de $\mathscr F_V$. 

Montrons (4b). Choisissons un point $t$ de $\mathrm{Ass}(\mathscr F_V)$ tel que
$\adhp tV=T$ ; posons $Z=\adhz TX=\adhp tX$. Le point $t$ appartenant à $\mathrm{Ass}(\mathscr F_V)$ appartient également à $\mathrm{Ass}(\mathscr F)$ 
(\ref{assx-dom}),
et $Z$ est donc une composante assassine de $X$. 
En appliquant (3) à $T$ et $Z$ on voit que 
 $\dim T=\dimcg Vt=\dimcg Xt=\dim Z$.
 Pour des raisons de dimension, $T$ est nécessairement une composante irréductible
 de $\adhz TX\cap V$. 
 \end{proof}

La proposition suivante étend au cas d'un espace $X$ quelconque une partie des résultats de la proposition précédente. 
On peut la voir comme une généralisation aux composantes
assassines des résultats
sur les composantes irréductibles rappelés au \ref{compirr-domaines}.

\begin{prop}\label{prop-equiv-ass}
Soit $X$ un espace $k$-analytique et soit $\mathscr F$ un faisceau cohérent sur $X$. 

\begin{enumerate}[1]

\item Soit $Z$ une composante assassine de $\mathscr F$ et soit $x\in Z$. Les assertions
suivantes sont équivalentes : 
\begin{enumerate}[j]
\item $x$ appartient à $\mathrm{Ass}(\mathscr F)$ et $\adhp xX=Z$ ; 
\item $\dimcg Xx=\dim Z$. 
\end{enumerate}

\item Soit $V$ un domaine analytique de $X$. 
\begin{enumerate}[b]
\item Soit $Z$ une composante assassine de $\mathscr F$. Toute composante irréductible de
$Z\cap V$ est une composante assassine de $\mathscr F_V$.
\item Soit $T$ une composante assassine
de $\mathscr F_V$. L'adhérence $\adhz TX$ est une 
composante assassine de $\mathscr F$
de même dimension que $T$, et $T$ est une composante irréductible
de $\adhz TX \cap V$. 
\end{enumerate}

\item Soit $V$ un domaine analytique compact de
$X$. Le domaine $V$ ne rencontre qu'un nombre fini de composantes
assassines de $\mathscr F$. 
\end{enumerate}

\end{prop}

\begin{proof}
Montrons d'abord (1). Supposons que (i) est vérifiée, et soit $V$ un domaine
affinoïde de $X$ contenant $x$ ; posons $T=\adhp xV$. Puisque $x$ appartient
à $\mathrm{Ass}(\mathscr F)$, il appartient à $\mathrm{Ass}(\mathscr F_V)$ ; par conséquent 
$T$ est une composante assassine de $\mathscr F_V$ ; soit $n$ sa dimension.
Pour tout domaine affinoïde $U$ de $X$, notons $U_n$ la réunion des composantes assassines
de $\mathscr F_U$ de dimension $n$. Si $U$ et $W$ sont deux domaines affinoïdes de $X$
tels que $W\subset U$, il découle de l'assertion (4) de la proposition \ref{pro-compass-aff}
que $U_n\cap W=W_n$ ; si l'on pose $Y=\bigcup_U U_n$, où $U$ parcourt l'ensemble des domaines
affinoïdes de $U$, on a donc $Y\cap U=U_n$ pour tout tel $U$. Par conséquent, $Y$ est un fermé de Zariski de $X$, 
purement de dimension $n$ ; pour tout domaine affinoïde $U$ de $X$, les composantes irréductibles de $Y\cap U$
sont exactement les composantes assassines de $\mathscr F_U$
de dimension $n$.
Puisque $Y$ contient $x$, il contient $Z$ qui est donc de dimension au plus $n$. Mais comme
$Z$ contient $T$ qui est de dimension $n$, on a $\dim Z=n$
(et $Z$ est donc une composante irréductible de $Y$).
Et comme $x$ appartient 
à $\mathrm{Ass}(\mathscr F_V)$, l'assertion (3) de la proposition \ref{pro-compass-aff}
assure que $\dimcg Vx=\dim T=n$ ; il vient
\[\dimcg Xx=\dimcg Vx=n=\dim Z,\]
d'où (ii). 

Supposons réciproquement que (ii) est vérifiée, et notons
$n$ la dimension de $Z$. Puisque $Z$ est égale par
définition à $\adhp zX$ pour un certain $z\in \mathrm{Ass}(\mathscr F)$, le raisonnement 
suivi ci-dessus (avec $z$ au lieu de $x$) montre que pour tout domaine affinoïde $U$ de $X$, l'intersection
$U\cap Z$ est une union finie de composantes assassines de $\mathscr F_U$ de dimension $n$. 
Choisissons un domaine affinoïde $U$ de $X$ qui contient $x$. Par ce qui précède, $x$ est situé
sur une composante assassine $T$ de $\mathscr F_U$ qui est de dimension $n$. 
On a $\dimcg Ux=\dimcg Xx=n=\dim T$ ; d'après l'assertion (3) de la proposition \ref{pro-compass-aff}, 
le point $x$ appartient à $\mathrm{Ass}(\mathscr F_U)$, et partant à $\mathrm{Ass}(\mathscr F)$
(\ref{assx-dom}). Par ailleurs l'égalité $\dimcg Xx=\dim Z$ entraîne que $\adhp xX=Z$, 
d'où (i). 

L'assertion (2) se démontre alors \emph{mutatis mutandis}
comme l'assertion (4) de la proposition \ref{pro-compass-aff}
(qui en est un cas particulier, mais nous nous en sommes servi ci-dessus pour démontrer (1) ; c'est 
la raison pour laquelle nous l'avions établie
au préalable). 

Montrons enfin (3). On peut supposer que $V$ est affinoïde. Soit $Z$ une composante
assassine de $\mathscr F$ rencontrant $V$ et soit $T$ une composante irréductible de $Z\cap V$. 
On a alors $Z=\adhz TX$, et l'assertion (2) déjà établie
assure par ailleurs que $T$ est une composante
assassine de $\mathscr F_V$. On conclut en remarquant que $\mathscr F_V$ n'a qu'un 
nombre fini de composantes assassines, en vertu de l'assertion (2)
de la proposition \ref{pro-compass-aff}
et du fait que l'assassin d'un faisceau cohérent sur un schéma noethérien est fini. 
\end{proof}

\begin{rema}\label{compass-autre-descrip}
Soit $X$ un espace
$k$-analytique, soit $\mathscr F$ un faisceau cohérent sur $X$.
Si $V$ est un domaine analytique de $X$ et si $T$ est une composante
assassine de $\mathscr F_V$, la proposition \ref{prop-equiv-ass}
assure que $\adhz TX$ est une composante assassine 
de $\mathscr F$. 

Réciproquement, soit $Z$ une composante assassine 
de $\mathscr F$. Elle est non vide
par définition, donc il existe un domaine affinoïde $V$ de $X$ rencontrant $Z$. 
Choisissons une composante irréductible $T$ de $Z\cap V$. On a alors
$Z=\adhz TX$, et la proposition \ref{prop-equiv-ass}
ci-dessus assure par ailleurs que $T$ est une composante
assassine de $\mathscr F_V$.

Ainsi, une partie $Z$ de $X$ est une composante assassine
de $\mathscr F$ si et seulement si elle est de la forme
$\adhz TX$ où $T$ est une composante assassine de $\mathscr F_V$
pour un certain domaine affinoïde $V$ de $X$. (Lorsque $X$ est $\Gamma$-strict, le domaine affinoïde $V$ peut être choisi $\Gamma$-strict, puisqu'il suffit par ce qu'il précède qu'il rencontre $Z$.)

Et rappelons que si $V$ est un domaine affinoïde de $X$, un fermé de Zariski
$T$ de $V$ est une composante assassine de $\mathscr F_V$ si et seulement si
$T\al$ est une composante assassine de $\mathscr F\al_V$ (prop. \ref{pro-compass-aff}). 

Ceci permettra assez souvent de ramener des questions portant sur les composantes 
assassines en géométrie analytique à des problèmes analogues en géométrie algébrique. Nous allons
en voir ci-dessous plusieurs exemples. 

\end{rema}

\begin{exem}\label{ex-irr-ass}
Soit $X$ un espace
$k$-analytique, soit $\mathscr F$ un faisceau cohérent sur $X$ et soit  $Z$
une composante irréductible du support de $\mathscr F$. Choisissons un domaine affinoïde
$V$
de $X$ qui rencontre $Z$, et soit $T$ une composante irréductible de $Z\cap V$. 
Par construction, $T\al$ est une composante irréductible de $\mathrm{Supp}(\mathscr F_V\al)$, 
et c'est donc une composante assassine de $\mathscr F_V\al$ ; par conséquent, $T$ est une composante
assassine de $\mathscr F_V$ (prop. \ref{pro-compass-aff}) et $Z=\adhz TX$ est une composante
assassine de $\mathscr F$ (prop. \ref{prop-equiv-ass}).

On appelle \emph{composante immergée}
de $\mathscr F$ (ou de $X$ si $\mathscr F=\mathscr O_X$) toute composante assassine de $\mathscr F$ qui n'est pas
une composante irréductible de son support. 
\end{exem}

\begin{lemm}\label{fonction-annulation}
Soit $X$ un espace 
$k$-analytique, 
soit $U$ un domaine analytique de $X$ rencontrant toutes les composantes
assassines de $X$ et soit $Z$ un sous-espace analytique fermé de $X$ tel que
$Z\cap U=U$ ; on a alors $Z=X$. \end{lemm}
\begin{proof}

Soit $\mathscr I$ le faisceau cohérent d'idéaux décrivant $Z$ ; nous allons montrer que $\mathscr I=0$.
Soit 
$\Omega$ le complémentaire de $\mathrm{Supp}(\mathscr I)$ ; par hypothèse, $\Omega$
contient $U$.

Soit $V$ un domaine affinoïde de $X$ et soit $T$ une composante
assassine de $V$. L'adhérence $\adhz TX$ est une composante
assassine de $X$ (prop. \ref{prop-equiv-ass}); elle rencontre donc $U$ par hypothèse, et \emph{a fortiori}
$\Omega$. 
L'intersection $\Omega\cap \adhz TX$ est alors un ouvert Zariski-dense de $\adhz TX$, et 
$\Omega \cap V\cap \adhz TX$ est de ce fait Zariski-dense dans $V\cap \adhz TX$ ; en particulier,
$\Omega$ rencontre $T$ (qui est une composante irréductible de $V\cap \adhz TX$). 

Compte-tenu de la caractérisation algébrique des composantes assassines
de $V$ (prop. \ref{pro-compass-aff}), il résulte de ce qui précède
que $(\Omega\cap V)\al$ rencontre toutes les composantes assassines de $V\al$ ;
par conséquent, toute section globale de $\mathscr I_V\al$ est nulle sur un ouvert schématiquement
dense de $V\al$, et est dès lors nulle. Il vient $\mathscr I_V=0$ ; comme ceci vaut quel que soit $V$, 
on a $\mathscr I=0$. 
\end{proof}

\begin{defi}\label{def-densal}
Soit $X$ un espace $k$-analytique et soit $U$ un domaine analytique de $X$. Nous dirons que $U$ est \emph{analytiquement dense
dans $X$} s'il rencontre toutes les composantes assassines de $X$.
\end{defi}

\subsection{}\label{andense-zardense}
Soit $X$ un espace $k$-analytique. Si $U$ est un domaine analytique analytiquement dense de $X$,
il est Zariski-dense dans $X$ : en effet, comme il rencontre toutes les composantes assassines de $X$, 
il rencontre en particulier toutes ses composantes irréductibles, si bien que $\adhz UX=X$.

\subsection{}\label{ss-andense-interv}
Soit $X$ un espace $k$-analytique et soit $U$ un
ouvert de Zariski de $X$.
Supposons $U$ analytiquement dense. Dans ce cas $U\cap V$ est analytiquement dense dans $V$ pour tout domaine
analytique $V$ de $X$. En effet, soit $T$ une composante assassine de $V$. Son adhérence $Z:=\adhz TX$ est une composante assassine de $X$ (prop. \ref{prop-equiv-ass}), et elle rencontre donc $U$. L'ouvert de Zariski $U\cap Z$ de $Z$ est donc dense ; son intersection avec $Z\cap V$ est par conséquent dense dans ce dernier, et elle rencontre donc $T$ qui en est une composante irréductible. 

\begin{enonce}[theorem]{Lemme-définition}\label{lem-adh-analytique}
Soit $X$ un espace
$k$-analytique
et soit $U$ un domaine
analytique de $X$.
Soit $\mathscr I$
le faisceau d'idéaux de $\mathscr O_X$
qui envoie un domaine analytique $V$ de $X$
sur l'ensemble des $f\in \mathscr O_X(V)$ telle
que $f|_{V\cap U}=0$. 

\begin{enumerate}[1]
\item Le faisceau $\mathscr I$ est  cohérent ; on note $Y$
le sous-espace analytique fermé de $X$
qu'il définit. 

\item
Les composantes
irréductibles de $Y$ sont les composantes
irréductibles de $X$ qui rencontrent $U$ et
ses composantes assassines sont les 
composantes assassines de $X$ qui rencontrent $U$. 

\item 
L'immersion $U\hookrightarrow X$ se factorise par $Y$, et $Y$
est le plus petit sous-espace analytique fermé de $X$ possédant cette propriété. 
Le morphisme induit $U\to Y$ identifie $U$ à un domaine analytique de $Y$ analytiquement dense dans ce dernier. 

\end{enumerate}

On dit que $Y$ est \emph{l'adhérence analytique}
de $U$. 
\end{enonce}

\begin{proof}
Commençons par remarquer que la conjonction de (1) et (2) entraîne (3). Supposons en effet avoir montré (1) et (2). 
Si $\mathscr J$ est un faisceau cohérent sur $X$, l'immersion $U\hookrightarrow X$ se factorise par le sous-espace
analytique fermé correspondant à $\mathscr J$ si et seulement si $\mathscr J|_U=0$, ce qui revient par définition de $\mathscr I$ à demander que $\mathscr J\subset \mathscr I$ ; par conséquent, $Y$ est bien le plus petit sous-espace analytique fermé de 
$X$ par lequel $U\hookrightarrow X$ se factorise. L'égalité $\mathscr I_U=0$ signifie que $Y\cap U=U$, si bien que $U$ s'identifie à un domaine analytique de $Y$. Et l'assertion (2) assure que toute composante assassine de $Y$ rencontre $U$ ; par conséquent, ce 
dernier est analytiquement dense dans $Y$. 

Il reste donc à montrer (1) et (2). On procède en deux étapes. 

\subsubsection{Supposons que $U$ est un ouvert de Zariski de $X$}
\label{caspart-ouvzar}
Grâce à la proposition \ref{prop-equiv-ass}
on peut montrer (1) et (2) G-localement. 
On peut donc supposer $X$ affinoïde. 
On note $Y$ le sous-espace analytique fermé de $X$ défini par $\mathscr I(X)$ et
$\mathscr J$ le faisceau cohérent associé. Par définition de $Y$, le sous-schéma fermé $Y\al$ de
$X\al$ est l'adhérence schématique de $U\al$ ; ses composantes irréductibles
sont donc les composantes irréductibles de $X\al$ rencontrant $U\al$, ses composantes assassines sont les composantes assassines de $X\al$ qui rencontrent $U\al$, et $U\al$ s'identifie à un sous-schéma ouvert de $Y\al$. 
En vertu des principes GAGA (et notamment de la proposition \ref{pro-compass-aff}),
les composantes irréductibles de $Y$ sont les composantes irréductibles de $X$ rencontrant $U$, ses composantes
assassines sont les composantes assassines de $X$ rencontrant $U$, et $U$ s'identifie (morphiquement) à un 
ouvert de Zariski de $Y$ ; puisque toute composante assassine de $Y$ rencontre $U$, ce dernier est 
analytiquement dense dans $Y$. 

Il suffit donc pour conclure de démontrer que $\mathscr I$
coïncide avec le faisceau cohérent $\mathscr J$.
L'égalité $\mathscr J(V)=\mathscr I(V)$ implique que $\mathscr J\subset \mathscr I$  ; montrons
l'inclusion réciproque. 

Soit $V$ un domaine affinoïde de $X$ et soit $f\in \mathscr I(V)$.
Par définition, $f|_{V\cap U}=0$. Comme $U$ est un ouvert de Zariski analytiquement dense de $Y$, 
l'intersection $V\cap U$ est un ouvert de Zariski analytiquement dense de $V\cap Y$ (\ref{ss-andense-interv}). 
L'égalité $f|_{V\cap U}=0$ implique alors en vertu du lemme \ref{fonction-annulation} (à appliquer
au faisceau cohérent d'idéaux de $Y\cap V$ engendré par $f|_{Y\cap V}$)
que $f|_{Y\cap V}=0$, ce qui signifie précisément
que $f$ appartient à $\mathscr J(V)$ et achève la démonstration
de (1) et (2) lorsque $U$ est un ouvert de Zariski de $X$. 

\subsubsection{Preuve de \textnormal{(1)} et \textnormal{(2)} dans le cas général}
Soit $Z$ la réunion des composantes assassines de $X$ qui ne rencontrent pas $U$ ; posons $\Omega=X\setminus Z$.
Par construction $Z\cap \Omega=\emptyset$, si bien que $U\subset \Omega$. 
Soit $T$ une composante assassine de $\Omega$. Son adhérence $\adhz TX$ est une composante assassine de $X$
(prop. \ref{prop-equiv-ass}) qui n'est pas contenue dans $Z$, et qui par conséquent rencontre $U$. 
En vertu de la proposition \ref{prop-equiv-ass}, $T$ est une composante
irréductible de $\Omega\cap \adhz TX$ ; par ailleurs ce dernier, en tant qu'ouvert de Zariski non vide de l'espace
irréductible $\adhz TX$, est lui-même irréductible (lemme \ref{lem-ouvzar-connexe}) ; il
vient $\Omega \cap \adhz TX=T$. 

Puisque $\Omega\cap \adhz TX$ est dense dans $\adhz TX$ (qui est irréductible),
$\Omega\cap \adhz TX\cap U=T\cap U$ est dense dans $ \adhz TX\cap U$. Comme ce dernier est non vide, 
$T\cap U$ est non vide. Ainsi, toute composante assassine de $\Omega$ rencontre $U$ ; autrement dit, $U$
est analytiquement dense dans $\Omega$. 

Il s'ensuit d'après le lemme \ref{lem-adh-analytique}
que pour tout domaine analytique $V$ de $X$, l'idéal $\mathscr I(V)$ de $\mathscr O_X(V)$ peut également
se décrire comme l'ensemble des $f\in \mathscr O_X(V)$ telles que $f|_{V\cap \Omega}=0$. 
Il résulte alors du cas particulier traité en \ref{caspart-ouvzar},
que
$\mathscr I$ est cohérent et
que le sous-espace analytique fermé $Y$ qu'il définit a pour composantes irréductibles
(resp. assassines) les composantes irréductibles (resp. assassines) de $X$ qui rencontrent $\Omega$. 

Mais une composante assassine de $X$ qui rencontre $\Omega$ n'est pas contenue dans $Z$, ce qui implique
par définition de $Z$ qu'elles rencontre $U$. Ainsi  $Y$ a pour composantes irréductibles
(resp. assassines) les composantes irréductibles (resp. assassines) de $X$ qui rencontrent $U$. 
\end{proof}

\begin{rema}
Soit $X$ un espace $k$-analytique et soit $U$ un domaine analytique de $X$. 
On déduit de \ref{andense-zardense} que le support de l'adhérence analytique de
$U$
dans $X$ est égal à $\adhz UX$.
\end{rema}

\begin{lemm}\label{lem-ass-s1}
Soit $X$ un espace
$k$-analytique et soit $\mathscr F$ un faisceau cohérent sur $X$. Le faisceau $\mathscr F$
est $S_1$ si et seulement s'il ne possède aucune composante immergée.
\end{lemm}

\begin{proof}
Supposons que $\mathscr F$ soit $S_1$ et soit $Z$ une composante assassine de $\mathscr F$. 
Il existe un domaine affinoïde $V$ de $X$ et un fermé de Zariski $T$ de $V$ tel que
$T\al$ soit une composante assassine de $\mathscr F_V\al$ et tel que $\adhz TX=Z$ (remarque \ref{compass-autre-descrip}).
Comme
$\mathscr F$ est $S_1$, le faisceau $\mathscr F_V\al$ est $S_1$, et $T\al$ est dès lors
une composante irréductible de $\mathrm{Supp}(\mathscr F_V\al)$ ; par conséquent, $T$ est une composante
irrédutible de $\mathrm{Supp}(\mathscr F_V)$, et $Z$ est dès lors une composante
irréductible de $\mathrm{Supp}(\mathscr F)$.

Réciproquement, supposons que toute composante assassine de $\mathscr F$ soit une composante irréductible de
$\mathrm{Supp}(\mathscr F)$, 
et soit $V$ un domaine affinoïde de $X$. Soit $T$ un fermé de Zariski de $V$ tel que $T\al$ soit une composante
assassine de $\mathscr F_V\al$. Le fermé $T$ est une composante assassine de
$\mathscr F_V$ par la proposition \ref{pro-compass-aff} ; l'adhérence $Z:=\adhz TX$ est alors une composante
assassine de $\mathscr F$ et $T$ est une composante irréductible de $Z\cap V$ (prop. \ref{prop-equiv-ass}). 
Par hypothèse sur $\mathscr F$, le fermé $Z$ est une composante irréductible de $\mathrm{Supp}(\mathscr F)$, et $T$ est donc
une composante irréductible de $\mathrm{Supp}(\mathscr F_V)$. Ainsi, $T\al$ est une composante irréductible de
$\mathrm{Supp}(\mathscr F_V\al)$, et $\mathscr F_V\al$ est donc $S_1$. Il s'ensuit que $\mathscr F_V$ est $S_1$ ; ceci valant pour tout
domaine affinoïde $V$ de $X$, le faisceau $\mathscr F$ est $S_1$. 
\end{proof}

\begin{prop}\label{prop-compass-morph}
Soit $X$ un espace $k$-analytique et soit $Y\to X$
un morphisme dont la source $Y$ est $L$-analytique pour une certaine extension complète
$L$ de $k$. 
Soit $\mathscr F$ un faisceau
cohérent sur $X$ et soit $\mathscr G$ un faisceau cohérent sur $Y$.
On suppose que pour tout domaine affinoïde $V$ de $X$ et tout domaine
affinoïde $W$ de $Y\times_X V$, le faisceau $\mathscr G_W\al$ est plat sur $V\al$ 
et que sa restriction à chaque fibre de $W\al \to V\al$ est $S_1$. 
Soit $Z$ un fermé de Zariski de $Y$. 
Les assertions suivantes sont équivalentes : 

\begin{enumerate}[i]
\item $Z$ est une composante assassine de
$\mathscr F\boxtimes \mathscr G:=\mathscr F_Y\otimes_{\mathscr O_Y}
\mathscr G$ ; 
\item il existe une composante assassine $T$ de $\mathscr F$ telle que $Z$
soit une composante irréductible de $\mathrm{Supp}(\mathscr G)\times_X T$. 
\end{enumerate}
\end{prop}

\begin{rema}\label{rema-compass-morph}
Les hypothèses de la proposition 
sont notamment satisfaites dans les cas importants qui suivent. 

\begin{enumerate}[1]
\item \emph{L'espace $Y$ est $k$-analytique,
$\mathscr G$ est $X$-plat et 
$\mathscr G_{Y_x}$ est $S_1$ pour tout $x\in X$}. 
En effet, soient $W$ et $V$ comme dans l'énoncé. La $X$-platitude de $\mathscr G$ entraîne que $\mathscr G_W\al$ est plat sur $V\al$. Soit $\xi\in V\al$, soit $\eta$ un antécédent de $\xi$ sur $W\al$, et soit $\mathscr X$ l'adhérence de $\xi$ dans $V\al$, munie de sa structure réduite. Comme $\mathscr G_{W\times_V\mathscr X\an}$ est plat sur $\mathscr X\an$ et $S_1$ en restriction à chaque fibre, on déduit de \cite{ducros2018}, Th. 11.3.3 (2b) que $\mathscr G_{W\times_V \mathscr X\an}$ est $S_1$ en tout point de $W$ situé au-dessus du lieu $S_1$ de $\mathscr X\an$. Choisissons maintenant un antécédent $y$ de $\eta$ sur $W$, et soit $x$ l'image de $y$ sur $V$ ; le point $x$ est situé au-dessus de $\xi$. Comme $\xi$ est le point générique du schéma $\mathscr X$, le schéma $\mathscr X$ est $S_1$ en $\xi$, et $\mathscr X\an$ est dès lors $S_1$ en $x$ ; par ce qui précède, $\mathscr G_{W\times_V \mathscr X\an}$ est $S_1$ en $y$ ; il s'ensuit que $(\mathscr G_W\al)_{W\al\times_{V\al}\mathscr X}$ est $S_1$ en $\eta$ ; puisque $\xi$ est le point générique de $\mathscr X$, cela signifie exactement que $(\mathscr G_W\al)_{W\al_\xi}$ est $S_1$ en $\eta$.

\item \emph{L'espace $Y$ est $k$-analytique,
$Y\to X$ est quasi-lisse, 
$\mathscr F=\mathscr O_X$ et $\mathscr G=\mathscr O_Y$}. En effet, on 
est alors dans un cas particulier de (1) puisqu'un morphisme quasi-lisse est plat à fibres géométriquement régulières (et en particulier $S_1$) ; on pourrait aussi invoquer directement le th. 5.5.3 (2) de \cite{ducros2018}. 

\item \emph{L'espace $Y$ est égal à $X_L$, 
le faisceau $\mathscr G$ est égal à $\mathscr O_Y$ et $\mathscr F=\mathscr O_X$}.
En effet, soient $V$ et $W$ comme dans la proposition. Comme $W$ s'identifie à un domaine
affinoïde de $V_L$, le morphisme $W\al \to {V_L}\al$ 
est régulier (\cite{ducros2009}, thm. 3.3). D'après \cite{ducros2018}, thm. 5.5.3  (2) de \emph{op. cit.}, 
le morphisme $V_L\al \to V\al$ est plat et à fibres d'intersection complète, et en particulier $S_1$. 
Il s'ensuit (\cite{ega42}, prop. 6.4.1) que $W\al \to V\al$ est plat et à fibres $S_1$. 
\end{enumerate}
\end{rema}

\begin{proof}[Démonstration de la proposition \ref{prop-compass-morph}]
Traitons d'abord
le cas particulier où $X$ et $Y$ sont affinoïdes. 
Puisque $\mathscr G\al$ est
plat sur $X\al$ par hypothèse, on déduit
de \ref{ss-plat-prof} qu'un
point $\eta$ de $Y\al$ appartient
à $\mathrm{Ass}(\mathscr F\boxtimes \mathscr G)$ si et seulement si son image $\xi$ sur $X\al$
appartient à $\mathrm{Ass}(\mathscr F)$ et $\eta$ appartient à $\mathrm{Ass}(\mathscr G\al_{Y\al_\xi})$. Mais
comme $\mathscr G\al_{Y\al_\xi}$ est $S_1$
par hypothèse, il est sans composante immergée, et
$\eta$ appartient donc à $\mathrm{Ass}(\mathscr G\al_{Y\al_\xi})$
si et seulement si $\eta$ est un point générique d'une composante irréductible du support de $\mathscr G\al_{Y\al_\xi}$ ;
par platitude, cette dernière condition
revient à demander que $\eta$ soit le point générique d'une composante irréductible
de $\mathrm{Supp}(\mathscr G\al)\times_{X\al}\adht {\{\xi\}} {X\al}$. 
Ainsi, un fermé $\mathscr Z$ de $Y\al$ est une composante
assassine de $(\mathscr F\boxtimes \mathscr G)\al$  si et seulement si il existe une composante
assassine $\mathscr T$ de $\mathscr F\al$ telle que $\mathscr Z$ soit une composante
irréductible de $\mathrm{Supp}(\mathscr G\al)\times_{X\al}\mathscr T$ ; la proposition s'en déduit alors
par le principe GAGA pour les composantes assassines (proposition \ref{pro-compass-aff} (2)). 

On ne suppose plus maintenant que $X$ et $Y$ sont affinoïdes. Supposons que (i) est vraie ; notons
que $Z\subset \mathrm{Supp}(\mathscr G)$ par définition. 
Choisissons un domaine affinoïde $V$ de $X$ qui rencontre l'image de $Z$ et un domaine
affinoïde $W$ de $Y\times_X V$ qui rencontre $Z$. Soit $Z_1$ une composante irréductible
de $Z\cap W$. C'est une composante assassine de $\mathscr F_V\boxtimes \mathscr G_W$, 
et $Z=\adhz {Z_1}{\supp G}$
(prop. \ref{prop-equiv-ass}). D'après le cas
affinoïde déjà traité, il existe donc une composante assassine $T_1$ de $\mathscr F_V$ telle que $Z_1$
soit une composante irréductible de $\mathrm{Supp}(\mathscr G_W)\times_V T_1$. Posons $T=\adht {T_1}X$. C'est une composante assassine de $\mathscr F$
et $T_1$ est une composante irréductible de $T\cap V$
(prop. \ref{prop-equiv-ass}).
Par platitude
de $\mathscr G_W\al$ sur $V\al$, les composantes irréductibles de $\mathrm{Supp}(\mathscr G_W\al)\times_{V\al}T_1\al$ 
sont des composantes irréductibles de $\mathrm{Supp}(\mathscr G_W\al)\times_{V\al}(T\cap V)\al$. Par conséquent, $Z_1$ est une composante
irréductible de $\mathrm{Supp}(\mathscr G_W)\times_V (T\cap V)$, c'est-à-dire encore de $(\mathrm{Supp}(\mathscr G)\times_X T)\cap W$. 
On en déduit que $Z=\adht{Z_1}{\mathrm{Supp}(\mathscr G)}$ est une composante irréductible de 
$\mathrm{Supp}(\mathscr G)\times_X T$.

Supposons maintenant que (ii) est vraie. Soit $V$ un domaine affinoïde de $X$ rencontrant l'image de $Z$, et soit $W$ un domaine affinoïde
de $Y\times_X V$ rencontrant $Z$. Soit $Z_1$ une composante irréductible de $Z\cap W$ ; c'est une composante irréductible de $(\mathrm{Supp}(\mathscr G)\times_X T)\cap W$, c'est-à-dire de $\mathrm{Supp}(\mathscr G_W)\times_V (T\cap V)$. Si on désigne par $T_1$ une
composante irréductible  de $V\cap T$ contenant l'image de $Z_1$
alors  $Z_1$ est une composante irréductible de $\mathrm{Supp}(\mathscr G_W)\times_V T_1.$ (Par platitude, $T_1$ est unique et $Z_1\al \to T_1\al$ est dominant, mais peu importe ici). 

Or $T$ est une composante assassine de $\mathscr F$. Il s'ensuit que $T_1$ est une composante assassine de $\mathscr F_V$
(prop. \ref{prop-equiv-ass}) ; par le cas affinoïde déjà traité,
$Z_1$ est une composante assassine de $\mathscr F_V\boxtimes \mathscr G_W$ ; en utilisant une dernière
fois la proposition \ref{prop-equiv-ass} on en déduit que
$Z=\adht {Z_1}Y$ est une composante
assassine de  $\mathscr F\boxtimes \mathscr G$.
\end{proof}

Nous avons vu plus haut (prop. \ref{pro-compass-aff})
un résultat de type GAGA pour les composantes assassines lorsque l'espace
ambiant est affinoïde. Cela s'étend en fait au cadre plus général de l'analytifié d'un schéma
de type fini sur un espace affinoïde. 

\begin{lemm}\label{lem-ass-gaga}
Soit $A$ une algèbre
$k$-affinoïde, soit $\mathscr X$ un $A$-schéma
de type fini, et soit $\mathscr F$ un faisceau cohérent sur $\mathscr X$. 
Les composantes assassines de $\mathscr F\an$ sont exactement les fermés de $\mathscr X\an$
de la forme $\mathscr Y\an$, où $\mathscr Y$ est une composante assassine de $\mathscr F$. 
\end{lemm}

\begin{proof}
Nous allons procéder par double inclusion.

\subsubsection{}
Soit $\mathscr Y$ une composante assassine de $\mathscr F$.
Comme $\mathscr Y$ est irréductible, $\mathscr Y\an$ l'est aussi  ; 
soit $n$ la dimension de $\mathscr Y\an$. Soit $x$ un point de $\mathscr Y\an$ 
tel que $d_k(x)=n$ ; son adhérence $\adhp x{\mathscr X\an}$ est égale à
$\mathscr Y\an$ (et $x\al$ est donc le point générique de $\mathscr Y$).
Choisissons un domaine affinoïde $V$
de $\mathscr X\an$ contenant $x$, et posons
$Z=\adhp xV$; puisque $d_k(x)=n$, on a $\dim Z=n$.
Comme la liste des dimensions des composantes irréductibles de
$\mathrm{Supp}( \mathscr F\an_V)$ contenant $x$ est la même que la liste des dimensions
des composantes irréductibles de $\mathrm{Supp}( \mathscr F\an)$
contenant $x$, on
en déduit
que la codimension de Krull
de $Z$ dans $\mathrm{Supp}( \mathscr F\an_V)$ est égale à la codimension de $\mathscr Y\an$
dans $\mathrm{Supp}(\mathscr F\an)$,
c'est-à-dire à celle de $\mathscr Y$ dans $\mathrm{Supp}(\mathscr F)$ (\cite{ducros2018}, Cor. 2.7.13).
Autrement dit, $\mathscr F_{V\al,x_V\al}$ et $\mathscr F_{\mathscr X,x\al}$ ont même dimension.
Comme
$V\al \to \mathscr X$
est régulier (th. 3.3 de \cite{ducros2009}), 
et
comme
$\mathscr F_{\mathscr X,x\al}$ est de profondeur nulle
(puisque $x\al$ est le point générique de la composante assassine $\mathscr Y$ de $\mathscr F$),
la profondeur de $\mathscr F_{V\al, x_V\al}$ est nulle par \ref{sss-equiv2-prof0}.
Comme
$d_k(x)=\dim Z$, on a $\dimcg Vx=\dim Z$ (\ref{dimcent-dx}), et $x$ appartient donc à $\mathrm{Ass}(\mathscr F_V\an)
\subset \mathrm{Ass}(\mathscr F\an)$
en vertu de la proposition \ref{prop-equiv-ass} ; par conséquent, 
$\mathscr Y\an$ est une composante assassine de $\mathscr F\an$. 

\subsubsection{}
Soit $Y$ une composante assassine de $\mathscr F\an$. Nous allons montrer qu'elle
est de la forme requise, en commençant par traiter deux cas particulier. 

\paragraph{Supposons $\mathscr X$ propre sur $A$} Par GAGA, la composante $Y$ est égale à $\mathscr Y\an$ pour un certain
fermé de Zariski irréductible $\mathscr Y$ de $\mathscr X$. Soit $x$ un point de $\mathrm{Ass}(\mathscr F\an)$
tel que $\overline{\{x\}}^{\mathscr X\an_{\mathrm{Zar}}}=Y$ ; le point $x\al$ est alors le point générique de $\mathscr Y$.
Comme $x\in \mathrm{Ass}(\mathscr F\an)$, la profondeur de $\mathscr F\an_{\mathscr X\an,x}$ est nulle ; par platitude
du morphisme d'espaces annelés $\mathscr X\an \to \mathscr X$ (\cite{berkovich1993}, Prop. 2.6.2)
et par \ref{sss-equiv2-prof0}, la profondeur
de $\mathscr F_{\mathscr X,x\al}$ est nulle ; par conséquent $x\al \in \mathrm{Ass}(\mathscr F)$ et
$\mathscr Y$ est une composante assassine de $\mathscr F$.

\paragraph{Supposons $\mathscr X$ affine}
Choisissons une compactification
projective $\overline{\mathscr X}$ du $A$-schéma $\mathscr X$,
et un prolongement
$\overline{\mathscr F}$ de $\mathscr F$ à $\overline{\mathscr X}$
(\cite{ega1} Cor. 9.4.8). 
Par la proposition \ref{prop-equiv-ass}, l'adhérence
$Z:=\adht Y{\mathscr X\an_{\mathrm{Zar}}}$ est une composante assassine de 
$\overline{\mathscr F}\an$, et $Y$ est une composante irréductible de $Z\cap \mathscr X\an$ ; 
par ailleurs $Z\cap \mathscr X\an$ est irréductible en vertu du lemme \ref{lem-ouvzar-connexe}, 
si bien que $Y=Z\cap \mathscr X\an$. 
Par le cas propre déjà traité, $Z=\mathscr Z\an$ pour une certaine composante assassine
$\mathscr Z$ de $\overline {\mathscr F}$. 
Il vient $Y=Z\cap \mathscr X\an=(\mathscr Z\cap \mathscr X)\an$, ce qui termine la preuve dans le cas
affine puisque $\mathscr Z\cap \mathscr X$ est une composante assassine de $\mathscr F$.

\paragraph{Preuve dans le cas général}
Choisissons un ouvert affine $\mathscr U$ de $\mathscr X$ tel que
$\mathscr U\an$ rencontre $Y$. L'intersection $Y\cap \mathscr U\an$ est irréductible
d'après le lemme \ref{lem-ouvzar-connexe}, et c'est donc une composante assassine
de $(\mathscr F_{\mathscr U})\an$ par la proposition \ref{prop-equiv-ass}, qui assure aussi
que $\adht{(Y\cap \mathscr U\an)}{\mathscr X\an_{\mathrm{Zar}}}=Y$. 

Par le cas affine déjà traité , $Y\cap \mathscr U\an$
est égale à $\mathscr Z\an$ pour une certaine
composante assassine $\mathscr Z$ de $\mathscr F_{\mathscr U}$. Soit $\mathscr Y$ l'adhérence de $\mathscr Z$
dans $\mathscr X$ ; c'est une composante assassine de $\mathscr F$. Puisque $\mathscr Y$
est irréductible, $\mathscr Y\an$ est un fermé de Zariski irréductible de $\mathscr X\an$ 
et son ouvert de Zariski non vide $Y\cap \mathscr U\an=\mathscr Z\an$ est donc Zariski-dense. 
On a par conséquent
\[\mathscr Y\an= \adht{(Y\cap \mathscr U\an)}{\mathscr X\an_{\mathrm{Zar}}}=Y,\]
ce qui achève la démonstration. 
\end{proof}

\subsection{Fonctions régulières et fonctions méromorphes}
Soit $X$ un espace $k$-analytique. Nous allons brièvement
rappeler les définitions 
de fonction régulière et de fonction méromorphe sur $X$, et établir  quelques énoncés élémentaires à leur sujet qui sont certainement bien connus, mais dont les preuves ne sont à notre connaissance pas disponibles dans la littérature 
(une partie de ce qui suit figure toutefois dans la prépublication \cite{chambertloir-d2012}). 

\subsubsection{}
On dit qu'une fonction $f\in \mathscr O_X(X)$ est \emph{régulière} si la
multiplication par $f$ est un endomorphisme injectif de $\mathscr O_X$.
C'est une propriété qu'on peut vérifier G-localement ; si $X$ est affinoïde, $f$ est régulière si et seulement si c'est un élément régulier de $\mathscr O_X(X)$ au sens usuel de l'algèbre commutative, c'est-à-dire un élément non diviseur de zéro.

\subsubsection{}\label{sss-reg-ass}
Une   fonction $f\in \mathscr O_X(X)$
est régulière si et seulement si son lieu des zéros ne contient aucune composante assassine de $X$. En effet, le caractère G-local de la régularité et la proposition \ref{prop-equiv-ass} (2) permettent de se ramener au cas affinoïde, dans lequel l'assertion découle du résultat analogue en théorie des schémas et du principe GAGA pour les composantes assassines (\ref{pro-compass-aff} (2)). 

\subsubsection{}
Soit $\mathscr S$ le sous-faisceau de $\mathscr O_X$ qui associe à un domaine analytique $U$ l'ensemble des sections régulières de $\mathscr O_X(U)$. On appelle \emph{faisceau des fonctions méromorphes sur $X$} et l'on note $\mathscr K_X$ le faisceautisé (pour la G-topologie) 
de $U\mapsto \mathscr S(U)\inv \mathscr O_X(U)$. 
Il résulte de la définition de $\mathscr S$ que le morphisme naturel 
$\mathscr O_X\to \mathscr K_X$ est injectif. On se permettra donc
d'identifier $\mathscr O_X$ à un sous-faisceau de $\mathscr K_X$. 

\subsubsection{}\label{sss-kx-ka}
Supposons $X$ affinoïde, disons $X=\mathscr M(A)$. \emph{L'application naturelle de l'anneau total des fractions de $A$ dans $\mathscr K_X(X)$ est un isomorphisme}.

L'injectivité est une conséquence formelle de l'injectivité de $A\to \mathscr K_X(X)$. Prouvons la surjectivité. 
Soit $f\in \mathscr K_X(X)$ et soit $\mathscr D$ le «faisceau des dénominateurs de $f$», c'est-à-dire le faisceau d'idéaux de $\mathscr O_X$ qui envoie un domaine analytique $U$
de $X$ sur l'ensemble 
des $h\in \mathscr O_X(U)$ tels que $hf\in \mathscr O_X(U)\subset \mathscr K_X(U)$.

Montrons que 
$\mathscr D$ est cohérent. On peut raisonner G-localement et partant supposer que $f=g/h$ où $g$
et $h$ appartiennent à $\mathscr O_X(X)$ et où
$h$ est régulière.
Mais dans ce cas $\mathscr D$ est le noyau du morphisme $\mathscr O_X\to \mathscr O_X/(h)$ induit par la multiplication par $g$, d'où l'assertion. 

Il suffit maintenant pour conclure de montrer que $\mathscr D$ possède une section globale qui est régulière. Procédons par l'absurde, en supposant que toute section globale de
$\mathscr D$ est diviseur de zéro. Le lemme d'évitement des idéaux premiers
couplé au principe GAGA pour les composantes
assassines (prop. \ref{pro-compass-aff} (2)) 
assure alors qu'il existe une composante assassine $Y$ de $X$
contenue dans le lieu des zéros de $\mathscr D$. Soit $x\in Y$ et soit $V$ un domaine affinoïde de $X$ contenant $y$
tel que $f|_V$ soit de la forme $g/h$ où
$g$
et $h$ appartiennent à $\mathscr O_X(V)$ et où
$h$ est régulière. Le lieu des zéros de $\mathscr D_V$ contient $Y\cap V$, qui est non vide (il contient $x$)
et est en vertu de la proposition \ref{prop-compass-morph} (2) une union de composantes assassines de $V$. Il en résulte 
en vertu de \ref{sss-reg-ass}
que $\mathscr D(V)$ ne contient aucune fonction régulière, contredisant le fait que $h$
appartient manifestement
à $\mathscr D(V)$. 

\subsubsection{}
On dit qu'un faisceau cohérent d'idéaux $\mathscr I$ sur $X$
est \emph{inversible} si $\mathscr I$ est G-localement libre de rang $1$, c'est-à-dire encore si $\mathscr I$ est G-localement
engendré par une 
section régulière de $\mathscr O_X$.

Soit $\mathscr I$ un faisceau cohérent d'idéaux sur $X$. 

\paragraph{}\label{par-i-inv}
Supposons $\mathscr I$ inversible. On note $\mathscr I\inv$ le sous-$\mathscr O_X$-module
de $\mathscr K_X$ constitué des sections $g$ telles que $gf$ soit une section de $\mathscr O_X$ pour toute section $f$ de $\mathscr I$. C'est un $\mathscr O_X$-module cohérent, et même G-localement libre de rang $1$ : si $h$ est une fonction régulière engendrant $\mathscr I$ sur un domaine analytique $U$ de $X$ alors $h\inv$ engendre $\mathscr I\inv$ sur $U$. 
Cette dernière remarque montre que le sous-faisceau $\mathscr I\cdot \mathscr I\inv$ de $\mathscr K_X$ est égal à 
$\mathscr O_X$. 

\paragraph{}\label{par-prod-inv}
Réciproquement, supposons qu'il existe un sous-$\mathscr O_X$-module cohérent $\mathscr J$ de $\mathscr K_X$
tel que $\mathscr I\cdot \mathscr J=\mathscr O_X$ ; nous allons montrer que $\mathscr I$ est inversible.
On peut raisonner G-localement. On peut donc supposer que $X=\mathscr M(A)$ pour une certaine algèbre
$k$-affinoïde $A$, que $\mathscr J(A)\subset h\inv A$ pour un certain élément
régulier $h$ de $A$, et que $1\in \mathscr I(A)\cdot \mathscr J(A)$.
Cette dernière condition assure que $\mathscr I(A)\cdot \mathscr J(A)
=A$, 
puis que  $\mathscr I(A)\cdot (h\mathscr J(A))=hA$. Comme $h\mathscr J(A)$
est un idéal de $A$ par choix de $h$,  on déduit du lemme
\cite[\href{https://stacks.math.columbia.edu/tag/0801}{Tag 09ME}]{stacks-project}
que $\mathscr I(A)$ est un $A$-module projectif de rang $1$ ; par conséquent, $\mathscr I$
est localement libre de rang $1$.

On déduit de ce qui précède que pour qu'un produit
fini $\mathscr I_1\cdot\ldots\cdot \mathscr I_n$ de faisceaux cohérents d'idéaux sur $X$ soit
inversible, il faut (et il suffit) que chacun des $\mathscr I_j$ le soit.

\subsection{}
Soit $X$ un espace $k$-analytique, soit $L$ une extension complète de $k$ et soit $X'$ un espace $L$-analytique. Soit $\phi \colon X'\to X$ un morphisme tel que le morphisme induit $X'\to X_L$ soit plat. 

\subsubsection{}\label{sss-plat-fonctreg}
Pour tout domaine affinoïde
$U$ de $X$ et
tout domaine affinoïde $V$ de $\phi\inv(U)$, la $\mathscr O_X(U)$-algèbre $\mathscr O_{X'(V)}$ est plate. On en déduit aussitôt que si $f$ est une fonction régulière sur un domaine analytique $W$ de $X$, son image réciproque $\phi^*f$ est une fonction régulière sur $\phi\inv(W)$. 

\subsubsection{}\label{sss-plat-idinv}
Soit $\mathscr I$ un faisceau cohérent d'idéaux sur $X$. Si $\mathscr I$ est inversible, il résulte de \ref{sss-plat-fonctreg} que 
$\phi\inv(\mathscr I)$ est inversible. 
\section{Idéal des coefficients}
\label{s-ideal}

\subsection{}
Nous dirons qu'un morphisme $f\colon Y\to X$ entre espaces analytiques est \emph{compact}
si $f\inv(V)$ est compact pour tout domaine affinoïde $V$ de $X$. C'est par exemple le cas
si $Y$ est compact et $X$ topologiquement séparé, ou si $f$ est propre.

\subsection{}
Un \emph{polyrayon}
est une famille finie de réels strictement positifs. Un polyrayon $r=(r_1,\ldots, r_n)$
est dit \emph{$k$-libre}
si les $r_i$ forment une famille libre du $\Q$-espace vectoriel $\R_{>0}/\abs{k\gpm}^\Q$.

Si $r$ est $k$-libre, on notera $k_r$
la $k$-algèbre 
constituée des séries $\sum_{I\in \Z^n}a_I T^I$ à coefficients dans $k$
telles que $\abs{a_I}r^I$ tende vers zéro quand $\abs I$ tend vers l'infini.
La norme $\sum a_I T^I\mapsto \max \abs{a_I}r^I$ fait de $k_r$
à la fois une algèbre de Banach $k$-affinoïde et une extension
complète de $k$.
Si $A$ (resp. $X$) est une algèbre $k$-affinoïde (resp. un espace
$k$-analytique) on écrira $A_r$ (resp. $X_r$) au lieu de $A\hotimes_k k_r$ (resp. $X_{k_r}$).

\begin{theo}[Descente fidèlement plate des faisceaux cohérents]\label{theo-desc-coh}
Soit $Y\to X$ un morphisme compact, plat et surjectif entre espaces
$k$-analytiques.

\begin{enumerate}[1]

\item Soit $\mathscr F$ un faisceau cohérent sur $X$. La suite

\begin{tikzcd}
0\ar[r]&\mathscr F(X)\ar[r]&\mathscr F_Y(Y)
\ar[r, shift left=1ex]\ar[r, shift right=1ex]&\mathscr
 F_{Y\times_XY}(Y\times_XY)
\end{tikzcd}

est exacte. 
\item 
La catégorie
des faisceaux cohérents sur $X$ est naturellement équivalente à celle des faisceaux cohérents sur $Y$
munis d'une donnée de descente relativement à $Y\to X$. 
\end{enumerate}
\end{theo}

\begin{rema}
Dans le cas strictement $k$-analytique, ce théorème est dû à Bosch et 
Görz (\cite{bosch-g1998}, thm. 3.1) mais nous n'utilisons pas leur résultat. Nous redémontrons donc
celui-ci, par des
méthodes complètement différentes (les leurs font appel à la géométrie
formelle suivant le point de vue de Raynaud). 
\end{rema}

\begin{rema}
Le théorème est faux sans hypothèse de compacité. Pour le voir, fixons deux réels strictement positifs $r$ et $s$ avec $r<s$, et notons $X$ la couronne $r\leq \abs T\leq s$. Pour tout $t$ compris entre $r$ et $s$, notons $X_t$ la couronne $\abs T=t$, qui est un domaine affinoïde de $X$ ; enfin, posons $Y=\coprod_{r\leq t\leq s}X_t$. Le morphisme naturel $Y\to X$ est alors plat et surjectif, mais ne satisfait visiblement ni (1) ni (2). 
\end{rema}

\begin{proof}[Démonstration du théorème \ref{theo-desc-coh}]
Il suffit de prouver (2) : en effet, (1)
en est une conséquence formelle car on a pour tout espace
analytique $Z$ et tout faisceau cohérent $\mathscr G$ sur $Z$ un isomorphisme
$\mathscr G(Z)\simeq \mathrm{Hom}(\mathscr O_Z,\mathscr G)$
fonctoriel en $Z$ et $\mathscr G$. 

Montrons donc (2). L'assertion est G-locale sur $X$, ce qui permet de le supposer affinoïde ; l'espace $Y$ est alors compact. 
On note $A$ l'algèbre des fonctions
analytiques sur $X$, et l'on se donne deux faisceaux cohérents
$\mathscr F$ et $\mathscr G$ sur $Y$ munis de données de descente relativement
à $X$, et une application $\mathscr O_Y$-linéaire $u \colon \mathscr F\to \mathscr G$ compatible aux données de descentes considérées. 
Le but de ce qui suit est de montrer que $\mathscr F$ et $\mathscr G$ proviennent de deux faisceaux cohérents $\mathscr F_0$ et $\mathscr G_0$
sur $X$, et que $u_0$ provient d'une unique application $\mathscr O_X$-linéaire de $\mathscr F_0$ dans $\mathscr G_0$. 

\subsubsection{}\label{des-coh-xr}
Supposons tout d'abord que $Y$ est de la forme $X_r$ pour un certain polyrayon $k$-libre $r$. Soit $A$ 
l'algèbre des fonctions analytiques sur $X$. Les faisceaux cohérents
$\mathscr F$ et $\mathscr G$ correspondent respectivement à des $A_r$-modules
de Banach finis $M$ et $N$. 

Notons $T$ la famille des fonctions coordonnées de $k_r$, et $T_1$ et $T_2$ les deux familles de fonctions coordonnées de $k_r\hotimes_k k_r$, et posons
$B=A_r\hotimes_A A_r$. On dispose d'identifications naturelles
\[A_r=A\{r\inv T, rT\inv\}=A\{r\inv T_1, rT_1\inv\}=A\{r\inv T_2, rT_2\inv\}.\]
Lorsqu'un «objet mathématique»
défini sur $A_r$
sera vu comme défini sur l'anneau $A\{r\inv T_i, rT_i\inv\}$ pour $i=1$ ou $2$,
nous l'indiquerons par un $i$ en indice.  
On a
\[B=A_{r,1}\{r\inv T_2, rT_2\inv\}=A_{r,2}\{r\inv T_1, rT_1\inv\}.\]
Nous allons tout d'abord montrer que $M$ provient d'un $A$-module de Banach fini $M_0$
La donnée de descente dont $\mathscr F$ est muni
consiste en un isomorphisme de $B$-modules $\iota$ entre 
$M_1\{r\inv T_2, rT_2\inv\}$ et $M_2\{r\inv T_1, rT_1\inv\}$. 

Donnons-nous un élément $\sum_{I,J}m_{I,J,1}T_2^IT_3^J$
de $M_1\{r\inv T_2, rT_2\inv, r\inv T_3, rT_3\inv\}$. Pour tout $I$,
écrivons $\iota(\sum_J m_{I,J,1}T_2^J)=\sum_J n_{I,J,2}T_1^J$. Pour tout $J$, écrivons
$\iota(\sum_I m_{I,J,1}T_2^I)=\sum_I\ell_{I,J,2}T_1^I$ ; et écrivons enfin
$\iota(\sum_J \ell_{I,J,1}T_2^J)=\sum_J\lambda_{I,J,2}T_1^J$ pour tout $I$. 
Lorsqu'on applique la condition de cocycle satisfaite par $\iota$
à $\sum_{I,J}m_{I,J,1}T_2^IT_3^J$, il vient $n_{I,J}=\lambda_{J,I}$ pour tout $(I,J)$.

Soit $M_0$ le sous-ensemble de $M$ constitué des éléments $m$ tels 
$\iota(m_1)=m_2$ ; c'est un $A$-module
de Banach. Soit $m\in M$ ; écrivons $\iota(m_1)=\sum_I \mu_{J,2} T_1^J$. 
Appliquons ce qui précède avec $m_{I,J}=m$ si $(I,J)=(0,0)$ et $0$ sinon. 
On obtient $n_{I,J}=\mu_J$ si $I=0$, et $0$ sinon ; puis $\ell_{I,J}=\mu_I$ si $J=0$ et $0$ sinon. 
On a $\lambda_{I,J}=n_{J,I}$ pour tout $(I,J)$, c'est-à-dire
$\lambda_{I,J}=\mu_I$ si $J=0$ et $0$ sinon. Il vient $\iota(\mu_{J,1})=\mu_{J,2}$ 
pour tout $J$ ; autrement dit, $\mu_J$ appartient à $M_0$ pour tout $I$. Comme $\iota$ est un morphisme
injectif de $B$-modules, on a nécessairement $m_1=\sum \mu_{J,1} T_1^J$. Le morphisme naturel
$M_0\{r\inv T, r T\inv\}\to M_1$ qui envoie $\sum \alpha_J T^J$ sur $\sum \alpha_{J,1}T_1^J$ 
est donc surjectif ; et en vertu de l'égalité
\[\iota(\sum \alpha_{J,1}T_1^J)=\sum \alpha_{J,2} T_1^J\in M_2\{r\inv T_1,rT_1\inv\},\]
il est injectif et admissible ; le module $M$ s'identifie donc à $M_0\hotimes_A A_r$. Comme $M$ est un $A_r$-module
de Banach fini, $M_0$ est un $A$-module de Banach fini, d'après \cite{berkovich1990}, Prop. 2.1.11.

Modulo l'isomorphisme $M\simeq M_0\hotimes_A A_r$ construit ci-dessus, l'isomorphisme $\iota$ est simplement l'isomorphisme
évident 
\[M_0\{r\inv T_1,rT_1\inv\}\{r\inv T_2,rT_2\inv\}\simeq M_0\{r\inv T_2,rT_2\inv\}\{r\inv T_1,rT_1\inv\}.\]

On fabrique par le même procédé un $A$-module de Banach fini $N_0$ et un isomorphisme
$N\simeq N_0\hotimes_A A_r$ modulo lequel la donnée de descente fournie avec $N$ est l'isomorphisme évident
\[N_0\{r\inv T_1,rT_1\inv\}\{r\inv T_2,rT_2\inv\}\simeq N_0\{r\inv T_2,rT_2\inv\}\{r\inv T_1,rT_1\inv\}.\]

Soit $m\in M_0$ ; écrivons $u(m)=\sum n_I T^I$.
Appliquons la commutation de $u$ aux données de descente à l'élément
$m$ de $M_0\subset M_0\{r\inv T_1,rT_1\inv\}\{r\inv T_2,rT_2\inv\}$ ; il vient
$\sum n_I T_1^I=\sum n_IT_2^I$, ce qui entraîne que tous les $n_I$ sont nuls à l'exception de $n_0$. 
Ainsi $u(M_0)\subset N_0$ et $u$ induit donc une application $A$-linéaire $u_0$ de $M_0$ dans $N_0$ ; 
comme $u$ est bornée, $u_0$ est bornée. Il est alors immédiat que $u$ est l'application déduite de $u_0$
par changement de base de $A$ à $A_r$, et que $u_0$ est la seule
application $A$-linéaire bornée de $M_0$ dans $N_0$ qui ait cette propriété; ceci achève la preuve du cas où $Y=X_r$. 

\subsubsection{}En vertu du cas particulier traité en \ref{des-coh-xr}
ci-dessus, on peut démontrer le théorème après extension des scalaires à $k_r$ pour n'importe
quel polyrayon $k$-libre $r$ ; cela permet de supposer que la valeur absolue de $k$ n'est pas triviale
et que $Y$ et $X$ sont strictement $k$-analytiques. On déduit alors des théorèmes
8.4.6 et 9.1.3 de \cite{ducros2018}
qu'il existe : 
\begin{itemize}[label=$\bullet$]
\item une famille finie $f_i \colon Z_i\to X$ de morphismes quasi-étales de source affinoïde, 
telle que $X=\bigcup f_i(Z_i)$ ; 
\item pour tout $i$, un morphisme fini, plat et surjectif $Z'_i\to Z_i$ et un $X$-morphisme $Z_i \to Y$. 
\end{itemize}
Posons $Z=\coprod Z_i$ et $Z'=\coprod Z'_i$.
Puisque $Y\times_X Z'\to Z'$ a par construction une section, il suffit de démontrer le
théorème lorsque $Y=Z'$ ; le morphisme $Y\to X$ se factorise alors par une flèche finie et plate $Y\to Z$ suivie d'une flèche
quasi-étale 
$Z\to X$. Par descente
finie et plate schématique, on peut supposer que $Y=Z$ ; il reste donc à traiter le cas où $Y\to X$ est quasi-étale. 
En raisonnant localement sur $X$, on peut supposer qu'il existe un revêtement fini
galoisien $X'\to X$ tel que $Y$ s'écrive comme une union 
finie $\bigcup Y_i$ où chaque $Y_i$ est un domaine affinoïde d'un quotient $X'_i$ de $X'$. En utilisant encore une fois
encore la descente étale schématique, on voit qu'on peut démontrer le théorème après changement de base de $X$ à $X'$, 
et l'on se ramène ainsi au cas où $Y$ s'écrit  $\bigcup Y_i$ où chaque $Y_i$ s'identifie à un domaine affinoïde de $X$. 

Puisqu'un faisceau cohérent est un faisceau pour la G-topologie, et puisque la cohérence est une notion G-locale, les faisceaux
cohérents satisfont la descente relativement aux G-recouvrements. En appliquant ceci sur l'espace $Y$, on voit qu'on peut remplacer ce dernier
par $\coprod Y_i$ ; et on conclut alors par descente pour la G-topologie sur $X$.
\end{proof}

Nous nous proposons maintenant de montrer l'existence d'un
«idéal des coefficients» associé à un sous-espace analytique fermé
de la source d'un morphisme quasi-lisse et compact
à fibres géométriquement irréductibles (théorème \ref{theo-id-coeff}). 
Ce résultat jouera un rôle crucial dans notre construction d'éclatements à effet
aplatissant ; en effet, les centres de ces éclatements seront tous définis par des idéaux
de coefficients bien choisis (selon le procédé général décrit à l'étape \ref{platif-eclat-premier}
de la preuve du théorème principal). 

Ce théorème sur l'existence d'un idéal des coefficients est l'analogue 
du théorème 4.1.1 de \cite{raynaud-g1971}, ou plutôt d'un 
cas particulier de celui-ci (celui où $X$
est lisse à fibres géométriquement irréductibles sur $S$
et où $\mathscr M=\mathscr O_X$). Lorsqu'on 
suit la preuve de ce théorème de Raynaud et Gruson en 
simplifiant ce qui peut l'être dans le cas particulier évoqué, on voit qu'elle repose \emph{in fine}
sur le résultat suivant, aussi joli que surprenant :
\emph{si $A\to B$ est un morphisme lisse
entre anneaux tel
que les fibres de $\spec B\to \spec A$ soient toutes géométriquement irréductibles de même dimension $n$, 
alors $B$ est libre comme $A$-module
(évidemment de rang
$\aleph_0$ si $A\neq \{0\}$ et $n>0$)}.
Ce résultat s'obtient 
en combinant 
la prop. 3.3.1,  le th. 3.3.5  et le cor. 3.3.11 de
\cite{raynaud-g1971}, 
du moins dans le cas $n>0$ ; dans le cas $n=0$,
on peut remarquer directement que $A\to B$ est un isomorphisme. 

Or nous ignorons totalement si une variante banachique de cet énoncé
vaut pour les morphismes quasi-lisses entre algèbres $k$-affinoïdes, et n'avons pas
de piste pour en démontrer éventuellement
une. 

Toutefois, il y a un exemple évident de morphisme quasi-lisse $A\to B$ entre algèbres $k$-affinoïde tel que
les fibres de $\mathscr M(B)\to \mathscr M(A)$ soient toutes géométriquement irréductibles de même dimension, et tel que $B$ soit
somme directe banachique d'une famille dénombrables de $A$-modules libres de rang $1$ : c'est celui où $B=\mathscr M(A\{T/r\})$ pour un certain polyrayon $r$. Et cela nous sera suffisant pour suivre plus ou moins la preuve par \cite{raynaud-g1971} de l'existence d'un idéal de coefficients, grâce au lemme \ref{ql-section-disque}
ci-dessous. 

\begin{enonce}[remark]{Notation}
Pour tout espace $k$-analytique $X$ et tout polyrayon
$r$ on notera $\mathbf D_X(r)$ le polydisque relatif fermé de rayon $r$ sur $X$, 
c'est-à-dire l'espace $k$-analytique
$X\times_k \mathscr M(k\{T/r\})$. 
\end{enonce}

\begin{lemm}\label{ql-section-disque}
Soit $Y\to X$ un morphisme quasi-lisse et séparé
entre espaces $k$-analytiques, et soit $\sigma$
une section de $Y\to X$. Il existe un G-recouvrement affinoïde $(X_i)$ de $X$,
et pour tout $i$ un voisinage affinoïde $V_i$ de $\sigma(X_i)$ dans $Y\times_X X_i$,
un polyrayon $r_i$ 
et un $X_i$-isomorphisme
$V_i\simeq \mathbf D_{X_i}(r_i)$
modulo
lequel $\sigma$ est la section nulle. 
\end{lemm}

\begin{proof}
On peut raisonner G-localement sur $X$, ce qui autorise à le supposer affinoïde. 
Soit $x$ un point de $X$ et soit $n$ la dimension de $Y\to X$
en $\sigma(x)$. Le point $\sigma(x)$ appartient à l'intérieur relatif
de $Y$ sur $X$ (\cite{berkovich1994}, déf. 1.5.4 et prop. 1.5.5 (i)),
ce qui entraîne que le germe $(Y,\sigma(x))$ est bon. Il existe donc un voisinage
affinoïde $U$
de $\sigma(x)$ dans $Y$
tel que $U\to X$ soit purement de dimension $n$ et tel que $\Omega_{U/X}$
soit libre de rang $n$, avec une base de la forme $(\mathrm{d} f_1,\ldots, \mathrm{d}f_n)$ où les
$f_i$ sont des fonctions analytiques sur $U$. 
Posons $X'=\sigma\inv(U)$. Quitte à remplacer $X$ par $X'$
(ce qui est licite puisqu'on peut raisonner localement sur $X$)
et $U$ par $U\times_X X'$, on peut supposer que $\sigma(X)$ est contenu dans $U$. 
Comme $\sigma(X)$ est par ailleurs contenu dans l'intérieur relatif de $Y$ sur $X$, il est contenu dans l'intérieur
relatif de $U$ dans $Y$, c'est-à-dire dans l'intérieur topologique de $U$ dans $Y$
(\cite{berkovich1994}, prop. 1.5.5 (ii)). 

On peut dès lors remplacer $Y$ par $U$,
c'est-à-dire se ramener au cas où $Y$ est affinoïde, où $Y\to X$ est purement de dimension $n$, et où
$\Omega_{Y/X}$ est libre de rang $n$, avec une base de la forme $(\mathrm{d} f_1,\ldots, \mathrm{d}f_n)$ où les
$f_i$ sont des fonctions analytiques sur $Y$. Pour tout $i$, posons $g_i=\sigma^*f_i$. C'est une fonction sur $X$, 
et l'on note encore $g_i$ son image dans $\mathscr O_Y(Y)$. 
Quitte à remplacer chacune des $f_i$ par $f_i-g_i$, on peut alors supposer que $f_i|_{\sigma(X)}=0$ 
pour tout $i$. 
Soit $f\colon Y\to \mathbf A^n_X$ le morphisme induit par les $f_i$. Il est quasi-étale
(\cite{ducros2018}, Lemma 5.4.5), et la composée $\tau :=f\circ \sigma$ est la section nulle
de $\mathbf A^n_X\to X$. Comme $\sigma(X)$ est contenu dans l'intérieur relatif de $Y$
sur $X$, il est contenu dans l'intérieur du morphisme $f$, qui est donc étale au voisinage de $\sigma(X)$. 
Puisque $f$ induit un isomorphisme de $\sigma(X)$ sur $\tau(X)$, il résulte de la proposition 
4.3.4 de \cite{berkovich1993}
que $f$ induit un isomorphisme d'un voisinage ouvert $W$ de $\sigma(X)$ dans $Y$ sur un voisinage 
ouvert $\Omega$ de $\tau(X)$ dans $\mathbf A^n_X$. Puisque $\tau$ est la section nulle, il existe un polyrayon 
$r$ de longueur $n$ tel que $\mathbf D_X(r)\subset \Omega$. Soit $V$ l'image réciproque
de $\mathbf D_X(r)$ par l'isomorphisme $W\simeq \Omega$. C'est par construction un voisinage affinoïde de $\sigma(X)$
muni d'un $X$-isomorphisme $V\simeq \mathbf D_X(r)$ qui identifie $\sigma$ à la section nulle. 
\end{proof}

Énonçons dès maintenant
un corollaire technique du lemme précédent, dont
nous aurons besoin plus tard. 

\begin{coro}\label{coro-tube-transverse}
Soit $Y\to X$ un morphisme quasi-lisse entre espaces $k$-analytiques compacts et soit $Z$
un fermé de Zariski de $Y$ tel que $Z_x$ soit d'intérieur vide dans $Y_x$ pour tout $x$. Il existe alors un domaine
analytique compact $V$ de $Y$ rencontrant toutes les fibres de $Y\to X$ et disjoint de $Z$.
\end{coro}

\begin{proof}
Comme une union finie de  domaines analytiques compacts est un domaine analytique
compact et comme le fait d'être d'intérieur vide est, pour un fermé de Zariski,
préservé par la restriction à un domaine analytique, on peut raisonner G-localement sur $Y$,
et partant le supposer affinoïde. Comme $Y\to X$ est quasi-lisse il est plat,
et l'image du morphisme $Y\to X$ est
par conséquent
un domaine analytique compact de $X$ ; quitte à remplacer $X$ par ce dernier,
on peut supposer que $Y\to X$ est surjective. Soit $(f_1,\ldots, f_m)$
une famille de fonctions analytiques sur $Y$ engendrant
l'idéal qui décrit le fermé $Z$ ; posons $\Phi=\max(\abs{f_1},\ldots, \abs{f_m})$.
C'est une fonction continue de $Y$ dans $\mathbf R_{\geq 0}$ dont le lieu d'annulation est $Z$.
Pour tout $x\in X$, notons $\Psi(x)$ le maximum de $\Phi$ sur la fibre $Y_x$. Le but de ce qui suit
est de montrer que $\Psi$ est minorée sur $X$ par un réel $\epsilon>0$ ;
on pourra alors prendre pour $V$ le domaine défini par l'inégalité $\Phi\geq \epsilon/2$. 

La minoration qu'on cherche à établir sur $\Psi$ peut se montrer après
n'importe quel changement de base surjectif $X'\to X$, ce qui autorise
à supposer tout d'abord que le corps de définition de $X$ n'est pas trivialement valué et que $Y$ et $X$ sont stricts, puis,
en vertu du lemme précédent,
que $X$ est affinoïde et que
$Y\to X$ possède une section $\sigma$
telle qu'il existe un voisinage analytique $W$ de $\sigma(X)$ dans $Y$ et un $X$-isomorphisme $W\simeq \mathbf D_X(r)$ pour un certain polyrayon $r=(r_1,\ldots, r_n)$ (isomorphisme
qui identifie $\sigma$ à la section nulle). 

Pour tout $x$, notons $\mathfrak s(x)$ le point de Shilov de
$W_x\simeq \mathbf D_{\hr x}(r)$ (c'est l'unique point en lequel toute fonction analytique
sur $\mathbf D_{\hr x}(r)$ atteint son maximum en norme ; il correspond
à la norme de l'algèbre $\hr x\{T/r\}$, qui est multiplicative).
Le fermé $Z_x$ est d'intérieur vide dans $Y_x$, et $Z\cap W_x$ est donc un fermé de Zariski strict de $\mathbf D_{\hr x}(r)$ ;  en particulier, il ne contient pas $\mathfrak s(x)$. 
L'application $\mathfrak s$
est continue (\cite{berkovich1990}, lemme
3.3.2 (i)).
La composée $\Phi\circ \mathfrak s$ est encore continue, et partout strictement positive par ce qui précède. Cette fonction est donc minorée sur $X$ par un réel $\epsilon>0$ ; il en va
\emph{a fortiori}
de même pour $\Psi$.
\end{proof}

\begin{theo}[Idéal des coefficients]\label{theo-id-coeff}
Soit $Y\to X$ un morphisme quasi-lisse et compact entre espaces $k$-analytiques, à fibres
géométriquement irréductibles. Soit $Z$ un sous-espace analytique 
fermé de $Y$. La catégorie des espaces analytiques $T$ définis sur une extension complète de $k$ et munis d'un $k$-morphisme
$T\to X$ tel que l'immersion fermée
$Z\times_X T\hookrightarrow Y\times_X T$ soit un isomorphisme
admet un objet final $S$. L'espace $S$ est $k$-analytique et $S\to X$ est une immersion fermée. La formation de $S$ commute aux changements de base. 

Si $\mathscr I$ désigne le faisceau d'idéaux définissant $Z$, le faisceau d'idéaux $\mathscr J$ qui définit $S$ est appelé \emph{l'idéal des coefficients}
de $\mathscr I$ ; si $\mathscr J$ est inversible alors le sous-faisceau cohérent 
$(\mathscr J\mathscr O_Y)\inv \mathscr I$ de $\mathscr K_Y$ est contenu dans
$\mathscr O_Y$, et son idéal des coefficients est $\mathscr O_X$. 
\end{theo}

\begin{proof}
Commençons par une remarque.
Soit $S\hookrightarrow X$ une immersion fermée, donnée par un faisceau cohérent d'idéaux $\mathscr I$
sur $X$. Soient $T$ un espace analytique défini sur une extension 
complète de $k$, et soit $X'$ un espace $k$-analytique. Soit
$T\to X$ un $k$-morphisme
et soit $X'\to X$ une surjection plate et compacte. Le morphisme $T\to X$ se factorise alors par $S$ si et seulement
si $T'\to X'$ se factorise par $S'$, en notant $T'$ et $S'$ les espaces déduits de $T$ et $S$ par le changement de base $X'\to X$. En effet, 
$T\to X$ se factorise par $S$ si et seulement si l'image réciproque de $\mathscr I$ dans $\mathscr O_T$ est nulle, et $T'\to X'$ se factorise
par $S'$ si et seulement si l'image réciproque de $\mathscr I$ dans $\mathscr O_{T'}$ est nulle, et on conclut en remarquant que 
par descente fidèlement plate pour les sections globales des faisceaux cohérents
(\ref{theo-desc-coh} (1)) $\mathscr O_T\to p_*\mathscr O_{T'}$ est injective, où $p$ désigne le morphisme plat, compact
et surjectif $T'\to T$. 

Combiné à la descente fidèlement plate pour les faisceaux cohérents eux-mêmes (th. \ref{theo-desc-coh} (2)), cette remarque nous montre qu'il suffit
d'établir le théorème 
après un changement de base compact, plat et surjectif. Comme il est par ailleurs de nature G-locale, on peut procéder aux réductions
successives suivantes. 

\begin{itemize}[label=$\bullet$]
\item On peut supposer que $X$ est compact ; l'espace $Y$ est alors compact. 
\item Quitte à remplacer $X$ par $X_r$ pour un polyrayon $r$ convenable,  
on peut supposer que
$k$ n'est pas trivialement valué et que 
$X$ et $Y$ sont stricts. 
\item Comme $Y\to X$ est surjectif (ses fibres sont intègres et en particulier non vides), on 
déduit du th. 9.1.3 de \cite{ducros2018} qu'il existe un espace $k$-analytique compact $X'$ et un morphisme
quasi-étale $X'\to X$ tel que $Y\times_X{X'}\to X'$ admette une section. Quitte à remplacer $X$ par $X'$, 
on peut supposer que $Y\to X$ admet une section $\sigma$. 
\item Enfin en utilisant le lemme \ref{ql-section-disque}, on se ramène au cas où $X$ est affinoïde et
où $\sigma(X)$ possède un voisinage
affinoïde $V$ dans $Y$ tel qu'il existe un $X$-isomorphisme entre $V$ et un polydisque relatif sur $X$
identifiant $\sigma(X)$ à la section nulle. Notons $n$ la dimension relative de ce polydisque. Comme
les fibres de $Y\to X$ sont géométriquement intègres, elles sont purement de dimension $n$. 

\end{itemize} 
Soit $A$ l'algèbre des fonctions analytiques sur $X$ ; l'algèbre des fonctions analytiques sur $V$ est
isomorphe à $A\{r\inv \tau\}$ pour un certain polyrayon $r$. Le sous-espace analytique fermé $Z\cap V$ de $V$
est alors défini par une famille finie $(\sum_I a_{I,j}\tau^I)_j$ de fonctions. Notons $S$ le sous-espace analytique fermé
de $X$ défini par l'idéal $(a_{I,j})_{I,j}$ ; nous allons montrer qu'il répond à la question. 

Soit $T\to X$ un $k$-morphisme d'espaces
analytiques, où $T$ est défini sur une extension complète de $k$ ; nous désignons désormais par $a_{I,j}$ l'image réciproque de $a_{I,j}$ sur $T$. Il s'agit de démontrer que
$Z\times_X T\simeq Y\times_X T$ si et seulement si tous les $a_{I,j}$ sont nuls. L'assertion étant G-locale sur $T$, on peut le supposer
affinoïde ; quitte à remplacer $X$ par $T$, on peut alors
supposer que $T=X$. 

Supposons que $Z=Y$. Dans ce cas $Z\cap V=V$, et l'idéal engendré par les $\sum_I a_{I,j}\tau^I$ pour $j$ variable est alors nul.
Cela signifie que 
$\sum_I a_{I,j}\tau^I=0$ quel que soit $j$, c'est-à-dire encore
que $a_{I,j}=0$ quel que soit $(I,j)$. 

Réciproquement, supposons que tous les coefficients
$a_{I,j}$ soient nuls. On a alors $Z\cap V=V$. Pour montrer que $Z=Y$, il suffit en vertu
du lemme \ref{fonction-annulation}
de vérifier que $V$ rencontre toutes les composantes assassines de $Y$. Soit $Y'$ une telle composante. D'après la proposition \ref{prop-compass-morph}
et la remarque \ref{rema-compass-morph} (2), il existe une composante assassine $X'$ de $X$ telle que $Y'$ soit une composante
irréductible de $Y\times_X X'$. Soit $Y''$ l'ouvert de $Y\times_X X'$ formé des points de $Y'$ qui ne sont situés sur aucune autre composante de $Y\times_X X'$. Soit $y\in Y''$ et soit $x$
son image sur $X'$. La fibre $Y''_x$ est un ouvert non vide de l'espace irréductible $Y_x$, et $Y''_x\subset Y'_x$ ; par conséquent $Y'_x=Y_x$. En particulier, $Y'$ rencontre $\sigma(x)\in V$.

Montrons maintenant la dernière assertion. On suppose donc que $\mathscr J$ est inversible. En raisonnant G-localement, on peut se ramener au cas où il existe
un élément
régulier $f$ de $A$ tel que $(a_{I,j})_{I,j}=(f)$. Pour tout $(I,j)$, notons $b_{I,j}$ l'unique élément de $A$ tel que $a_{I,j}=fb_{I,j}$. 
Le sous-$\mathscr O_Y$-module $(\mathscr J\mathscr O_Y)\inv \mathscr I$
de $\mathscr K_Y$ est alors engendré par $(\sum_I b_{I,j}\tau^I)_j$ et est en particulier contenu dans $\mathscr
O_Y$. Et son idéal des coefficients
est engendré par les $b_{I,j}$, et est donc égal à $\mathscr O_X$ puisque $(b_{I,j})_{I,j}=A$ par construction. 
\end{proof}

Ce théorème sur l'existence d'un idéal de coefficients s'applique à un morphisme
quasi-lisse \emph{à condition que ses fibres soient géométriquement irréductibles}, c'est-à-dire
encore non vides et géométriquement connexes (puisqu'un espace quasi-lisse est normal). 
Nous allons maintenant décrire un procédé permettant de se ramener à cette situation. 

\begin{defi}
Soit $Y\to X$ un morphisme d'espaces $k$-analytiques. 
Un \emph{découpage
de $Y$ relativement à $X$}, ou encore un \emph{$X$-découpage de $Y$}, 
est un G-recouvrement de $Y$ par des domaines
analytiques compacts dont les fibres au-dessus de $X$ sont géométriquement connexes. 
Si $Y$ et $X$ sont $\Gamma$-stricts, nous dirons qu'un découpage est $\Gamma$-strict s'il
est constitué de domaines analytiques $\Gamma$-stricts. 
\end{defi}

\begin{theo}\label{exist-decoupages}
Soit $Y\to X$ un morphisme d'espaces $k$-analytiques
$\Gamma$-stricts. On suppose que les deux assertions suivantes sont satisfaites : 
\begin{enumerate}[1]
\item $Y\to X$
est plat à fibres géométriquement réduites ; 
\item $Y\to X$ est topologiquement propre.
\end{enumerate}
ll existe alors un recouvrement $(Z_i\to X)$ de $X$ pour la topologie quasi-étale tel que pour tout indice $i$, les deux propriétés suivantes
soient satisfaites : 

\begin{itemize}[label=$\diamond$]
\item l'espace $Z_i$ est affinoïde et $\Gamma$-strict ;
\item l'espace 
$Y\times_X Z_i$
possède un $Z_i$-découpage $\Gamma$-strict.
\end{itemize}
\end{theo}

\begin{proof} Nous allons procéder en plusieurs étapes. 

\subsubsection{Réduction au cas affinoïde}
L'énoncé à prouver est G-local sur $X$, ce qui permet de supposer que $X$
est affinoïde. Dans ce cas $Y$ est quasi-compact par propreté topologique, et il admet donc
un $G$-recouvrement fini par des domaines affinoïdes $\Gamma$-stricts. Il suffit de démontrer le théorème pour chacun de ces domaines ; on
suppose donc désormais que $Y$ et $X$ sont affinoïdes. 

\subsubsection{Insensibilité aux revêtements radiciels}\label{deviss-insensible-radiciel}
Si $X'\to X$ est un morphisme
fini et radiciel, et si le théorème vaut au-dessus de $X'$, il vaut sur $X$. Supposons
en effet qu'il existe une famille finie $(Z_i\to X')$ de morphismes quasi-étales de sources
affinoïdes, recouvrant $X'$, et tels que $Y\times_X Z_i$ admette un $Z_i$-découpage
$\Gamma$-strict pour tout $i$. Quitte à raffiner le recouvrement $(Z_i)$, on peut supposer
qu'il existe pour tout $i$ un domaine $\Gamma$-rationnel $V_i$ 
de $X'$ (\ref{ss-ger-grau})
et un revêtement fini étale $W_i$ de $V_i$
dont $Z_i$ est un domaine $\Gamma$-rationnel. 

L'image de $V_i$ sur $X$ est un domaine $\Gamma$-rationnel $U_i$ de $X$
(qui se décrit en élevant les inéquations décrivant $V_i$ à une puissance
convenable de l'exposant caractéristique). 
L'invariance topologique du site étale (\cite{berkovich1993}, th. 4.3.4) assure l'existence d'un revêtement fini
étale $T_i\to U_i$ tel que $W_i=V_i\times_{U_i}T_i$. L'image de $Z_i$ sur $T_i$ par le 
morphisme radiciel $W_i\to T_i$ est un domaine
$\Gamma$-rationnel $S_i$ de $T_i$ ; la famille
$(S_i\to X)$ est un recouvrement quasi-étale de $X$. 

Fixons $i$. L'espace affinoïde $Y\times_X Z_i$ admet un recouvrement fini $(Y_{ij})$ par des domaines analytiques
$\Gamma$-stricts compacts, tels que $Y_{ij}\to Z_i$ soit à fibres géométriquement connexes pour tout $j$. Pour tout $j$, l'image
$\Omega_{ij}$ de $Y_{ij}$ sur $Y\times_X S_{ij}$ est un domaine analytique compact et $\Gamma$-strict de $Y\times_X S_i$, et
les fibres de $\Omega_{ij}\to S_i$ sont géométriquement connexes,
d'où l'assertion. 

\subsubsection{}
Nous allons utiliser librement
le langage de l'algèbre graduée à la Temkin,
dans sa variante introduite à la section 1 de \cite{ducros2019},
ainsi que la notion de \emph{présentation $\Gamma$-sympathique}
d'un morphisme entre algèbres affinoïdes, \emph{cf. op. cit.}, définition 3.2. 

Le théorème 3.5 de \emph{op. cit.} assure qu'il existe une famille finie 
de morphismes $X_i\to X$, dont les images recouvrent $X$, telles que les propriétés suivantes soient satisfaites pour tout $i$ : 

\begin{itemize}[label=$\diamond$] 
\item l'espace $X_i$ est affinoïde et $\Gamma$-strict ; 
\item si $\abs{k\gpm}\neq \{1\}$ alors $X_i\to X$ est quasi-étale ; en général, il admet une décomposition $X_i\to X'_i\to X$ avec $X_i\to X'_i$ fini, radiciel et plat, et $X'_i\to X$ quasi-étale ; 
\item $Y\times_X X_i\to X_i$ admet une présentation $\Gamma^\Q$-sympathique.
\end{itemize}
Il découle alors de \ref{deviss-insensible-radiciel}, et du fait que le théorème à montrer est, par son énoncé même, local pour la topologie quasi-étale sur $X$, qu'il suffit de traiter le cas où $Y\to X$ admet
une présentation $\Gamma^\Q$-sympathique. 

\subsubsection{}
Nous utilisons dans ce qui suit les réductions \emph{$\Gamma^\Q$-graduées} à la Temkin, \emph{cf.} \cite{ducros2019}, 2.1.3.

L'existence d'une présentation $\Gamma^\Q$-sympathique de $Y\to X$ implique
par définition que le morphisme naturel $\widetilde Y\to \widetilde X$ induit par la flèche $Y\to X$ admet une factorisation $\widetilde Y\to \mathsf Y\to \widetilde X$ telles que les
propriétés suivantes (entre autres) soient satisfaites :
\begin{enumerate}[1]
\item $\widetilde Y$ est finie sur $\mathsf Y$ ; 
\item il existe un recouvrement ouvert fini
$(\mathsf U_i)$ de $\mathsf Y$ tel que pour tout $\xi\in \widetilde X$ et tout $i$, la fibre $\mathsf U_{i,\xi}$ soit géométriquement connexe et soit ou bien vide, ou bien une composante connexe de $\mathsf Y_\xi$ ; 
\item pour tout point $x\in X$ dont on note $\widetilde x$ l'image sur $\widetilde X$ et toute extension complète $L$ de $\hr x$, le morphisme naturel $\widetilde {Y_{x,L}}\to \mathsf Y_{\widetilde x,\widetilde L}$ est un isomorphisme.
\end{enumerate}

Pour tout $i$, désignons par $U_i$ l'image réciproque
de $\mathsf U_i$ sur $Y$ par la flèche composée $Y\to \widetilde Y\to \mathsf Y$. Chacun des $U_i$ est un domaine analytique compact et $\Gamma$-strict de $Y$, et les $U_i$ recouvrent $Y$. Il suffit de montrer que pour tout $i$,
les fibres de $U_i \to X$ sont géométriquement connexes. 

Fixons $i$, et donnons-nous un point $x\in X$. Pour toute extension complète $L$ de $\hr x$, l'espace $U_{i,x,L}$ s'identifie à l'image réciproque 
de l'ouvert $\mathsf U_{i,\widetilde x,\widetilde L}$ de $\widetilde {Y_{x,L}}= \mathsf Y_{\widetilde x,\widetilde L}$, et 
$\mathsf U_{i,\widetilde x,\widetilde L}$ est connexe car $\mathsf U_{i,\widetilde x}$ est géométriquement connexe par hypothèse. 

Il n'y a plus pour conclure qu'à expliquer pourquoi la connexité de $\mathsf U_{i,\widetilde x,\widetilde L}$ entraîne celle de $U_{i,x,L}$. 
Or d'une manière générale, si $Z=\mathscr M(C)$ est un espace affinoïde $\Gamma$-strict et si $\rho$ désigne l'application de réduction $Z\to \widetilde Z$ alors $\rho$ induit pour tout ouvert $\mathsf U$ de $\widetilde Z$ une bijection $\pi_0(\rho\inv(\mathsf U))\simeq \pi_0(\mathsf U)$. En effet il suffit en vertu de la surjectivité et
de l'anti-continuité de $\rho$ de le vérifier  pour $\mathsf U$ parcourant une une base d'ouverts, et donc par exemple pour
$\mathsf U$ principal, auquel cas $\rho\inv(\mathsf U)$ est un domaine affinoïde de réduction $\mathsf U$ ; on s'est donc ramené au cas $\mathsf U=\widetilde Z$. La connexité de $Z$ entraîne celle de $\widetilde Z$ parce que $\rho$ est anti-continue et surjective. 
Et comme tout isomorphisme $C\simeq C_1\times C_2$ induit un isomorphisme $\widetilde C\simeq \widetilde {C_1}\times \widetilde {C_2}$, la connexité de $\widetilde Z$ entraîne celle de $Z$.
\end{proof}

\begin{rema}
Soit $Y\to X$ un morphisme plat, topologiquement propre et topologiquement
séparé entre espaces $k$-analytiques $\Gamma$-stricts. Supposons que $Y$ admette
un découpage $\Gamma$-strict $(Y_i)$ sur $X$. Il possède alors un découpage 
$\Gamma$-strict $(Z_j)$ sur $X$ que pour tout $j$, chaque fibre
non vide de $Z_j\to X$ soit une composante
connexe de la fibre de $Y\to X$ qui la contient. 

Pour le voir, on peut raisonner G-localement sur $X$, et partant supposer $Y$ et $X$ compacts ; 
on peut alors également supposer que le découpage $(Y_i)$ est fini, et indexé par $\{1,\ldots r\}$ pour un certain $r$.
On construit récursivement
pour tout $i$ 
une suite $(Z_{ij})_{1\leq j\leq r}$ de domaines analytiques compacts
et $\Gamma$-stricts de $Y$ par le procédé suivant :

\begin{itemize}[label=$\diamond$] 
\item $Z_{i1}=Y_i$ ; 
\item $Z_{ij}=Z_{i,j-1}\cup \bigcup_{1\leq \ell \leq r} Y_\ell \cap f^{-1}(f(Y_\ell\cap Z_{i,j-1}))$
\end{itemize}
Pour tout $i$, on pose $Z_i=Z_{i,r}$. 

Soit $x$ un point de $X$ et soit $y
\in Y_x$.
Soit $i$ un entier compris entre $1$ et $r$. Notons $\mathscr S_i$ l'ensemble 
des entiers $s$ compris entre $1$ et $r$ tels 
qu'il existe une suite $i=a_1,a_2,\ldots, a_s$ d'entiers deux à deux distincts compris entre $1$ et $r$
sujette aux conditions suivante : 
$(Y_{a_\ell}\cap Y_{a_{\ell +1}})_x$ est non vide pour tout $\ell$ compris
entre $1$ et $s-1$, et $y\in Y_{a_s}$. Si $\mathscr S_i$ est non vide
on note $\delta(y,i)$ son plus petit élément ; si $\mathscr S_i$ est vide on pose
$\delta(y,i)=+\infty$.

Par construction, $\delta(y,i)=1$ si et seulement si $y\in Y_i$ ; et comme
les $Y_{i,x}$ sont des parties compactes et connexes
 de la fibre $Y_x$ qui recouvrent cette dernière, 
on voit que $\delta(y,i)<+\infty$ si et seulement si $Y_{i,x}\neq \emptyset$ et $y$ appartient 
à la composante connexe de $Y_x$ contenant $Y_{i,x}$.

Par ailleurs il résulte de la définition que $y$ appartient au
domaine analytique $Z_{ij}$ 
si et seulement $\delta(y,i)\leq j$. Il s'ensuit que pour tout $i$, la
fibre $Z_{i,x}$ est vide si $Y_{i,x}=\varnothing$, et égale à la composante
connexe de $Y_x$ contenant $Y_{i,x}$ sinon. 

Ainsi, $(Z_i)_{1\leq i\leq r}$ constitue un recouvrement de $Y$ par des
domaines analytiques compacts et $\Gamma$-stricts
telles que pour tout $i$, toute fibre
non vide de $Z_i\to X$ soit une composante connexe de la fibre de $Y\to X$
qui la contient. Et si $x$ est un point de $X$ tel que $Z_{i,x}$
soit non vide alors $Y_{i,x}$ est non vide, contenu dans $Z_{i,x}$,
et géométriquement connexe ; par conséquent $Z_{i,x}$ est géométriquement connexe. 
\end{rema}

\section{Rappels et compléments sur les dévissages}
\label{s-complement-devissages}

\begin{enonce}[remark]{Notation}
Pour tout morphisme $u\colon \mathscr F\to \mathscr G$ entre faisceaux cohérents sur un espace
$k$-analytique $X$, on notera $\mathrm{Bij}(u)$ l'ensemble des points de $x$ en lesquels $u$
est bijectif ; c'est un ouvert de Zariski de $X$. 
\end{enonce}

Nous allons maintenant
donner la définition de dévissage en un point. Elle est inspirée
de la notion analogue dans le contexte schématique, introduite par 
Raynaud et Gruson dans \cite{raynaud-g1971}.

\begin{defi}\label{defi-deviss}
Soit $Y\to X$ un
morphisme d'espaces $k$-analytiques, soit $y$ un point de $Y$
et soit $x$ son image sur $X$. Soit $\mathscr F$ un faisceau cohérent sur $Y$ dont le support contient $y$. 
Soit $r$ un entier strictement positif et soit $n_1>n_2>\ldots >n_r$ une suite strictement décroissante
d'entiers. Un \emph{$X$-dévissage
de $\mathscr F$ en $y$ en dimensions $n_1>n_2>\ldots n_r$} est une liste de données
$(V,\{T_i,\pi_i,u_i,t_i,\mathscr L_i,\mathscr P_i\}_{i\in \{1,\ldots, r\}})$ où : 

\begin{itemize}[label=$\bullet$] 
\item $V$ est un voisinage affinoïde de $y$ dans $Y$ ; 
\item $T_i$ est pour tout $i$ un domaine
affinoïde d'un espace $k$-analytique $T^\dagger_i$
muni d'un morphisme $T_i^\dagger\to X$
lisse
et purement de dimension relative $n_i$, et $t_i$ est un point de $T_i$ situé au-dessus de $x$ ; 
\item pour tout $i$, $\mathscr L_i$ et $\mathscr P_i$ sont des $\mathscr O_{T_i}$-modules
cohérents, $\mathscr L_i$ est libre, et $\mathscr P_i$ est non nul dès que $i\leq r-1$ ; 

\item $t_i\in \mathrm{Supp}(\mathscr P_i)$ pour tout $i<r$, et $t_r\in \mathrm{Supp}(\mathscr P_r)$
si $\mathscr P_r\neq 0$ ; 
\item $\pi_1$ est un $X$-morphisme fini de $\mathrm{Supp}(\mathscr F_V)$ vers $T_1$
tel que $\pi_1\inv(t_1)=\{y\}$ \emph{ensemblistement} ; 
\item $\pi_i$ est pour tout $i$ compris entre $2$ et $r$ un morphisme fini de
$\mathrm{Supp}(\mathscr P_{i-1})$ vers $T_i$ tel que $\pi_i\inv(t_i)=\{t_{i-1}\}$
\emph{ensemblistement} ; 
\item $u_1$ est un morphisme de $\mathscr L_1$ vers $\pi_{1*}\mathscr F_V$ 
de conoyau $\mathscr P_1$ et tel que $t_1$ appartienne à $\adht {\mathrm{Bij}(u_1|_{T_{1,x}})}{T_{1,x}}$ ; 
\item pour tout $i$ compris entre $2$ et $r$, $u_i$ est un morphisme de $\mathscr L_i$
vers $\pi_{1*}\mathscr P_{i-1}$ 
de conoyau $\mathscr P_i$ et tel que $t_i$
appartienne à $\adht {\mathrm{Bij}(u_i|_{T_{i,x}})}{T_{i,x}}$.
\end{itemize}

On dit que ce dévissage est \emph{$\Gamma$-strict}
si $V$ et tous les $T_i$ sont $\Gamma$-stricts, et qu'il est \emph{total}
si $\mathscr P_r=0$. 
\end{defi}

\begin{rema}
Notre définition diffère légèrement de celle de \cite{ducros2018}, déf. 8.2.4 : 
ce que l'on appelait simplement dévissage dans \emph{loc. cit.}
est ce que nous appelons ici dévissage \emph{total}. 
\end{rema}

\subsection{Existence de dévissages}\label{ss-exist-deviss}
Si $Y$ et $X$ sont bons
et si $Y$ est $\Gamma$-strict,
le théorème 8.2.5 de \cite{ducros2018}
assure l'existence d'un $X$-dévissage $\Gamma$-strict 
total de
$\mathscr F$ en $y$ en dimensions comprises entre
$\dim_y (\mathrm{Supp}(\mathscr F))_x$ et $\mathrm{coprof}_y \mathscr F_{Y_x}$.

\subsection{Un critère de platitude}\label{ss-deviss-critplat}
L'intérêt des dévissages est de fournir un critère de platitude qui dans certaines circonstances se révèle plus
applicable que la définition directe. 
Ainsi, soient $X, Y$, $x$, $y$ et $\mathscr F$
comme dans la définition \ref{defi-deviss}, et soit
$(V, \{T_i,\pi_i,u_i,t_i,\mathscr L_i,\mathscr P_i\}_{i\in \{1,\ldots, r\}})$
un $X$-dévissage de $\mathscr F$ en $y$. Les assertions suivantes 
sont  équivalentes par une application répétée de la proposition 8.1.11 (1) de \cite{ducros2018} :

\begin{enumerate}[i]
\item $\mathscr F$ est $X$-plat en $y$ ; 
\item le faisceau cohérent $\mathscr P_r$ est $X$-plat en $t_r$ et $u_i$ est injectif en $t_i$
pour tout $i$. 
\end{enumerate}

\begin{rema}
Notons que si le dévissage est total, $\mathscr P_r$ est nul et en particulier automatiquement $X$-plat en $t_r$ : dans ce cas
$\mathscr F$ est donc $X$-plat en $x$ si et seulement si $u_i$ est injectif en $t_i$ pour tout $i$, 
énoncé qui constitue le th. 8.4.3 (2) de \cite{ducros2018}.

\end{rema}

Nous allons avoir besoin d'une variante légèrement plus globale de la notion de dévissage.

\begin{defi}
Soit $Y\to X$ un morphisme d'espaces
$k$-analytiques et soit $E$ 
un sous-ensemble de $X$. 
Soit $\mathscr F$ un faisceau cohérent sur $Y$. 
Soit $r$ un entier strictement positif et soit $n_1>n_2>\ldots >n_r$ une suite strictement décroissante
d'entiers $\geq 0$. Un \emph{$X$-dévissage de $\mathscr F$ au-dessus de $E$
en dimensions $n_1>n_2>\ldots >n_r$} est une liste de données 
$\{T_i,\pi_i,u_i, \mathscr L_i,\mathscr P_i\}_{i\in \{1,\ldots, r\}}$ où: 

\begin{itemize}[label=$\bullet$] 

\item $T_i$ est pour tout $i$ un espace $k$-analytique
muni d'un morphisme $T_i\to X$
quasi-lisse
et purement de dimension relative $n_i$ ; 
\item pour tout $i$, $\mathscr L_i$ et $\mathscr P_i$ sont des $\mathscr O_{T_i}$-modules
cohérents, et $\mathscr L_i$ est libre ; 
\item $\pi_1$ est un $X$-morphisme fini de $\mathrm{Supp}(\mathscr F)$ vers $T_1$
; 
\item $\pi_i$ est pour tout $i$ compris entre $2$ et $r$ un morphisme fini de
$\mathrm{Supp}(\mathscr P_{i-1})$ vers $T_i$ ; 
\item $u_1$ est un morphisme de $\mathscr L_1$ vers $\pi_{1*}\mathscr F$ 
de conoyau $\mathscr P_1$ tel que $\mathrm{Bij}(u_1|_{T_{1,x}})$ 
soit dense dans $T_{1,x}$ pour tout $x\in E$ ; 
\item pour tout $i$ compris entre $2$ et $r$, $u_i$ est un morphisme de $\mathscr L_i$ vers $\pi_{1*}\mathscr P_{i-1}$ 
de conoyau $\mathscr P_i$ tel que $\mathrm{Bij}(u_i|_{T_{i,x}})$ 
soit dense dans $T_{i,x}$ pour tout $x\in E$ ; 
\end{itemize}

Ce dévissage est dit \emph{$\Gamma$-strict}
si $V$ et tous les $T_i$ sont $\Gamma$-stricts ; 
il est dit \emph{total}
s'il existe $i$ tel que $\mathscr P_i=0$ (auquel cas $\mathscr P_j=0$ pour tout $j>i$). 
\end{defi}

\begin{rema}
Conservons les notations de la définition ci-dessus. 
Si $\{T_i,\pi_i,u_i, \mathscr L_i,\mathscr P_i\}_{i\in \{1,\ldots, r\}}$  est un 
dévissage de $\mathscr F$ au-dessus de $E$ alors pour tout $x\in E$
et tout indice $i$, la fibre $T_{i,x}$ est purement de dimension $n_i$, 
et le lieu de bijectivité de $u_i|_{T_{i,x}}$ est dense dans $T_{i,x}$ ; 
le conoyau de $u_i|_{T_{i,x}}$ étant égal à
$\mathscr P_{i,T_{i,x}}$, ce dernier a un support
d'intérieur vide dans $T_{i,x}$, et en particulier de dimension $<n_i$. 
\end{rema}

\begin{lemm}\label{exist-deviss-glob}
Soit $Y\to X$ un morphisme entre bons espaces
$k$-analytiques
; on suppose de plus que $Y$ est $\Gamma$-strict. 
Soit $y$ un point de $Y$
et soit $x$ son image sur $X$. Soit $\mathscr F$ un faisceau cohérent sur $Y$ dont le support contient $y$. 
Il existe alors un voisinage affinoïde
$\Gamma$-strict de $y$ dans $Y$ en restriction auquel
$\mathscr F$ admet un $X$-dévissage
$\Gamma$-strict total au-dessus de $x$
en dimensions comprises entre $\dim_y (\mathrm{Supp}(\mathscr F))_x$ et $\mathrm{coprof}_y \mathscr F_{Y_x}$, 
qui ne fait intervenir que des espaces affinoïdes. 
\end{lemm}

\begin{proof}
On commence par choisir un $X$-dévissage $\Gamma$-strict
et total $(V, \{T_i,\pi_i, u_i, t_i,\mathscr L_i,\mathscr P_i\}_{i\in \{1,\ldots, r\}})$
de $\mathscr F$ en $y$ en dimensions appartenant à la plage requise
(\ref{ss-exist-deviss}). 
Nous allons montrer par récurrence sur $r$
qu'il existe un voisinage $\Gamma$-rationnel (\ref{ss-ger-grau})
$V'$ de $y$ dans $V$ qu'on peut choisir arbitrairement petit, et un 
$X$-dévissage $\Gamma$-strict total $(V', \{T'_i,\pi'_i,u'_i, \mathscr L'_i,\mathscr P'_i\}_{i\in \{1,\ldots, r\}})$ de $\mathscr F_{V'}$ en $y$ où chacun des $T'_i$ est un voisinage rationnel
$\Gamma$-strict de $t_i$ dans $T_i$ et où les autres données sont induites par celles du dévissage initial. 

On suppose donc l'assertion vraie en rang $<r$ (condition qui est vide si $r=1$ puisqu'il n'y a pas de dévissage de longueur nulle). 
Fixons un voisinage
$\Gamma$-rationnel arbitraire
arbitraire $\Omega$ de $y$ dans $V$. Par propreté topologique des morphismes finis, il existe un voisinage
$\Gamma$-rationnel 
$T'_1$ de $t_1$ dans $T_1$ qui possède les trois propriétés suivantes : 
\begin{enumerate}[1]
\item $T'_{1,x}$ est contenue dans la composante connexe de $t_1$ dans $T_{1,x}$ ; 
\item $\pi_1\inv(T'_1)$ est contenu dans $\Omega$ ; 
\item Si $r\geq 2$, la restriction de $\mathscr F$ à $T'_1$ possède un $X$-dévissage
total au-dessus de $x$ du type fourni par l'hypothèse de récurrence.
\end{enumerate}
Par construction, $\pi_1\inv(T'_1)$ est un domaine
$\Gamma$-rationnel
de $\mathrm{Supp}(\mathscr F_V)$ contenu dans $\Omega$. 
En relevant un système d'inégalités convenables définissant $\pi_1\inv(T'_1)$
(comme dans la preuve
du lemme
\ref{lem-gerr-grau})
on obtient un voisinage
$\Gamma$-rationnel
$W$ de
$y$ dans $V$ tel 
que $W\cap \mathrm{Supp}(\mathscr F_V)=\pi_1\inv(T'_1)$. Soit $V'$ l'intersection $W\cap \Omega$ ; c'est encore un voisinage $\Gamma$-rationnel
de $y$
dans $V$, et l'on a $V'\cap \mathrm{Supp}(\mathscr F_V)=\pi_1\inv(T'_1)$. On obtient alors un dévissage satisfaisant les conditions requises
en prenant la famille $(V',T'_1)$ (et les objets induits par $\pi_1$, $u_1$, $\mathscr L_1$ et $\mathscr P_1$...), et en la concaténant avec le dévissage évoqué
en (3)
si $r>1$.
\end{proof}

\begin{enonce}[remark]{Notation}
Soit $Y\to X$ un morphisme entre espaces $k$-analytiques et soit 
$\mathscr F$ un faisceau cohérent sur $Y$. Nous noterons
$\pl {\mathscr F}X$ l'ensemble des points de $Y$
en lesquels $\mathscr F$ est $X$-plat. C'est un ouvert de Zariski de $Y$ dont le fermé
complémentaire sera noté $\ql {\mathscr F}X$. 
\end{enonce}

\begin{defi}
Soit $Y\to X$ un morphisme entre espaces $k$-analytiques
et soit $u\colon \mathscr F\to \mathscr G$ un morphisme 
entre faisceaux cohérents sur $Y$. Soit $y$ un point de $Y$.
On dit que $u$
est \emph{universellement $X$-injectif} en $y$
si pour tout espace analytique $X'$ défini sur une extension 
complète de $k$, tout $k$-morphisme $X'\to X$
et tout  antécédent $y'$ de $Y$ sur $Y':=Y\times_X
X'$, le morphisme $\mathscr F_{Y'}\to \mathscr G_{Y'}$ induit
par $u
$ est injectif en $y'$. On dit que $u$ est $X$-universellement injectif s'il est 
$X$-universellement injectif en tout point de $Y$.
\end{defi}

\begin{prop}\label{deviss-dim-pasplat}
Soit $Y\to X$ un
morphisme entre espaces $k$-analytiques et soit $x$ 
un point de $X$.
Soit $\mathscr F$ un faisceau cohérent sur $Y$. 
Soit $r$ un entier strictement positif et soit 
$\mathscr D=\{T_i,\pi_i,u_i, \mathscr L_i,\mathscr P_i\}_{i\in \{1,\ldots, r\}}$ un $X$-dévissage
de $\mathscr F$ au-dessus de $x$, en dimensions $n_1>\ldots >n_r$. On suppose que $X$, $Y$ et les espaces $T_i$
sont tous compacts.

\begin{enumerate}[1]
\item Il existe un voisinage analytique compact $X'$ de $x$ dans $X$ tel que pour tout $i$ et tout $\xi \in X'$, l'intersection 
du lieu de surjectivité de $u_i$ avec $T_{i,\xi}$ soit dense dans $T_{i,\xi}$. 

\item Soit $m$ un entier tel que $\dim \ql{\mathscr P_r}X_x<m$ (c'est par exemple le cas
si $n_r\leq m$ ou si $\mathscr P_r=0$). Les assertions suivantes sont équivalentes :

\begin{enumerate}[j]
\item la dimension de $\ql {\mathscr F}X_x$ est $<m$ ; 
\item il existe un voisinage analytique compact $X'$ de $x$ dans $X$ tel que pour tout $i$ vérifiant l'inégalité
$n_i\geq m$, le morphisme $u_i|_{T_i\times_X X'}$ soit $X'$-universellement injectif ; 
\item pour tout $i$ vérifiant l'inégalité $n_i\geq m$, le morphisme $u_i$ est injectif en tout point de $T_{i,x}$. 
\end{enumerate}
\end{enumerate}
\end{prop}

\begin{proof}
Prouvons tout d'abord (1). Fixons $i$ et
appelons $Z$ le lieu des points de $T_i$ en lesquels $u_i$ n'est pas surjectif ; c'est un fermé
de Zariski de $T_i$. Par hypothèse, le lieu de bijectivité de $u_i|_{T_{i,x}}$ est dense dans $T_{i,x}$ ; il en va \emph{a fortiori}
de même de son lieu de surjectivité, qui
en vertu de \cite{ducros2018}, 2.5.4 (lequel repose
\emph {in fine} sur le
lemme de Nakayama)
est égal à la fibre en $x$ du lieu de
surjectivité de
$u_i$. Par conséquent, 
$Z_x$ est d'intérieur vide dans $T_{i,x}$ ; comme ce dernier est purement de dimension $n_i$, cela signifie que la dimension de $Z_x$ est $<n_i$. 
Soit $Z'$ le
sous-ensemble
de $Z$ constitué des points en lesquels la dimension relative de $Z$
sur $X$ est $\geq n_i$. C'est un fermé
de Zariski de $Z$ 
dont l'image 
$Z'$ sur $X$ est
par ce qui précède un compact ne contenant pas $x$ ; choisissons un voisinage analytique compact $X_i$ de $x$ dans $X$ qui ne rencontre pas cette image. 
Par construction, $Z\times_X X_i$ est de dimension relative partout $<n_i$ sur $X_i$, ce qui veut dire
que l'intersection du lieu de surjectivité de $u_i$ avec $T_{i,\xi}$ est dense dans $T_{i,\xi}$ pour tout $\xi \in X_i$.
Il suffit alors de poser $X'=\bigcap_i X_i$. 

Montrons maintenant (2).
Le fermé de Zariski $\ql{\pi_{1*}\mathscr F}X$ de $T_1$ est égal à l'image de $\ql {\mathscr F}X$ par
le morphisme fini $\pi_1$ (\cite{ducros2018}, Lemma 4.1.15 (3) ; notons que notre phrase a un sens car $\ql {\mathscr F}X$
est contenu dans $\mathrm{Supp}(\mathscr F)$).
On a donc $\dim \ql{\pi_{1,*}\mathscr F}X_x=\dim \ql {\mathscr F}X_x$.
Pour prouver (2), 
on raisonne par récurrence sur $r$, et l'on fait donc
l'hypothèse que (2) est vraie lorsque le dévissage en jeu est de longueur $<r$.
Supposons que (i) soit vraie et montrons (ii) ; si $n_1<m$ il n'y a rien à démontrer, et l'on suppose
donc à partir de maintenant que $n_1\geq m$.
La proposition 8.1.10 de \cite{ducros2018}
assure qu'en tout point de $\pl{\pi_{1*}\mathscr F}X$, 
le faisceau cohérent $\mathscr P_1$ est $X$-plat et le morphisme $u_1$
est injectif ; il vient $\ql{\mathscr P_1}X_x\subset \ql{\pi_{1,*}\mathscr F}X_x$
puis 
\[\dim\ql{\mathscr P_1}X_x\leq \dim \ql{\pi_{1,*}\mathscr F}X_x<m,\]
la seconde
inégalité provenant de l'hypothèse (i). 
L'hypothèse de récurrence assure donc qu'il existe un voisinage analytique compact
$X'$ de $x$ dans $X$ tel que pour tout $i$ vérifiant les inégalités
$i\geq 2$ et $r_i\geq m$, le morphisme $u_i|_{T_i\times_X X'}$ soit $X'$-universellement injectif. 
Comme $n_1\geq m$
et comme $\ql{\pi_{1,*}\mathscr F}X_x$ est de dimension $<m$, l'ouvert $\pl{\pi_{1*}\mathscr F}X_x$
de $T_{1,x}$, qui est contenu dans le lieu d'injectivité de $u_1$, est dense dans $T_{1,x}$ ; et comme par ailleurs le lieu de bijectivité de $u_1|_{T_{1,x}}$ est dense par définition d'un dévissage,  il en va \emph{a fortiori}
de même de son lieu de surjectivité ; par  \cite{ducros2018}, 2.5.4, ceci entraîne que le lieu de surjectivité de $u_1$ contient un ouvert dense de $T_{1,x}$. 
Par conséquent, $(\mathrm{Bij}(u_1))_x$ est dense dans $T_{1,x}$ . Soit $\Sigma$ le complémentaire de $\mathrm{Bij}(u_1)$ dans $T_1$. C'est un fermé de Zariski de $T_1$ ; soit $\Sigma'$ le lieu des points de $\Sigma$ en lequel ce dernier est de dimension relative $n_1$ sur $X$. C'est un fermé de Zariski de $\Sigma$ (\cite{ducros2007}, Th. 4.9)
qui ne rencontre pas $T_{1,x}$ puisque $\Sigma_x$ est d'intérieur vide dans $T_{1,x}$, qui est purement de dimension $n_1$ ; il existe donc un voisinage analytique compact $X''$ de $x$ dans $X$ ne rencontrant pas l'image de $\Sigma'$. Par construction, l'intersection de $\mathrm{Bij}(u_1)$ avec $T_{1,\xi}$ est dense dans $T_{1,\xi}$ pour tout $\xi \in X''$ ; il résulte alors de la proposition 8.1.8 de \cite{ducros2018} que $u_1|_{T_1\times_X X''}$ est universellement injectif. En conséquence, $u_i|_{T_i\times_X(X'\cap X'')}$ est universellement injectif pour tout $i$ tel que $r_i\geq m$, et (ii) est vraie. 

L'implication (ii)$\Rightarrow$(iii) est évidente (et ne fait pas appel à l'hypothèse de récurrence). 

Supposons que (iii) est vraie et montrons (i). Si $n_1<m$ 
on a alors évidemment $\dim \ql{\pi_{1*}\mathscr F}X_x<m$ et
partant
$\dim \ql{\mathscr F}X_x<m$. Supposons
maintenant que $n_1\geq m$. Le morphisme $u_1$ est alors injectif en tout point de $T_{1,x}$,
ce qui entraîne en vertu de la proposition 8.1.10 de \cite{ducros2018} que 
$\pl{\pi_{1*}\mathscr F}X_x=\pl{\mathscr P_1}X_x$, et partant
que $\ql{\pi_{1*}\mathscr F}X_x=\ql{\mathscr P_1}X_x$.

On distingue deux cas. Supposons tout d'abord que $r\geq 2$. Le dévissage $\mathscr D$ définit alors par restriction aux entiers compris entre $2$ et $r$ un $X$-dévissage du faisceau cohérent $\mathscr P_1$
au-dessus de $x$ en dimensions $n_2>\ldots >n_r$, qui vérifie (iii) ; en vertu de
notre hypothèse de récurrence, il vient $\dim \ql{\mathscr P_1}X_x<m$, c'est-à-dire
 $\dim \ql{\pi_{1*}\mathscr F}X_x<m$, ce qui entraîne que
que $\dim \ql{\mathscr F}X_x<m$. Supposons maintenant que $r=1$ ; on a alors $\dim \ql{\mathscr P_1}X_x<m$, et 
donc $\dim \ql{\pi_{1*}\mathscr F}X_x<m$, et par conséquent $\dim
\ql{\mathscr F}X_x<m$. 
\end{proof}

\begin{lemm}\label{lem-plat-libre}
Soit $Y\to X$ un morphisme quasi-lisse d'espaces
$k$-analytiques. 
Soit $\mathscr F$ un faisceau cohérent sur $Y$.
Soit $U$ l'ouvert de Zariski de $Y$ formé des points en lesquels
$\mathscr F$ est libre, soit $y$ un point 
de $Y$ en lequel
$\mathscr F$ est $X$-plat, et soit $x$ l'image de $y$
sur $X$.
Le point $y$ appartient
alors à $\adht{U_x}{Y_x}$. 
\end{lemm}
\begin{proof}
L'assertion peut se prouver en remplaçant $Y$ 
par n'importe quel domaine affinoïde de $Y$ contenant
$y$. On peut
donc supposer $X$ et $Y$ affinoïdes.
Soit $Z$ la composante connexe de $y$ dans l'espace
quasi-lisse $Y_x$ (remarquons que comme $Y_x$ est quasi-lisse sur $\hr x$ il est normal, et ses composantes connexes sont donc irréductibles) ; soit $\eta$ le point générique de la composante
connexe $Z\al$
de $Y_x\al$ et soit $\kappa(\eta)$ son 
corps résiduel. 
Choisissons une famille finie
d'éléments de $\mathscr F(Y)$
induisant une base
du $\kappa(\eta)$-espace
vectoriel $\mathscr F(Y)\otimes_{\mathscr O_Y(Y)}\kappa(\eta)$ ;
cette famille définit un morphisme $u$ de 
$\mathscr O_Y^r$ vers $\mathscr F$ ; soit $\Omega$ l'ouvert de Zariski
de $Y$ constitué des points en lesquels $u$ est un isomorphisme. Par construction, il existe un ouvert de Zariski non vide $T$ de $Z$ tel que $u|_T$ soit un isomorphisme. Le point
$y$ appartient à $\adht T{Y_x}$ ; puisque $\mathscr F$
est $X$-plat en $y$, il s'ensuit en vertu de la proposition 8.1.10 de \cite{ducros2018} que $u$ est injectif en $y$ ; il en résulte par la proposition 8.1.8 de \emph{op. cit.}
que $y\in \adht{\Omega_x}{Y_x}$. Comme $\Omega$ est contenu dans $U$ il vient $y\in \adht{U_x}{Y_x}$.
\end{proof}

\section{Les éclatements en géométrie analytique}
Nous utiliserons régulièrement dans ce qui suit les notions 
d'adhérence analytique et de densité analytique (définition \ref{def-densal}
et lemme-définition \ref{lem-adh-analytique}). 

\begin{defi}
Si $X$ est un espace
$k$-analytique, nous appellerons
\emph{diviseur de Cartier}
sur $X$ tout sous-espace analytique fermé de $X$ défini
par un faisceau cohérent d'idéaux inversible. 
\end{defi}

\begin{rema}
Nous attirons l'attention du lecteur sur le fait 
que dans ce texte les diviseurs de Cartier sont \emph{par définition}
effectifs.
\end{rema}

\subsection{}\label{ss-divcart-andense}
Soit $X$ un espace
$k$- analytique et soit $S$ un diviseur de Cartier de $X$. Son complémentaire $X\setminus S$
est analytiquement dense dans $X$ (et en particulier dense pour la topologie de Zariski), et $S$ est partout de codimension $1$
dans $X$: on peut en effet
vérifier cet énoncé G-localement, ce qui permet de supposer $X$ affinoïde ; et on se ramène alors par les théorèmes de type GAGA 
aux assertions analogues dans le contexte schématique.

\subsection{}
Soit $X$ un espace $k$-analytique, soit $\mathscr I$ un faisceau cohérent d'idéaux sur $X$ et soit $S$
le sous-espace analytique fermé de $X$ défini par $\mathscr I$. La catégorie des
espaces analytiques
$Y$ définis sur une extension complète de $k$ et munis d'un $k$-morphisme $f\colon Y\to X$ tel que le faisceau cohérent d'idéaux
$f\inv \mathscr I$ soit inversible, c'est-à-dire encore tels que $Y\times_X S$ soit un diviseur de Cartier de $Y$, admet un objet final, appelé \emph{éclatement de $X$ le long de $\mathscr I$}
(ou le long de $S$, ou de centre $S$), qui est un espace $k$-analytique. Il suffit en effet de le vérifier G-localement, ce qui permet de supposer que $X$ est affinoïde. Désignons par $\mathscr Y$ l'éclaté
$\mathrm{Proj}(\bigoplus (\mathscr I\al)^n)$ de $X\al$ le long de $\mathscr I\al$. On vérifie alors (exactement comme en géométrie algébrique)
que l'espace $X$-analytique $\mathscr Y\an$ possède la propriété universelle requise. 

L'éclatement de $X$ le long de $S$
sera noté $\ecl XS$ ; 
le produit fibré $\ecl XS\times_X S$ est un
diviseur de Cartier de $\ecl XS$, qu'on appelle
le \emph{diviseur exceptionnel}
de l'éclatement
et qu'on note
$\dex XS$. 

\subsection{}
Soit $X$ un espace $k$-analytique et soit $S$ un sous-espace analytique fermé de $X$, défini par un faisceau quasi-cohérent d'idéaux
$\mathscr I$.
Soit $f$ le morphisme canonique $\ecl XS$ vers $X$.
Les faits suivants découlent immédiatement de la construction de l'éclatement et/ou de sa propriété universelle. 

\subsubsection{}
Le morphisme $f$ est propre et induit un isomorphisme
\[\ecl XS\setminus\dex XS\simeq X\setminus S.
\]
De plus, $\ecl XS\setminus \dex XS$
est analytiquement dense dans $\ecl XS$. 

\subsubsection{}
Si $X$ est réduit, $\ecl XS$ est réduit. 

\subsubsection{}
Si $\mathscr J$ est un faisceau cohérent d'idéaux inversibles sur $X$, le faisceau cohérent d'idéaux $f\inv (\mathscr J)$ sur $\ecl XS$
est inversible. Autrement dit, $\ecl XS\times_X T$ est un diviseur de Cartier de $\ecl XS$ pour tout diviseur de Cartier $T$ de $X$. 

\subsubsection{}\label{sss-ecl-chgebase}
Soit $L$ une extension complète de $k$, soit $X'$
un espace $L$-analytique
et soit $X'\to X$ un $k$-morphisme ; posons
$S'=S\times_X X'$. 
Supposons que
\[ \dex XS\times_{\ecl XS}(\ecl XS \times_X X')
=(\ecl XS \times_X X')\times_{X'}S'
\]
soit un diviseur de Cartier de $\ecl XS\times_X X'$, 
ce qui est par exemple le cas en vertu de 
\ref{sss-plat-idinv} dès que le morphisme
$X'\to X_L$ est plat. 
On a alors un $X'$-isomorphisme
canonique entre 
$\ecl{X'}{S'}$ et $\ecl XS\times_X X'$
(il est en effet tautologique que ce dernier satisfait la propriété universelle
requise).

\subsubsection{}\label{sss-cartier-ecliso}
Si $S$ est un diviseur de Cartier, $f$ est un isomorphisme. 

\subsubsection{}\label{sss-centre-change}
Soit $\mathscr J$ un faisceau d'idéaux cohérent et inversible sur $X$, et soit $n$ un entier strictement
positif. 
L'éclatement de $X$ le long de $\mathscr I^n\cdot \mathscr J$
s'identifie
canoniquement à $\ecl XS$. 

\subsubsection{}\label{sss-eclatetout-vide}
Si le support de $S$ est égal (ensemblistement) à l'espace $X$ tout entier alors $\ecl XS=\emptyset$.

\subsection{}
Nous dirons qu'un morphisme $f\colon Y\to X$ entre espaces $k$-analytiques est un éclatement s'il existe un sous-espace analytique fermé $S$ de $X$ et un $X$-isomorphisme $Y\simeq \ecl XS$. 

Insistons sur le fait que si c'est le cas, le centre $S$ n'est pas uniquement déterminé, comme en attestent les exemples donnés en 
\ref{sss-cartier-ecliso}, \ref{sss-centre-change} 
et \ref{sss-eclatetout-vide} ; et le diviseur exceptionnel de l'éclatement
ne l'est pas davantage. 

On ne parlera
dès lors «du» centre ou «du» diviseur exceptionnel
d'un éclatement que si ce dernier a été présenté comme l'éclatement le long d'un centre bien précis.

\subsection{}
Soit $X$ un espace
$k$-analytique et soient $S$ et $Y$ deux sous-espaces
analytiques fermés de $X$. On note $S\cap Y$ le sous-espace
analytique fermé $S\times_XY$ de $X$ et $f$ le morphisme
canonique $\ecl XS\to X$. L'éclatement 
$\ecl Y{S\cap Y}$ s'identifie à un sous-espace analytique fermé de $\ecl XS$ qu'on appelle le \emph{transformé strict}
de $Y$ (relatif à l'éclatement de centre $S$). Modulo cette identification, $f\inv(Y\setminus S)$ est l'ouvert
$\ecl Y{S\cap Y}\setminus \dex Y{S \cap Y}$ de $\ecl Y{S\cap Y}$, qui est analytiquement dense dans $\ecl Y{S\cap Y}$
puisque $\dex Y{S\cap Y}$ est un diviseur de Cartier de 
$\ecl Y{S\cap Y}$ (\ref{ss-divcart-andense}). 
Autrement dit, le transformé strict de $Y$ est l'adhérence analytique de $f\inv(Y\setminus S)$ dans $\ecl XS$, qu'on peut également voir comme l'adhérence analytique de  $f\inv(Y\setminus S)$ 
dans $f\inv(Y)$. 

Considérons $\mathscr O_{f\inv(Y)}$ et
$\mathscr O_{\ecl Y{S\cap Y}}$
comme deux faisceaux cohérents sur $\ecl XS$. Par ce qui précède, 
$\mathscr O_{\ecl Y{S\cap Y}}$ est le quotient de $\mathscr O_{f\inv(Y)}$ par son sous-module formé des sections qui s'annulent sur
$\ecl XS\setminus \dex XS$, c'est-à-dire des sections qui sont à support (ensemblistement) dans $\dex XS$.

\subsection{}
Soit $X$ un espace $k$-analytique
et soit $S$ un sous-espace analytique fermé de $X$. On suppose 
que $X$ est irréductible et que $S$ est ensemblistement distinct de $X$.

\subsubsection{}\label{sss-ecl-irred}
Soit $f\colon \ecl XS\to X$ le morphisme canonique. 
Comme $f$ induit un isomorphisme $\ecl XS\setminus \dex XS\simeq X\setminus S$, l'ouvert $X\setminus S$ de $X$ est contenu dans $f(\ecl XS)$. Puisque $S$ est ensemblistement distinct de $X$, l'ouvert
de Zariski $X\setminus S$  est dense  dans $X$ (qui est irréductible) ; comme $f$ est fermé en tant que morphisme propre, il vient
$f(\ecl XS)=X$. 

\subsubsection{}\label{sss-ecl-irred2}
L'ouvert $X\setminus S$ de $X$ est irréductible
(lemme \ref{lem-ouvzar-connexe}). 
Or $\ecl XS$ induit un isomorphisme 
$\ecl XS\setminus \dex XS \simeq X\setminus S$. 
Puisque $\ecl XS\setminus \dex XS$ est
Zariski-dense dans $\ecl
XS$ 
(car $\dex XS$ est un diviseur de Cartier de $\ecl XS$),
il
en résulte que
$\ecl XS$ est irréductible. 
Comme il possède un ouvert dense isomorphe à $X\setminus S$, 
sa dimension est égale à celle de $X$.

\subsection{}
Soit  $X$ un espace
$k$-analytique et soit $S$ un sous-espace
analytique fermé de $X$. Notons
$f$ le morphisme
canonique $\ecl XS\to X$. Soit $(X_i)_{i\in I}$ la famille des composantes
irréductibles de $X$ n'étant pas contenues dans $S$, 
chacune étant munie de sa structure réduite. 
Pour tout $i$, on pose $S_i=S\times_X X_i$ et l'on identifie l'éclaté
$\ecl{X_i}{S_i}$ à un sous-espace analytique fermé
de $\ecl XS$. 

\subsubsection{}\label{sss-image-eclatement}
Il
résulte de \ref{sss-ecl-irred}
que $f(\ecl{X_i}{S_i})=X_i$. Par conséquent, $f(\ecl XS)$ contient
la réunion des
$X_i$ pour $i\in I$.
Soit $U$ le complémentaire de $\bigcup_{i\in I}X_i$ dans $X$. 
On a 
\[\ecl XS\times_X U=\ecl U{S\cap U}=\emptyset\]
car $S\cap U=U$ par définition de la famille $(X_i)$. 
Il vient $f(\ecl XS)=\bigcup_ {i\in I}X_i$. 

\subsubsection{}\label{sss-composantes-eclate}
L'ouvert $\ecl XS\setminus \dex XS$ est analytiquement dense 
dans $\ecl XS$ et il est égal à $\bigcup_{i\in I}f\inv(X_i\setminus S_i)$ ; par conséquent $\ecl XS$ s'identifie ensemblistement à
\[
\bigcup_i \adhz{f\inv(X_i\setminus S_i)}{\ecl XS}=
\bigcup_i \ecl{X_i}{S_i}.\]

D'après \ref{sss-ecl-irred2}, l'espace
$\ecl{X_i}{S_i}$ est pour tout $i$
un espace irréductible de dimension $\dim X_i$. Par ailleurs si $i$ et $j$ sont deux éléments
distincts de $I$, les fermés $\ecl{X_i}{S_i}$ et $\ecl{X_j}{S_j}$
de $\ecl XS$ sont non comparables pour l'inclusion : il suffit en effet de s'assurer que leurs traces sur $\ecl XS\setminus \dex XS$ ne le sont pas, mais elles s'identifient \emph{via} $f$ à 
$X_i\setminus S_i$ et $X_j\setminus S_j$, lesquels ne sont pas comparables pour l'inclusion puisque $X_i$ et $X_j$ ne le sont pas. 

Par conséquent $(\ecl{X_i}{S_i})_i$ est la famille des composantes 
irréductibles de $\ecl XS$.

\begin{lemm}\label{lem-compose-eclat}
Soit $X$ un espace
$k$-analytique compact, soit $S$ un sous-espace analytique fermé de $X$, et soit $f$ le morphisme canonique
$\ecl XS\to X$. Soit $T$ un sous-espace analytique fermé de $\ecl XS$. Il existe un sous-espace analytique fermé $\Sigma$ de $X$, dont le support est ensemblistement égal à $S\cup f(T)$, tel que
$\ecl{\ecl XS}T\simeq_X \ecl X\Sigma$.
\end{lemm}

\begin{proof}
Nous allons suivre \emph{mutatis mutandis}
la preuve de \cite[\href{https://stacks.math.columbia.edu/tag/0801}{Tag 0801}]{stacks-project}.
Soit $\mathscr I$ le faisceau cohérent d'idéaux qui définit $S$
et soit $\mathscr A$ la $\mathscr O_X$-algèbre graduée 
$\bigoplus_n \mathscr I^n$. Soit $i$ l'immersion fermée
$T\hookrightarrow \ecl XS$. On dispose d'un morphisme naturel
de $\mathscr O_X$-algèbres graduées
\[\mathscr A\to\bigoplus_n  f_*(i_*\mathscr O_T(n))\]
où la torsion par
$n$ fait référence au faisceau $\mathscr O(1)$ relatif
sur l'espace $X$-analytique $\ecl XS=\mathrm{Proj}\;\mathscr A$. 
Soit $\mathscr J$ son noyau. Ce dernier possède les propriétés suivantes : 
\begin{enumerate}[1]
\item pour tout $n$, le sommande $\mathscr J_n$ est un $\mathscr O_X$-module cohérent ; 
\item pour $n$ assez grand, l'immersion fermée 
\[\mathrm{Proj}(\mathscr A/\mathscr J_n
\cdot\mathscr A)\hookrightarrow \mathrm{Proj}\;\mathscr A=\ecl XS\]
identifie $\mathrm{Proj}(\mathscr A/\mathscr J_n
\cdot\mathscr A)$
à $Z$, et le support du faisceau cohérent sur $\ecl XS$ 
associé à $\mathscr A/\mathscr J_n\cdot\mathscr A$ est égal à $T$.
\end{enumerate}
En effet, (1) découle immédiatement de la noethérianité des algèbres affinoïdes ; et par compacité de $X$, on peut pour démontrer (2) supposer que $X$ est affinoïde, et partant se ramener à l'énoncé schématique analogue, qui est 
une version plus simple
(grâce à la cohérence de $\mathscr J_n$) de \cite[\href{https://stacks.math.columbia.edu/tag/0803}{Tag 0803}]{stacks-project}

Par définition, $\mathscr J_n$ est un sous-faisceau cohérent de $\mathscr I^n$ ; on peut donc le voir comme un faisceau cohérent d'idéaux sur $X$, qu'on note $\mathscr K$. Le sous-espace analytique fermé $\Sigma$ de $X$ défini par le faisceaux d'idéaux 
$\mathscr K\cdot \mathscr I$ convient alors : il suffit en effet de le prouver G-localement, ce qui permet de se ramener au cas affinoïde, puis à la variante schématique de l'énoncé, qui est prouvée dans \cite[\href{https://stacks.math.columbia.edu/tag/080B}{Tag 080B}]{stacks-project}.
\end{proof}

\begin{lemm}\label{lem-partition-ouverte}
Soit $X$ un espace
$k$-analytique
compact et soit $U$ un ouvert de Zariski de $X$. Supposons donnée une partition finie $U=\coprod_i U_i$ où les $U_i$ sont aussi des ouverts de Zariski de $X$. Il existe alors un sous-espace analytique fermé $S$ de $X$, supporté par $X\setminus U$, et une partition $\ecl XS=\coprod V_i$, où les $V_i$ sont des ouverts de Zariski de $\ecl XS$, telle que
$V_i\setminus \dex XS=\ecl XS\times_X U_i$ pour tout $i$.
\end{lemm}

\begin{proof}
Nous reprenons
essentiellement la preuve
du lemme
5.1.5 de \cite{raynaud-g1971}. 
Le lemme \ref{lem-compose-eclat}
ci-dessus
permet de se ramener (par une récurrence immédiate)
au cas où la partition donnée de $U$
comprend exactement deux ouverts $U_1$ et $U_2$.
Choisissons un faisceau cohérent d'idéaux $\mathscr J$
de lieu (ensembliste) des zéros $X\setminus U$.
Pour $i\in \{1,2\}$ on désigne par $\mathscr I_i$ le faisceau cohérent d'idéaux définissant l'adhérence analytique de $U_i$ dans $X$ ; 
puis on pose $\mathscr J_i=\mathscr J
\cdot\mathscr I_i$. 
Par construction, $\mathscr J_i|_{U_i}=0$, $\mathscr J_i|_{U_j}=\mathscr O_{U_j}$ pour $j\neq i$, et le lieu des zéros ensembliste de $\mathscr J_i$ est $X\setminus U_i$. 

Soit $S$ le sous-espace analytique fermé de $X$ défini par l'idéal $\mathscr J_1+\mathscr J_2$ ; il est supporté par $X\setminus U$. 
Pour $i\in\{1,2\}$, notons $U'_i$ l'image réciproque de $U_i$ sur $\ecl XS$ et $\mathscr J'_i$ l'image réciproque de $\mathscr J_i$ dans $\mathscr O_{\ecl XS}$. Par construction, $ \mathscr J'_1+\mathscr J'_2$ est un faisceau d'idéaux inversible de  $\mathscr O_{\ecl XS}$, et $(\mathscr J'_1\cap \mathscr J'_2)|_{U'_1\coprod U'_2}=0$. Puisque $U'_1\coprod U'_2$ est analytiquement dense dans
$\ecl XS$ (c'est le complémentaire du 
diviseur de Cartier $\dex XS$), il vient 
$\mathscr J'_1\cap \mathscr J'_2=0$. 
Pour $i\in\{1,2\}$, notons $T_i$ le sous-espace analytique
fermé de $\ecl XS$ défini par $\mathscr K_i:=(\mathscr J'_1+\mathscr J'_2)\inv \mathscr J'_i\subset
\mathscr O_{\ecl XS}$. On pose $V_2=\ecl XS\setminus T_2$ et $V_1=\ecl XS\setminus T_1$. 

On a $\mathscr K_1+\mathscr K_2=\mathscr O_{\ecl XS}$,
et donc $T_1\cap T_2=\emptyset$. Et par ailleurs $\mathscr K_1\cdot \mathscr K_2$ est contenu dans
$\mathscr J'_1\cap \mathscr J'_2=0$, si bien que $T_1\cup T_2=\ecl XS$ ; par conséquent, $\ecl XS=V_1\coprod V_2$. 

On a de plus $\mathscr K_1|_{U'_1}=0$ et $\mathscr K_1|_{U'_2}=\mathscr O_{U'_2}$. Il en résulte que $V_2\setminus \ecl XS=U'_2$, et on a de même $V_1\setminus \ecl XS=U'_1$.
\end{proof} 

Nous allons maintenant introduire les morphismes d'espaces analytiques qui nous servirons à aplatir les faisceaux cohérents. Il s'agit essentiellement des composés d'éclatements \emph{et de morphismes quasi-étales}, ce dernier point étant une nouveauté par rapport à ce qu'on rencontre en théorie des schémas. 

\subsection{}\label{ss-def-akgamma}
On note $\mathfrak A_{k,\Gamma}^{\mathrm{elem}}$
la classe des morphismes $Y\to X$ entre espaces $k$-analytiques
compacts et $\Gamma$-stricts qui admettent une factorisation de la forme
$Y\to Z\to X$ où $Z$ est compact et $\Gamma$-strict, où $Z\to X$ est quasi-étale et où $Y\to Z$ est un éclatement. 
Notons que tout éclatement et tout morphisme quasi-étale dont les sources et buts sont compacts et $\Gamma$-stricts appartiennent à $\mathfrak A_{k,\Gamma}^{\mathrm{elem}}$ (composer d'un côté ou de l'autre avec l'identité, qui est à la fois un morphisme quasi-étale et un éclatement). 

On définit récursivement sur l'entier naturel $h$ une classe 
$\akg k\Gamma h$ de morphismes entre espaces $k$-analytiques compacts et $\Gamma$-stricts, par le procédé suivant : 

\begin{itemize}[label=$\diamond$]

\item $Y\to X$ appartient 
à $\akg k\Gamma 0$ si et seulement si $X$ est compact et $\Gamma$-strict et $Y\to X$ est un isomorphisme ; 

\item si $h>0$ alors $Y\to X$ appartient 
à $\akg k\Gamma h$ si et seulement si
elle admet une factorisation 
$Y\to Z\to X$  
où $Z\to X$ appartient à $\akg k\Gamma{h-1}$
et où $Y\to Z$ appartient à $
\mathfrak A_{k,\Gamma}^{\mathrm{elem}}$. 
\end{itemize}
Nos notations sont consistantes : si 
$Y\to X$ appartient à $\akg k\Gamma h$,
il appartient aussi à $\akg k\Gamma {h+1}$, ce qu'on peut voir
en écrivant $Y\to X$ comme la flèche composée
\[\begin{tikzcd}
Y\ar[r,"\mathrm{Id}_Y"]&
Y
\ar[r]&X.\end{tikzcd}\]
Remarquons que $\akg k\Gamma 1$ n'est autre que
$\mathfrak A_{k,\Gamma}^{\mathrm{elem}}$. 
Nous noterons $\mathfrak A_{k,\Gamma}$ la réunion
des $\akg k\Gamma h$ pour $h$ parcourant $\N$.

Soit $h$ un entier et soit $Y\to X$ une flèche de $\akg
k\Gamma h$. Soit $L$ une extension complète de $k$ et soit $X'$ un espace $L$-analytique $\Gamma$-strict et compact muni d'un morphisme plat vers $X_L$ ; posons $Y'=Y\times_X X'$. On voit à l'aide de \ref{sss-ecl-chgebase}
et par une récurrence immédiate sur $h$ que $Y'\to X'$ est une flèche de
$\akg L\Gamma h$. 

\subsection{}\label{ss-def-dkgamma}
Nous allons maintenant introduire une variante de la définition précédente «avec spécification d'un diviseur exceptionnel».
On définit récursivement sur l'entier naturel $h$
une classe $\bkg k\Gamma h$ de diagrammes
$S\hookrightarrow Y\to X$ dans
la catégorie des espaces $k$-analytiques compacts et $\Gamma$-stricts,
par le procédé suivant : 

\begin{itemize}[label=$\diamond$]

\item $S\hookrightarrow Y\to X$ appartient 
à $\bkg k\Gamma 0$ si et seulement si $X$ est compact et $\Gamma$-strict, $Y\to X$ est un isomorphisme, et
$S$ est un diviseur de Cartier de $Y$ ; 

\item si $h>0$ alors $S\hookrightarrow Y\to X$ appartient 
à $\bkg k\Gamma h$ si et seulement s'il existe :
\begin{itemize}[label=$\bullet$] 
\item un diagramme
$T\hookrightarrow Z\to X$ appartenant à $\bkg k\Gamma{h-1}$ ; 
\item un espace analytique compact et $\Gamma$-strict
$Z'$ muni d'un morphisme quasi-étale $Z'\to Z$ ; 

\item un sous-espace analytique fermé $T'$
de $Z'$ majorant $Z'\times_Z T$,
 \end{itemize}
 tels que $Y\to X$ admette une factorisation 
 $Y\to Z'\to Z\to X$ 
 identifiant $Y$ à $\ecl{Z'}{T'}$ et $S$
 à $\dex{Z'}{T'}$. 
\end{itemize}

Nos notations sont consistantes : si 
$S\hookrightarrow Y\to X$ appartient à $\bkg k\Gamma h$, il appartient aussi à $\bkg k\Gamma {h+1}$, ce qu'on peut voir
en écrivant $Y\to X$ comme la composée
\[\begin{tikzcd}
\ecl YS\ar[r]&Y\ar[r,"\mathrm{Id}_Y"]&Y\ar[r]
&X.\end{tikzcd}\]

\subsection{}
Soit $h$ un entier et soit $S\hookrightarrow Y\to X$
un diagramme appartenant à $\bkg k \Gamma h$.

\subsubsection{}
\label{sss-defelem}
Il résulte immédiatement
de la définition et d'une récurrence sur $h$ que $Y\to X$ appartient à $\akg k \Gamma h$, que $S$ est un diviseur de Cartier de $Y$, et que $Y\to X$ est quasi-étale en dehors de
$S$.

\subsubsection{}
Soit $L$ une extension complète de $k$ et soit $X'$ un espace $L$-analytique $\Gamma$-strict et compact muni d'un morphisme plat vers $X_L$ ; posons $Y'=Y\times_X X'$ et $S'=S\times_Y Y'$. En combinant \ref{sss-ecl-chgebase} avec une récurrence sur $h$, on voit que le diagramme
$S'\hookrightarrow Y'\to X'$ appartient à $\bkg L\Gamma h$.

\begin{enonce}[remark]{Notation}\label{not-ax}
Soit $X$ un espace $k$-analytique. Nous noterons
$\mathsf A(X)$ l'ensemble des \emph{points d'Abhyankar de rang maximal} de $X$, c'est-à-dire l'ensemble des points $x$ de $X$
tels que $d_k(x)=\dim_x X$. 
\end{enonce}

\subsection{}
Soit $X$
un espace $k$-analytique.

\subsubsection{}\label{sss-ax-irred}
Si $X$ est irréductible tout point de $\mathsf A(X)$ est Zariski-dense dans $X$, et tout ouvert non vide $U$
de $X$ rencontre $\mathsf A(X)$ (car $\dim  U=\dim X$). 

\subsubsection{}\label{sss-ax-axi}
On ne suppose plus $X$ irréductible. Soit $(X_i)_{i\in I}$ la famille des composantes irréductibles de $X$. 
On a alors $\mathsf A(X)=\coprod_i \mathsf A(X_i)$.

En effet, soit $i\in I$ et soit $x\in \mathsf A(X_i)$. 
Comme $x$ est Zariski-dense dans $X_i$ il n'appartient pas à $\bigcup_{j\neq i}X_j$ (et \emph{a fortiori}
pas à $\bigcup_{j\neq i}\mathsf A(X_j)$). Il en résulte que
\[\dim_x X=\dim_x X_i=d_k(x)\]
et $x\in \mathsf A(X)$. 
Ainsi les $\mathsf A(X_i)$ sont deux à deux disjoints et leur réunion est contenue dans $\mathsf A(X)$. 

Réciproquement, soit $x\in \mathsf A(X)$ et soit $i$ tel que $x\in X_i$ et $\dim_x X=\dim X_i$. 
Comme $\dim_x X=d_k(x)$ on a $\dim X_i=d_k(x)$ et $x\in \mathsf A(X_i)$. 

\subsubsection{}\label{sss-ax-compirr}
Pour tout indice $i$, il résulte
de \ref{sss-ax-irred} que $\mathsf A(X_i)$ est dense dans $X_i$ et que chaque élément
de $\mathsf A(X_i)$ est Zariski-dense dans $X_i$. 
On en déduit en vertu de \ref{sss-ax-axi}
que pour tout fermé de Zariski $Y$ de $X$, 
l'adhérence $\adht{\mathsf A(X)\cap Y}X$
est la réunion des composantes irréductibles de $X$ contenues dans $Y$.

\begin{lemm}\label{dz-ecl-qet}
Soit $f\colon Y\to X$ une flèche
de $\mathfrak A_{k,\Gamma}$
et soit $y\in \mathsf A(Y)$. L'image $f(y)$ appartient alors
à $\mathsf A(X)$ 
\end{lemm}

\begin{proof}
On se ramène immédiatement au cas où $f$
est ou bien 
quasi-étale, ou bien un éclatement. Si $f$ est quasi-étale, il 
est en particulier plat et de dimension relative nulle.
On en déduit que
$\dim_{f(y)}X=\dim_y Y$, et 
la finitude de $\hr y$
sur $\hr {f(y)}$ assure par ailleurs que $d_k(f(y))=d_k(y)$ ; par conséquent, 
le lemme est vrai lorsque $f$ est
quasi-étale. 

Supposons maintenant que $f$ soit un éclatement, et choisissons-en un centre $S$. L'hypothèse 
faite sur $y$ implique que celui-ci n'est situé sur aucun fermé de Zariski d'intérieur vide de $Y$ ; 
en
particulier, $y$ n'est pas situé sur le diviseur
de Cartier $\dex XS$ ; le lemme découle alors
du fait que $f$ induit
un isomorphisme $Y\setminus \dex XS\simeq X\setminus S$. 
\end{proof}

\section{Aplatissement d'un faisceau cohérent}

\subsection{}\label{ss-oslash-coherent}
Soit $X$ un espace $k$-analytique, 
soit $\mathscr F$ un faisceau cohérent sur $X$
et soit $S$ un fermé de Zariski de $X$. 
Le sous-faisceau $\mathscr G$ de $\mathscr F$ formé des sections
dont le support est \emph{ensemblistement} contenu dans $S$ 
est alors cohérent. 

En effet, on peut pour le vérifier supposer $X$ affinoïde. Posons
$A=\mathscr O_X(X)$ ; notons $I$ un idéal de $A$ définissant $S$ 
et notons $M$ le $A$-module $\mathscr F(X)$. Il existe par noethérianité un entier $n$
tel que tout élément de $M$ annulé par une puissance de $I$ soit annulé par $I^n$.
Soit $N$ le sous-module de $M$ formé des éléments de $I^n$-torsion. Pour tout $f\in I$, 
la multiplication par $f$ définit un endomorphisme injectif du module $M/N$, puis du faisceau
cohérent associé $\widetilde{M/N}$. Il s'ensuit que
l'injection canonique $\widetilde N\hookrightarrow \mathscr G$ est un isomorphisme.

\subsection{}
Soit $Y\to X$ un morphisme d'espaces $k$-analytiques,
soit
$S$ un fermé de Zariski de $X$ et soit $\mathscr F$ un faisceau cohérent sur $Y$. 
Nous noterons $\mathscr F\oslash_X S$ le quotient de $\mathscr F$ par son sous-faisceau formé
des sections dont le support est \emph{ensemblistement}
contenu dans $Y\times_X S$ (nous emploierons aussi cette notation lorsque $S$ est un sous-espace analytique
fermé de $X$, mais seul compte alors son
ensemble sous-jacent). C'est un faisceau cohérent
d'après \ref{ss-oslash-coherent}. 
Si $Y=X$ nous écrirons simplement $\mathscr F\oslash S$.

\subsection{}
Soit $f\colon Y\to X$ un morphisme
d'espaces $k$-analytique
et soit $\mathscr F$ un faisceau cohérent sur $Y$. Rappelons que
$\pl {\mathscr F}X$ désigne le lieu de $X$-platitude de $\mathscr F$, et $\ql {\mathscr F} X$ son fermé de Zariski
complémentaire ; notons que $\ql {\mathscr F}X\subset \mathrm{Supp}(\mathscr F)$.
Pour tout entier $n$, le lieu des points de $\ql {\mathscr F} X$
en lesquels ce dernier est de dimension relative $\geq n$ sur $X$ sera noté $\qln {\mathscr F} X n$ ; c'est un fermé de Zariski
de $\ql {\mathscr F}X$ (remarquons que
$\qln {\mathscr F}X0$ est égal à $\ql {\mathscr F}X$).
Enfin, nous noterons $\qlns {\mathscr F}Xn$ le saturé de $\qln {\mathscr F} X n$ 
pour la relation d'équivalence ensembliste induite par $f$, c'est-à-dire l'ensemble
$f\inv(f(\qln{\mathscr F}Xn))$. 

\subsection{}\label{cartier-plat-inutile}
Soit $Y\to X$ un
morphisme d'espaces $k$-analytiques
et soit $\mathscr F$ un faisceau cohérent sur $Y$. 
Soit $S$ un diviseur de Cartier sur $X$. La flèche quotient de $\mathscr F$
vers $\mathscr F\oslash_XS$ est
alors un isomorphisme au-dessus de l'ouvert de Zariski $\pl {\mathscr F}X$ de $Y$.
En effet, on peut pour le voir
raisonner G-localement
et se ramener au cas où $Y$  et $X$ sont affinoïdes, et où $S$ est défini par l'équation $f=0$ pour une certaine
fonction $f$
non diviseur de zéro dans $\mathscr O_X(X)$. Pour tout $n$, la multiplication par $f^n$ est un endomorphisme injectif
de $\mathscr O_X$, et la multiplication par $f^n$ est donc un endomorphisme injectif de $\mathscr F_{\pl{\mathscr F}X}$,
ce qui entraîne aussitôt notre assertion. 

\begin{defi}
Soit $Y\to X$ un morphisme
d'espaces $k$-analytiques, et soit $\mathscr F$ un faisceau cohérent sur $Y$. 
Soit $n$ un entier. Nous dirons que le faisceau
$\mathscr F$ est \emph{$X$-plat en dimensions $\geq n$} si $\qln {\mathscr F}Xn=\varnothing$. 
Notons que $\mathscr F$ est $X$-plat en dimensions $\geq 0$
si et seulement s'il est $X$-plat.
\end{defi}

Nous pouvons maintenant énoncer le théorème principal
de cet article, qu'on peut voir comme une version analytique
du théorème 5.2.2 de \cite{raynaud-g1971}. Son énoncé utilise de façon cruciale
la classe de diagrammes $\bkg k\Gamma h$ introduite en 
\ref{ss-def-dkgamma}. 

\begin{theo}\label{flat-eclat}
Soit $f\colon Y\to X$ un morphisme entre espaces
$k$-analytiques compacts
et
$\Gamma$-stricts. 
Soient $n$ et $d$ deux entiers avec $n\leq d+1$, 
et soit $\mathscr F$ un faisceau cohérent sur $Y$ dont le support est partout
de dimension relative $\leq d$ sur $X$. 
Il existe 
une famille finie de couples
$[(S_i\hookrightarrow Z_i\to X,Y_i)]_i$, où 
\begin{itemize}[label=$\diamond$]
\item $S_i\hookrightarrow Z_i\to X$ est un
diagramme appartenant à $\bkg k\Gamma {d+1-n}$ ; 
\item $Y_i$ est un domaine $k$-analytique compact et $\Gamma$-strict
de $Y\times_X Z_i$,
\end{itemize}
famille qui satisfait les propriétés suivantes :
\begin{enumerate}[1]

\item pour tout  $i$, le faisceau cohérent $\mathscr F_{Y_i}\oslash_{Z_i}S_i$
est $Z_i$-plat en dimensions $\geq n$ ; 
\item pour tout $i$, l'ouvert $Y_i\setminus (Y_i\times_{Z_i}S_i)$
de $Y_i$ est contenu dans l'image réciproque de $Y\setminus \qln {\mathscr F}Xn$ ; 
\item pour tout $i$, le fermé
$Y_i\times_{Z_i}S_i$ 
de $Y_i$ est contenu dans l'image réciproque de $\qlns {\mathscr F}Xn$ ; 
\item la réunion des images des morphismes $Y_i\to Y$ contient $Y\setminus\qlns {\mathscr F}Xn$ ; 
\item la réunion des images des morphismes $Y_i
\times_{Z_i}S_i\to Y$
contient tout point de $\qln{\mathscr F}Xn$
dont l'image sur $X$ n'est pas
adhérente à $f(\qln{\mathscr F}Xn)\cap \mathsf A(X)$. 
\end{enumerate}

\end{theo}

\subsection{Commentaires}
Avant de démontrer le théorème, nous allons passer en revue
quelques cas particuliers. 

\subsubsection{Supposons que $\mathscr F$
est déjà $X$-plat en dimensions $\geq n$}\label{sss-theo-vide}
Le théorème est alors
vide car la famille constituée
de l'unique couple
$(\begin{tikzcd}
\emptyset\ar[r,hook]&X\ar[r,"\mathrm{Id}"]&X\end{tikzcd}, Y)$ satisfait tautologiquement
les conditions (1) à (5). 

\subsubsection{Supposons que $n=0$}
\label{aplat-strictosensu}
Nous avons alors
affaire à un théorème
d'aplatissement \emph{stricto sensu} : en effet, dans ce cas, 
le faisceau
cohérent $\mathscr F_{Y_i}\oslash_{Z_i}S_i$
est
$Z_i$-plat pour tout $i$.

\subsubsection{Supposons que $n=d+1$}
Le théorème est alors
vide  : en effet, comme les fibres de $\supp F\to X$ sont de dimension $\leq d$,
le faisceau cohérent $\mathscr F$ est automatiquement $X$-plat en dimensions $\geq d+1$, et 
on est donc dans la situation déjà décrite en \ref{sss-theo-vide}. 

Si nous avons choisi d'inclure le cas $n=d+1$ dans l'énoncé (au lieu d'imposer $n\leq d$, 
ce qui pourrait paraître plus naturel) c'est d'une part pour rendre triviale l'initialisation
de la récurrence descendante sur $n$ sur laquelle est fondée notre preuve, d'autre part pour
disposer d'un théorème d'utilisation aussi souple que possible. Par exemple, cela nous permettra de démontrer
le théorème \ref{thm-equidim}
sans perdre du temps à traiter à part le cas (par ailleurs trivial)
où $d=\delta$.

\subsubsection{Supposons que $\qlns {\mathscr F}Xn$
est d'intérieur vide dans $Y$}
\label{comment-presque-plat}
La propriété (4)
combinée à la compacité
des $Y_i$ implique
alors que les images des morphismes $Y_i\to Y$ recouvrent $Y$. 

\subsubsection{Supposons que $\qln {\mathscr F}Xn=Y$}
\label{comment-domfact}
La propriété (4)
est alors triviale et n'apporte aucune information sur les images des $Y_i\to Y$, et (2) assure par ailleurs que
$Y_i$ est contenu pour tout $i$ dans l'image réciproque de $S_i$, 
ce qui entraîne que $\mathscr F_{Y_i}\oslash_{Z_i}S_i=0$ : la 
propriété (1) est donc également évidente.
Notre
théorème est ainsi dans ce cas sans contenu
\emph{en ce qui concerne la platitude}. C'est conforme au principe général évoqué en introduction selon lequel les
techniques d'aplatissement n'ont pas vocation à créer de la platitude \emph{ex-nihilo}, mais à la propager un peu lorsqu'elle existe déjà ; remarquons que \ref{comment-presque-plat} est également une manifestation de ce principe, mais dans le cas opposé où il y a déjà «beaucoup» de platitude au départ. 

Mais notre théorème n'est pas vide pour autant lorsque
$\qln {\mathscr F}Xn=Y$.
Il affirme alors en effet comme on vient de le voir l'existence pour tout $i$ d'une factorisation
ensembliste de $Y_i\to Z_i$ par $S_i$, qui est un fermé de Zariski d'intérieur vide 
de $Z_i$ ;  et si de plus $X$ est réduit,
la fibre de $f$
en un point de $\mathsf A(X)$ est entièrement contenue 
dans $\pl{\mathscr F}X=\emptyset$, si bien que
$f(Y)=f(\qln {\mathscr F}Xn)$
ne contient aucun point de $\mathsf A(X)$ ; 
la combinaison de (1) et (2) entraîne
alors que la réunion des images
des $Y_i\to Y$ est égale à $Y$. 

C'est sur ce type de remarque
que se fonde
le théorème \ref{thm-dominant-factor}
que nous verrons plus loin.

\subsection{Démonstration du théorème \ref{flat-eclat}}
Elle est longue et comporte de nombreuses étapes. 

\subsubsection{Automaticité
de la propriété \textnormal{(2)}}\label{sss-auto2}
Remarquons pour commencer que (2) est une conséquence de (1). En effet, supposons (1) satisfaite, soit $i$ un indice, soit $t$ un point de $Y_i$ qui n'est pas situé au-dessus de $S_i$, et soit $y$ l'image de $t$ sur $Y$. L'assertion (1) assure que $t$ n'appartient pas à $\qln{\mathscr F_{Y_i}}Xn$ ; et comme $t$ n'est pas situé au-dessus de $S_i$, son image sur $Z_i$ appartient à $\pl {Z_i}X$. On en déduit alors par descente plate de la platitude que $y$ n'appartient pas 
à $\qln {\mathscr F}Xn$.

\subsubsection{Preuve
du théorème dans un premier cas particulier}\label{platif-eclat-premier}
On suppose que $d$ est égal à $n$
et que $f$ est quasi-lisse à fibres géométriquement
irréductibles de dimension $n$. 
Nous allons
montrer qu'il existe un sous-espace analytique fermé $S$ de $X$
de support $f(\qln {\mathscr F}Xn)$
tel
que le faisceau cohérent
$\mathscr F_{Y\times_X \ecl XS}\oslash_X S$
soit $\ecl XS$-plat en dimensions $\geq n$. C'est en un certain sens le cœur 
de notre démonstration, car c'est la seule étape lors de laquelle on construit
explicitement un éclatement ayant des vertus aplatissantes. Notre raisonnement
est une adaptation de celui suivi par Raynaud et Gruson à la section 5.4 de \cite{raynaud-g1971}. 

\paragraph{}
Supposons qu'on ait établi ce qui est 
annoncé en \ref{platif-eclat-premier}. 
Le théorème est alors valable dans notre cas particulier ; plus précisément, la
famille singleton
$\{(\dex XS
\hookrightarrow \ecl XS\to X, Y\times_X \ecl XS)\}$ répond à nos exigences. 
Il est en effet immédiat qu'elle satisfait (1), et partant (2) d'après \ref{sss-auto2}. En ce qui concerne (3), (4) et (5) commençons par remarquer que comme $S$ a pour support 
$f(\qln {\mathscr F}Xn)$, le fermé $\qlns {\mathscr F}Xn$
de $Y$ est égal à $Y\times_X S$. L'assertion (4) en découle tautologiquement, et l'assertion (3) s'en déduit compte-tenu
du fait que l'image de $\ecl XS\to X$ contient $X\setminus S$.
Il reste à vérifier (5). Soit $y$ un point 
de $\qln {\mathscr F}Xn$ dont l'image $x$ n'appartient
pas à l'adhérence de $S\cap \mathsf A(X)$ dans $S$.
Il résulte de 
\ref{sss-image-eclatement}
que $x$ appartient à l'image de $\dex XS$ sur $X$ ; il s'ensuit
que $y$ appartient à l'image de $Y\times_X \dex XS\to Y$.

\paragraph{}
Nous allons maintenant prouver l'assertion annoncée
en \ref{platif-eclat-premier}
en plusieurs étapes.
Remarquons pour commencer que sous nos hypothèses $\qln {\mathscr F}Xn$ est exactement la réunion des fibres de $Y\to X$ qui 
ne rencontrent pas $\pl{\mathscr F}X$. 
Soit $U$ le complémentaire de $\qln {\mathscr F}Xn$ ; c'est un ouvert de Zariski de $Y$.

Soit $V$ le lieu des points en lesquels $\mathscr F$ est
libre ;
c'est un ouvert de Zariski de $Y$. Comme $Y$ est quasi-lisse
sur $X$, le faisceau cohérent $\mathscr O_Y$ est plat
sur $X$, et $V$ est par conséquent contenu dans $\pl{\mathscr F}X$ et \emph{a fortiori}
dans $U$.

Soit $x$ un point de $X$ tel que $U_x$ soit non vide. Cela signifie 
que $Y_x$
rencontre $\pl{\mathscr F}X$ ; on déduit alors du lemme \ref{lem-plat-libre}
que $V$ rencontre $Y_x$. 

Ceci valant pour tout $x$ tel que $U_x$ soit non vide, 
on peut également décrire
$\qln {\mathscr F}Xn$
comme la réunion des fibres de $Y\to X$ qui ne rencontrent pas $V$. 

\paragraph{}
Soit $T$ le complémentaire de $V$ dans $Y$, muni (disons) de sa structure réduite. Les fibres de $Y\to X$ étant géométriquement irréductibles, le fermé $T$ de $Y$ possède  un «idéal des coefficients» sur $X$, définissant un sous-espace analytique fermé $S$ de $X$ (th. \ref{theo-id-coeff}). Par construction, un point $x$ de $X$ appartient à $S$ si et seulement si sa fibre $Y_x$ ne rencontre pas $V$, c'est-à-dire si et seulement si $x$ appartient à $f(\qln {\mathscr F}Xn)$ (en effet comme la fibre $Y_x$ est quasi-lisse, elle est régulière et en particulier réduite, et un sous-espace analytique fermé de $Y_x$ est donc isomorphe à $Y_x$ si et seulement si son support
est égal à $Y_x$); autrement dit, le support de $S$
est égal à $f(\qln {\mathscr F}Xn)$.

Nous allons maintenant démontrer l'existence d'un sous-espace analytique fermé $S'$
de $X$ de même support que $S$ tel que $\mathscr F_{Y\times_X \ecl X{S'}}\oslash_X S'$
soit $\ecl X{S'}$-plat en dimensions $\geq n$, ce qui achèvera la preuve de l'assertion souhaitée ; nous utiliserons 
désormais uniquement les faits suivants : 

\begin{itemize}[label=$\bullet$]
 \item $Y\to X$ est
quasi-lisse à fibres géométriquement irréductibles de dimension $n$ ; 
\item $V$ est un ouvert de Zariski de $Y$ tel que $\mathscr F_V$ soit G-localement libre (et $T$
désigne son complémentaire) ; 
\item $Y\times_X S$ est la réunion des fibres de $Y\to X$ qui ne rencontrent pas $V$. 
\end{itemize}
Nous procèderons à plusieurs reprises à des changement de base par des éclatements convenables, ce qui sera licite car les trois propriétés ci-dessus 
sont préservées par ce type d'opération, et car la
composée de deux éclatements est un éclatement
dont le support du centre est ce qu'on attend,
voir le lemme \ref{lem-compose-eclat}.
Nous n'utiliserons plus le fait que $Y\times_X S=\qln {\mathscr F}Xn$ (qui pourrait ne pas survivre à un changement de base). 

\paragraph{}
Soit $\Omega$ l'ouvert complémentaire de $S$ dans $X$. Pour tout point $x$ de $\Omega$, la fibre $V_x$ est non vide ; le faisceau $\mathscr F$ est libre en tout point de $V$, et en particulier en tout point de $V_x$ ; ce dernier étant un ouvert de Zariski de l'espace normal et connexe $Y_x$, il est connexe (lemme \ref{lem-ouvzar-connexe})
et le rang de $\mathscr F$ est donc égal à un même entier $r(x)$ en tout point de $V_x$.

Montrons que la fonction $x\mapsto r(x)$ est localement constante sur $\Omega$. Soit $x$ un point de $\Omega$, et soit $\Omega'$ un voisinage analytique compact de $x$ dans $\Omega$. Par hypothèse, le fermé $T_\omega$ est d'intérieur vide dans $Y_\omega$ pour tout $\omega \in \Omega'$. Le corollaire
\ref{coro-tube-transverse} assure alors qu'il existe un domaine analytique compact $T'$ de $Y\times_X \Omega'$ qui évite $T\times_X \Omega'$ et rencontre
toutes les fibres de $Y$ au-dessus de $\Omega'$. Comme $T'$ est contenu dans $V$, le faisceau $\mathscr F|_{T'}$ est G-localement libre. Pour tout $i$, notons $T'_i$ l'ouvert fermé de $T'$ sur lequel le rang
de $\mathscr F$ est égal à $i$. On a $T'_i=\emptyset$ pour presque tout $i$, et $T'=\coprod_i T'_i$. De plus, le rang étant constant sur $V_x$ pour tout $x$, les images des $T'_i$ sur $\Omega'$ sont deux à deux disjointes, et compactes. On obtient ainsi une écriture de $\Omega'$ comme une union disjointe finie d'ouverts fermés sur lesquels $r$ est constant ; en particulier, $r$ est constant au voisinage de $x$. 

Fixons un entier $m$. Le lieu $\Omega_m$ des points $x$ de $\Omega$ tels que $r(x)=m$ est en vertu
de ce qui précède un ouvert fermé de $\Omega$, et partant un ouvert de Zariski de $X$
(lemme \ref{zar-transitive}). 
D'après le lemme \ref{lem-partition-ouverte}, il existe un
sous-espace analytique fermé
$S'$ de $X$
de support $S$ tel que la partition finie $\Omega\times_X\ecl X{S'}=\coprod (\Omega_i\times_X\ecl X{S'})$ soit induite par une partition finie en ouverts de Zariski
de $\ecl X{S'}$. En remplaçant $X$ par $\ecl X{S'}$ et en raisonnant ouvert par ouvert, on se ramène au cas où $r$ a une valeur constante sur $\Omega$, encore notée $r$ (cette notation désigne donc désormais un entier, et non une fonction), et où $S$ est un diviseur de Cartier.  
Sous ces hypothèses,
toute composante assassine de $X$ rencontre $\Omega$.

\paragraph{}\label{par-compass-rencontre}
Montrons que $V$ est analytiquement dense dans $Y$.
Soit $\Sigma$ une composante assassine de $Y$. Comme $Y\to X$ est quasi-lisse, $\Sigma$ est une composante irréductible de $f\inv(\Theta)$ pour une certaine composante assassine $\Theta$ de $X$ (munie, disons, de sa structure réduite). 
Puisque $f\inv(\Theta)\to \Theta$ est plat l'image
$f(\Sigma)$ est Zariski-dense dans $\Theta$.
En tant que composante assassine de $X$, le fermé $\Theta$ rencontre $\Omega$, et l'image $f(\Sigma)$ n'est donc
par ce qui précède
pas contenue dans $S$; autrement dit,  $f\inv(S)\cap \Sigma$ est d'intérieur vide dans $\Sigma$. Il existe par conséquent un point $y$ de $\Sigma$, dont on note $x$ l'image sur $X$, qui n'est situé ni sur $f\inv(S)$, ni sur une composante irréductible de $f\inv(\Theta)$ autre que $\Sigma$. La dimension relative de $f|_\Sigma$ en $y$ est alors égale à la dimension relative de $f$ en $y$, c'est-à-dire à $n$. Par conséquent $\Sigma\cap Y_x=Y_x$ (la fibre $Y_x$ étant irréductible de dimension $n$).  Comme $x$ n'est pas situé sur $S$, la fibre $V_x$ est non vide, et $\Sigma$ rencontre donc $V$. Ce dernier est bien par conséquent 
analytiquement dense dans $Y$. 


\paragraph{}

Soit $\mathscr I$ le $r$-ième idéal de Fitting de $\mathscr F$
(\ref{ss-fitting})  et soit $\mathscr J\subset \mathscr O_X$
son «idéal des coefficients» (thm. \ref{theo-id-coeff}). Comme $\mathscr F_V$ est G-localement libre de rang $r$, on a $\mathscr I_V=\mathscr O_V$ ; puisque $V$ rencontre toutes les fibres de $Y\to X$ au-dessus de $\Omega$, il vient $\mathscr J_\Omega=\mathscr O_\Omega$.
Soit $\mathscr J'$ le faisceau d'idéaux inversible
définissant $S$. En remplaçant $X$ par son éclatement le long de $\mathscr J\cdot \mathscr J'$ (qui s'identifie à l'éclatement de $X$ le long de $\mathscr J$ mais a un centre qui 
est ensemblistement exactement égal à $S$), 
on se ramène au cas où $\mathscr J$ est inversible. 

Soit $\mathscr H$ le sous-faiceau cohérent $(\mathscr J\mathscr O_Y)\inv \mathscr I$
de $\mathscr K_Y$. Il est en fait contenu dans $\mathscr O_Y$, et son idéal des coefficients est lui-même égal à $\mathscr O_X$
(th. \ref{theo-id-coeff}). Soit $Y'$ l'ouvert complémentaire du fermé de Zariski de $Y$ défini par
le faisceau d'idéaux $\mathscr H$ ;
par définition on a $\mathscr H_{Y'}=\mathscr O_{Y'}$, 
et partant 
$\mathscr I_{Y'}=\mathscr J\mathscr O_{Y'}$ ; en particulier, $\mathscr I_{Y'}$
est 
inversible. 

Soit $x\in X$. Comme
l'idéal des coefficients
de $\mathscr H$ est égal à $\mathscr O_X$, le sous-espace
analytique fermé de $Y_x$ défini par $\mathscr H_{Y_x}$ n'est pas isomorphe
à $Y_x$, ce qui entraîne que son support est différent de $Y_x$ car l'espace $Y_x$ est quasi-lisse, donc régulier, donc réduit. 
En conséquence, l'ouvert $Y'$ rencontre toutes les fibres de $f$. 

Notons par ailleurs
que puisque
$\mathscr J_\Omega=\mathscr O_\Omega$,
on a $\mathscr J\mathscr O_{Y\times_X \Omega}=\mathscr O_{Y\times_X \Omega}$,
si bien que $\mathscr H_{Y\times_X \Omega}=\mathscr I_{Y\times_X \Omega}$. Il en résulte que $\mathscr H_V=\mathscr I_V=\mathscr O_V$ (et $V$ est donc contenu dans~$Y'$)
et que $\mathscr I_{Y'\times_X \Omega}=\mathscr H_{Y'\times_X \Omega}
=\mathscr O_{Y'\times_X \Omega}$. 

Nous allons maintenant montrer que 
$\mathscr F\oslash_X S$ est $X$-plat en dimensions $\geq n$, ce qui achèvera la démonstration de
l'assertion énoncée en \ref{platif-eclat-premier}. 
Soit $\mathscr N$ le sous-faisceau cohérent de $\mathscr F$ constitué des sections annulées par $\mathscr I$. Puisque $\mathscr I_{Y'\times_X \Omega}=\mathscr O_{Y'\times_X \Omega}$, 
on a $\mathscr N_{Y'\times_X\Omega}=0$ ; par conséquent, 
le support de toute section de $\mathscr N$ sur un ouvert de $Y'$ est contenu dans $Y'\times_X S$ et l'on dispose donc d'une suite exacte 
de faisceaux cohérents sur $Y'$
\[0\to \mathscr P\to (\mathscr F/\mathscr N)_{Y'}\to (\mathscr F\oslash_X S)_{Y'}\to 0\]
(où $\mathscr P$ est le faisceau  cohérent sur $Y'$
\emph{défini}
par cette suite exacte). 

Nous allons tout d'abord montrer que $(\mathscr F/\mathscr N)_{Y'}$ est
G-localement libre de rang $r$ ; comme $f$ est quasi-lisse, cela assurera que $(\mathscr F/\mathscr N)_{Y'}$ est $X$-plat.
Soit $W$ un domaine affinoïde de $Y$. Comme
l'ouvert de Zariski $V$ est analytiquement dense
dans $Y$ (\ref{par-compass-rencontre}), l'intersection $V\cap W$ est analytiquement dense
dans $W$ (\ref{ss-andense-interv}). Autrement dit, 
toute composante assassine de $W$ rencontre $V$, 
ce qui signifie en vertu de la proposition \ref{pro-compass-aff}  (2)
que l'assassin de $W\al$ est contenu dans $(W\cap V)\al$. Comme le faisceau $\mathscr F_V$ est 
G-localement libre de rang $r$, le faisceau
$(\mathscr F_W\al)_{(V\cap W)\al}$ est localement libre de rang $r$ ; le faisceau $\mathscr F_W\al$ est en conséquence
localement libre de rang $r$ en tout point de $\mathrm{Ass}(W\al)$.
On déduit alors 
du lemme 5.4.3 de \cite{raynaud-g1971}
que $\mathscr F_W\al/\mathscr N_W\al$ est localement libre de rang $r$ au-dessus du plus grand ouvert de $W\al$ sur lequel
le
$r$-ième idéal de Fitting $\mathscr I\al$
de $\mathscr F\al$ est inversible, ouvert qui contient
 $(Y'\cap W)\al$. 
Ceci valant pour tout domaine affinoïde $W$ de $Y$,  le faisceau $(\mathscr F/\mathscr N)_{Y'}$ est
G-localement libre de rang $r$ et par conséquent $X$-plat, comme annoncé. 

Comme $V$ ne rencontre pas $Y\times_X S$ on a 
$(\mathscr F/\mathscr N)_V=(\mathscr F\oslash_XS)_V=\mathscr F_V$, et il vient $\mathscr P_V=0$. Par ailleurs, le faisceau cohérent $\mathscr P$ est un sous-faisceau du faisceau
G-localement libre
$(\mathscr F/\mathscr N)_{Y'}$ ; toute composante assassine de $\mathscr P$ est alors une composante assassine de $Y'$, et rencontre dès lors $V$. Mais comme $\mathscr P_V=0$, ceci entraîne que $\mathscr P$ n'a pas de composantes assassines, ce qui signifie qu'il est nul ; par conséquent
$(\mathscr F\oslash_X S)_{Y'}=(\mathscr F/\mathscr N)_{Y'}$, et $(\mathscr F\oslash_X S)_{Y'}$
est de ce fait $X$-plat. 
Puisque $Y'$ rencontre toutes les fibres de $f$ (fibres qui sont irréductibles et de dimension $n$), le faisceau
cohérent $\mathscr F\oslash_X S$ est $X$-plat en dimensions $\geq n$.

\subsubsection{Preuve d'une variante du théorème 
dans un second cas particulier}\label{platif-coeur}
Supposons donné un diviseur de Cartier 
$S$ de $X$ tel que
le faisceau cohérent
$\mathscr F\oslash_X S$
soit $X$-plat
en dimensions $\geq n+1$.
Nous allons montrer qu'il existe 

\begin{itemize}[label=$\bullet$]
\item une famille finie $(Z_i\to X)_{1\leq i\leq m}$ de morphismes quasi-étales à sources
compactes et $\Gamma$-strictes ; 
\item pour tout $i$, un sous-espace analytique fermé $S_i$ de $Z_i$ majorant $Z_i\times_X S$
et un domaine analytique compact
et $\Gamma$-strict
$Y_i$ de $Y\times_{Z_i}\ecl{Z_i}{S_i}$, 
\end{itemize}
ces données étant assujetties aux conditions suivantes : 

\begin{enumerate}[i]
\item pour tout $i$, le faisceau
cohérent $\mathscr F_{Y_i}\oslash_{Z_i}S_i$
est $\ecl{Z_i}{S_i}$-plat en dimensions $\geq n$ ; 
\item pour tout $i$, le
fermé
$Y_i\times_{Z_i}S_i$
de $Y_i$ est contenu dans l'image réciproque 
de $\qlns {\mathscr F}Xn\cup (Y\times_X S)$  ; 
\item la réunion des images des $Y_i\to Y$ contient $Y\setminus (\qlns {\mathscr F}Xn\cup (Y\times_X S))$ ;  
\item la réunion des images des morphismes $Y_i\times_{X_i}S_i\to Y$
contient l'ensemble des points de
$\qln {\mathscr F}Xn\cup (Y\times_X S)$
dont l'image sur $X$ n'est pas adhérente à
\[
[f(\qln {\mathscr F}Xn)
\cup (f(Y)\cap S)]\cap \mathsf A(X)
.\]
\end{enumerate}

\paragraph{}
Pour établir l'assertion requise, on peut raisonner
G-localement
sur $X$ et $Y$. On peut donc supposer qu'ils sont tous deux affinoïdes. Soit $x\in X$. Par compacité
de $X$, il
suffit de démontrer l'assertion au-dessus d'un voisinage de $x$ ; nous nous autoriserons donc 
dans la suite du raisonnement à restreindre $X$ autour de $x$ autant que nécessaire. 

En raisonnant une dernière fois G-localement sur $Y$, on peut en vertu 
du lemme \ref{exist-deviss-glob}
supposer que l'on est dans l'un des deux cas suivants : 
\begin{itemize}[label=$\diamond$]
\item $Y_x$ ne rencontre pas le support de $\mathscr F\oslash_XS $ ; 
\item $Y_x$ rencontre le support de $\mathscr F\oslash_XS$
et il existe un $X$-dévissage total
et $\Gamma$-strict
\[\{T_i,\pi_i, u_i, \mathscr L_i,\mathscr P_i\}_{i\in \{1,\ldots, r\}}\]
de $\mathscr F\oslash_XS$ au-dessus de $x$, en dimensions $n_1>\ldots >n_r$
majorées par $\dim \mathrm{Supp}(\mathscr F\oslash_XS)_x$. 
\end{itemize}

\paragraph{}\label{g-deja-platn}Supposons
que $Y_x$ ne rencontre pas le support de $\mathscr F\oslash_XS$.
Par propreté topologique de $f$ on peut alors restreindre $X$
pour se ramener au cas où $\mathscr F\oslash_XS=0$. Sous cette hypothèse, et plus
généralement si $\mathscr F\oslash_XS$ est $X$-plat en dimensions $\geq n$, l'assertion requise
est 
vraie avec $m=1,Z_1=X,S_1=S$ et $Y_1=Y\times_X \ecl XS$. 

%
%

\paragraph{}\label{par-conditions-deviss}
On suppose désormais que $Y_x$ rencontre le support de $\mathscr F\oslash_XS$ et qu'il existe un
$X$-dévissage total \[\mathscr D=\{T_i,\pi_i, u_i, t_i,\mathscr L_i,\mathscr P_i\}_{i\in \{1,\ldots, r\}}\]
du faisceau $\mathscr F\oslash_XS$ au-dessus de $x$, en dimensions $n_1>\ldots >n_r$
majorées par $\dim \mathrm{Supp}(\mathscr F\oslash_XS)_x$. Quitte à tronquer
ce dévissage (qui ne sera dès lors peut-être plus total), on peut remplacer
la condition $\mathscr P_r=0$ par les deux conditions suivantes : 

\begin{enumerate}[a]
\item $\dim \mathrm{Supp}(\mathscr P_r)_x<n$ ; 
\item pour tout entier $i$ tel que $1\leq i<r$ on a $\dim \mathrm{Supp}(\mathscr P_i)_x\geq n$
(et partant $n_i\geq n+1$ et $n_{i+1}\geq n$).
\end{enumerate}

En vertu de la
proposition \ref{deviss-dim-pasplat}, on peut également
supposer quitte à restreindre $X$ que les deux assertions suivantes sont satisfaites : 

\begin{enumerate}[a]
\setcounter{enumi}{2}
\item pour tout $i$, l'intersection 
$\mathrm{Supp}(\mathscr P_i)\cap Y_\xi$ est d'intérieur vide dans $Y_\xi$ pour tout
$\xi \in X$ ; 
\item pour tout $i$ tel que $n_i\geq n+1$ (et en particulier pour tout $i<r$)
le morphisme $u_i$ est universellement $X$-injectif. 
\end{enumerate}

\paragraph{Supposons que $n_r<n$}
Comme $n_{i+1}\geq n$ pour tout $i$
tel que $1\leq i<r$ d'après  la
condition (b) de \ref{par-conditions-deviss}, il
vient $r=1$. 
La dimension relative
du support de  
$\mathscr F\oslash_XS$
sur $X$ est donc  strictement inférieure à $n$, et $\mathscr F\oslash_XS$ est en particulier $X$-plat en dimensions $\geq n$ ; 
l'assertion cherchée
est alors vraie (\ref{g-deja-platn}). 

\paragraph{Supposons que $n_r>n$} Dans ce cas $u_i$ est universellement injectif pour tout $i$ ; 
joint à la condition (a) ci-dessus ceci implique
d'après l'assertion (2) de la proposition \ref{deviss-dim-pasplat}
que $\mathscr F\oslash_XS$ est $X$-plat en dimensions $\geq n$, et l'assertion cherchée
est dès lors vraie (\ref{g-deja-platn}). 

\paragraph{Supposons que $n_r=n$}
Le théorème \ref{exist-decoupages}
appliqué au morphisme quasi-lisse
$T_r \to X$ assure l'existence d'une surjection quasi-étale à source affinoïde
et $\Gamma$-stricte
$X'\to X$
telle qu'il existe un $X'$-découpage $\Gamma$-strict
de $T_r\times_X X'$. Il suffit alors par propreté topologique de $X'\to X$
de montrer le théorème au voisinage
de chacun des antécédents de $x$ dans $X'$. Autrement dit quitte à remplacer $X$ par $X'$ et $x$
par n'importe lequel de ses antécédents sur $X'$, on peut supposer que $T_r\to X$
possède un $X$-découpage 
$\Gamma$-strict
$(T_{rj})_{1\leq j\leq e_r}$.

Nous allons construire récursivement 
pour tout $i$ un recouvrement fini $(T_{ij})_{1\leq j\leq e_i}$ de $T_i$ par des domaines analytiques compacts
et $\Gamma$-stricts. Le recouvrement
$(T_{rj})_j$ est déjà construit. 
Supposons $(T_{ij})_j$ construit pour un certain $i$ compris entre $2$ et $r$, et expliquons comment 
construire $T_{i-1,j}$ ; pour tout $j$ compris entre $1$ et $e_i$ on note $X_{ij}$ l'image
de $T_{ij}$ sur $X$ ; c'est un domaine
analytique compact
et $\Gamma$-strict
de $X$. Soit $j$ compris entre $1$ et $e_i$. On choisit un domaine analytique compact
et $\Gamma$-strict
$\Upsilon$ de $T_{i-1}$ tel que $\Upsilon\cap \mathrm{Supp}(\mathscr P_{i-1})=\pi_i\inv(T_{ij})$
(ce qui est possible
grâce au lemme \ref{lem-gerr-grau})
et l'on 
pose alors $T_{i-1,j}=\Upsilon\times_X{X_{ij}}$. Par construction, $T_{i-1,j}\cap  \mathrm{Supp}(\mathscr P_{i-1})=\pi_i\inv(T_{ij})$
pour tout $j$ compris entre $1$ et $e_i$. La réunion $\bigcup_{1\leq j\leq e_i}T_{i-1,j}$ est donc
un domaine
analytique compact
et $\Gamma$-strict
de $T_{i-1}$ qui contient $\mathrm{Supp}(\mathscr P_{i-1})$ et en est dès lors
un voisinage. Si  $\bigcup_{1\leq j\leq e_i}T_{i-1,j}$ est égal à $T_{i-1}$ tout entier, on considère que la famille $(T_{i-1,j})_j$ est construite.
Sinon il existe un domaine analytique
compact et $\Gamma$-strict
$T_{i-1,e_i+1}$ de $T_{i-1}$ ne rencontrant pas 
 $\mathrm{Supp}(\mathscr P_{i-1})$ et tel que $\bigcup_{1\leq j\leq e_i+1}T_{i-1,j}=T_{i-1}$, et l'on considère que la famille
 $(T_{i,j})_j$ est construite. (On a donc $e_{i-1}=e_i$ ou $e_{i-1}=e_i+1$). 
 
 On construit enfin un recouvrement $(Y_j)_{1\leq j\leq e_0}$ de $Y$ par un procédé analogue (prendre ci-dessus
 $i=1$ et remplacer $T_{i-1}$ par $Y$ et $\mathscr P_{i-1}$ par $\mathscr F\oslash S$). Soit $j$ compris entre $1$ et $e_0$. 
 Notons $\epsilon_j$ le plus grand entier $i$ compris entre $0$ et $r$ tel que $j\leq e_i$ (si $\epsilon_j <r$ on a donc $j=e_{\epsilon_j}$ et $e_{\epsilon_j+1}=e_{\epsilon_j}-1$, 
 ce qui entraîne que $\mathscr P_{e_j,T_{\epsilon_jj}}=0$), 
 et notons $X_j$ l'image de $T_{\epsilon_jj}$ sur $X$ (c'est un domaine analytique compact
 et $\Gamma$-strict
de $X$) si $\epsilon_j>0$ ; si $\epsilon_j=0$ (ce qui veut dire
 que $j=e_0$ et que $e_1=e_0-1$) on pose $X_j=X$. Il résulte de nos constructions que $Y_j\to X$ se factorise par $X_j$.
 Il suffit maintenant de montrer l'assertion requise
 pour chacun des morphismes $Y_j\to X_j$.
On fixer donc $j$, et l'on distingue plusieurs cas. 

\emph{Supposons que $\epsilon_j=0$}. On a alors $j=e_0$ et $(\mathscr F\oslash_XS)_{Y_j}=0$, et l'assertion requise vaut
dès lors pour $Y_j\to X_j$
(\ref{g-deja-platn}).

\emph{Supposons que $0<\epsilon_j<r$}. La donnée $\mathscr D_j$ des $T_{ij}$ pour $i$ variant
de $1$ à $\epsilon_j$ et des restrictions correspondantes des faisceaux $\mathscr L_i$, des morphismes
$\pi_i$ et des applications linéaires $u_i$ définit alors un $X_j$-dévissage 
de $(\mathscr F\oslash_XS)_{Y_j}$
au-dessus de tout point de $X_j$. En effet, soit $\xi$ un point de $X_j$ et
soit $i$ un entier compris entre $1$ et $\epsilon_j$. 
Par hypothèse, l'intersection du support de $\mathscr P_{i,T_{ij}}$ avec $T_{ij,\xi}$ est d'intérieur
vide dans $T_{ij,\xi}$ ; par un raisonnement
reposant \emph{in fine}
sur le lemme de Nakayama
(\cite{ducros2018}, 2.5.4)
cela signifie que le lieu de surjectivité de
$u_{i,T_{ij}}$ contient un ouvert dense de $T_{ij,\xi}$, et il résulte
par ailleurs de nos hypothèses et des inégalités $i\leq \epsilon_j<r$ que
$u_{i,T_{ij}}$ est
(universellement) injectif ; par conséquent, 
le lieu de bijectivité de $u_{i,T_{ij}}$ contient un ouvert dense
de $T_{ij,\xi}$, et le lieu de bijectivité de $u_{i,T_{ij,\xi}}$ est
\emph{a fortiori}
dense
dans $T_{ij,\xi}$. Ainsi, $\mathscr D_j$ est un $X_j$-dévissage de $(\mathscr F\oslash_XS)_{Y_j}$ 
au-dessus de tout point de $X_j$, comme annoncé. Ce dévissage est total : en effet comme
$\epsilon_j<r$ on a $\mathscr P_{e_j,T_{\epsilon_jj}}=0$. Puisque 
$u_{i,T_{ij}}$ est (universellement) injectif pour tout $i$ compris entre $1$ et $\epsilon_j$, 
on déduit de \ref{ss-deviss-critplat}
que $(\mathscr F\oslash_XS)_{Y_j}$ est $X_j$-plat ; l'assertion requise vaut alors pour 
$Y_j\to X_j$
(\ref{g-deja-platn}).

\emph{Supposons que $e_j=r$}. Le morphisme $T_{rj}\to X_j$ est quasi-lisse à fibres géométriquement
connexes, et surjectif ; autrement dit, il est quasi-lisse à fibres géométriquement irréductibles. Ses fibres sont
de dimension $n$
(rappelons en effet qu'on s'est placé dans le cas où $n_r=n$). Posons $\mathscr Q=(\pi_{r*}\mathscr P_{r-1})_{T_{rj}}$.
En vertu du cas particulier traité au \ref{platif-eclat-premier}, 
il existe un sous-espace analytique fermé $F$
de $X_j$, de
support l'image de $\qln {\mathscr Q}{X_j}n=\qlns{\mathscr Q}{X_j}n$, tel que
$\mathscr Q_{T_{r,j}\times_{X_j} \ecl{X_j}F}\oslash_
{X_j}F$ 
soit $\ecl {X_j}F$-plat
en dimensions $\geq n$. Soit 
$\Sigma$ le sous-espace analytique fermé de $X_j$ défini par le produit des idéaux définissant respectivement $F$ et $S\cap X_j$ ; son support est la réunion de $S\cap X_j$ et
de l'image de $\qln {\mathscr Q}{X_j}n$. 
L'éclaté $\ecl {X_j}\Sigma$ est égal à $\ecl {X_j}F$
(\ref{sss-centre-change}). De plus, le lieu de
$\ecl {X_j}\Sigma$-platitude
de $\mathscr Q_{T_{r,j}\times_{X_j}
\ecl{X_j}\Sigma}\oslash_{X_j}\Sigma$
contient
celui de
$\mathscr Q_{T_{r,j}\times_{X_j} \ecl{X_j}\Sigma}\oslash
_{X_j}
F$ ;
en effet,
le premier est le quotient du second par son sous-faisceau formé des sections à support dans
$T_{rj}\times_{X_j}\dex {X_j}\Sigma$,
qui est un diviseur de Cartier de
$T_{rj}\times_{X_j}\ecl {X_j}\Sigma$ car $T_{r,j}\to X_j$ est plat (puisque quasi-étale), et notre affirmation est alors une conséquence de 
\ref{cartier-plat-inutile}. En conséquence, 
$\mathscr Q_{T_{r,j}\times_{X_j}
\ecl{X_j}\Sigma}\oslash_{X_j}
\Sigma$ est
$\ecl {X_j}\Sigma$-plat en dimensions $\geq n$. 

Nous allons
maintenant montrer les énoncés suivants :

\begin{enumerate}[A]
\item On a l'égalité
\[f(\qln {\mathscr F_{Y_j}}{X_j}n)\cup (S\cap X_j)=
f(\qln {\mathscr Q}{X_j}n)
\cup (S\cap X_j).\]

\item Le faisceau
cohérent
\[(\mathscr F\oslash_XS)_{Y_j\times_{X_j}\ecl {X_j}\Sigma}\oslash_{X_j}\Sigma
=\mathscr F_{Y_j\times_{X_j}\ecl {X_j}\Sigma}\oslash_{X_j}\Sigma\]
est $\ecl{X_j}\Sigma$-plat
en dimensions $\geq n$. 
\end{enumerate}
Avant de le faire, expliquons pourquoi ceci entraîne que $Y_j\to X_j$ satisfait l'énoncé \ref{platif-coeur}, avec $m=1, Z_1=X_j, S_1=\Sigma$, et en prenant comme domaine analytique de $Y_j\times_{X_ j}\ecl {X_j}\Sigma$ cet espace tout entier. 

L'assertion (i) est une conséquence directe de (B). Les
assertions (ii)
et (iii) résultent de (A), du fait que
$\Sigma$ est égal ensemblistement à $(S\cap X_j)\cup f(\qln {\mathscr Q}{X_j}n)$, et  du fait que $\ecl {X_j}\Sigma\to X_j$ induit un isomorphisme au-dessus de $X_j\setminus \Sigma$ (ce qui entraîne que 
$Y_j\times_{X_j}\ecl{X_j}\Sigma\to Y_j$ induit un isomorphisme
au-dessus de $Y_j\setminus (Y_j\times_{X_j}\Sigma)$).

Vérifions enfin (iv). Donnons-nous
un point 
$y$ sur $(Y_j\times_X S)
\cup \qln {\mathscr F_{Y_j}}{X_j}n$
dont l'image $x$ par $f$ n'est pas adhérente à 
\[[(f(Y_j)\cap S)\cup f(\qln {\mathscr F_{Y_j}}{X_j}n)]
\cap \mathsf A(X)
=[(f(Y_j)\cap S)\cup f(\qln {\mathscr Q}{X_j}n)]
\cap \mathsf A(X),\]
et montrons
que $y$ appartient à l'image de $Y_j\times_{X_j}\dex{X_j}\Sigma$. 
Le point $x$ appartient à $\Sigma$ par définition, et aucune composante irréductible de $X_j$ contenue dans $\Sigma$
ne passe par $x$.  En effet si c'était le cas cette composante serait contenue dans 
$f(\qln {\mathscr F_{Y_j}}{X_j}n)$, car
$S\cap X_j$ est un diviseur de Cartier de $X_j$, et est en particulier d'intérieur vide dans $X_j$, 
et $x$ serait adhérent à 
$ f(\qln {\mathscr F_{Y_j}}{X_j}n)\cap \mathsf A(X)$, en contradiction avec notre hypothèse.

Puisqu'aucune composante irréductible de $X_j$ contenue dans $\Sigma$ ne passe par $x$, le point $x$ appartient à l'image de $\dex {X_j}\Sigma\to X_j$ (\ref{sss-composantes-eclate}) ; il en résulte que $y$ appartient à l'image de 
de $Y_j\times_{X_j}\dex{X_j}\Sigma$.

\paragraph{Preuve de l'assertion \textnormal{(A)}}
Soit $\xi \in X_j$. La dimension
de $\mathrm{Supp}(\mathscr P_{r,T_{rj}})_\xi$ est strictement
inférieure à $n_r=n$, et on sait par ailleurs
que $u_{i,T_{ij}}$ est universellement
$X_j$-injectif pour $i$ variant entre $1$ et $r-1$. 
En appliquant la proposition \ref{deviss-dim-pasplat}
d'une part au $X_j$-dévissage à un cran 
$\{T_{rj}, \mathrm{Id}, u_{r,T_{rj}},
\mathscr L_{r,T_{rj}}, \mathscr P_{r,T_{rj}}\}$ de $\mathscr Q$
et d'autre part au $X_j$-dévissage $\mathscr D_j$ du
faisceau cohérent
$(\mathscr F\oslash_XS)_{Y_j}$, on voit que les assertions suivantes
sont équivalentes : 

\begin{enumerate}[a]
\item $\dim \ql{\mathscr Q}{X_j}_\xi<n$ ; 
\item $u_r$ est injectif en tout point de $T_{rj,\xi}$ ; 
\item $\dim \ql{(\mathscr F\oslash_XS)_{Y_j}}{X_j}_\xi <n$.
\end{enumerate}
Si $\xi$ n'appartient pas à $S$
l'assertion (c) équivaut simplement à demander que
$\dim \ql{\mathscr F_{Y_j}}{X_j}_\xi <n$. 
L'équivalence (a)$\iff$(c)
implique donc que
\[f(\qln {\mathscr F_{Y_j}}{X_j}n)\setminus S=
f(\qln {\mathscr Q}{X_j}n)
\setminus S,\]
d'où (A).

\paragraph{Preuve de \textnormal{(B)}}
On pose $\mathscr G=(\mathscr F\oslash_XS)_{Y_j}$. En procédant
au changement de base $\ecl{X_j}F\to X_j$ on se ramène,
sans changer les propriétés du dévissage $\mathscr D_j$,
au cas où $\Sigma$
est un diviseur de Cartier, où $\ecl{X_j}\Sigma=X_j$
et où $\mathscr Q\oslash_{X_j}\Sigma$ est
$X_j$-plat en dimensions $\geq n$, et il suffit de
démontrer que ceci entraîne la $X_j$-platitude de $\mathscr G\oslash_{X_j}\Sigma$ en dimensions $\geq n$. 

Comme $\mathscr Q\oslash_{X_j}\Sigma$ est $X_j$-plat en dimensions $\geq n$, il en 
va de même de $\mathscr P_{r-1, T_{r-1,j}}\oslash_{X_j}\Sigma$
(\cite{ducros2018}, Lemma 4.1.15 (3)). 
Pour tout $i$ compris entre $1$ et $r-1$, la suite exacte
\[0\to \mathscr L_{i,T_{ij}}\to \pi_{i*}\mathscr P_{i-1,T_{ij}}\to \mathscr P_{i,T_{ij}}\to 0\]
(en posant $T_{0j}=Y_j$ et $\mathscr P_{0j}=\mathscr G$) induit d'après 
le lemme 5.5.3 (ii)
de \cite{raynaud-g1971}
une suite exacte 
\[0\to \mathscr L_{i,T_{ij}}\to (\pi_{i*}\mathscr P_{i-1,T_{ij}})\oslash_{X_j}\Sigma\to \mathscr P_{i,T_{ij}}\oslash_{X_j}\Sigma\to 0,\]
d'où un dévissage $\mathscr D'_j$ de $\mathscr G\oslash_{X_j}\Sigma$ en dimensions $n_1>\ldots>n_{r-1}$. 
Comme les applications linéaires qui le constituent sont injectives, et comme
le faisceau cohérent
$\mathscr P_{r-1,T_{r-1,j}}\oslash_{X_j}\Sigma$ est $X_j$-plat en dimensions $\geq n$, il résulte de la 
proposition  \ref{deviss-dim-pasplat}
que le faisceau cohérent
$\mathscr G\oslash_{X_j}\Sigma$ est $X_j$-plat en dimensions $\geq n$.

\subsubsection{Preuve du cas général}Nous allons démontrer le cas
général en procédant par récurrence descendante sur l'entier $n\leq d+1$.
On suppose donc que le théorème est vrai pour tout entier strictement supérieur à $n$ et inférieur ou égal à $d+1$.

\paragraph{Le cas où $n=d+1$}
Comme les fibres du morphisme 
$\supp F\to X$ sont toutes de dimension $\leq d$, le faisceau $\mathscr F$ est
$X$-plat en dimensions $\geq d+1$. Par conséquent, les conclusions du théorème sont vérifiées par la famille
constituée de l'unique couple
$(\begin{tikzcd}
\emptyset\ar[r,hook]&X\ar[r,"\mathrm{Id}"]&X\end{tikzcd}, Y)$.

\paragraph{Le cas où $n<d+1$}
L'hypothèse de récurrence assure
qu'il existe une famille finie $(S_i\hookrightarrow
Z_i\to X)$ de
diagrammes appartenant à
$\bkg k\Gamma {d-n}$
et, pour tout $i$, un
domaine analytique compact 
et $\Gamma$-strict
$Y_i$ de $Y\times_X Z_i$, tels que les propriétés suivantes soient satisfaites : 

\begin{enumerate}[s]
\item pour tout $i$, le faisceau cohérent $(\mathscr F_{Y_i})\oslash_{Z_i}S_i$
est $Z_i$-plat en dimensions $\geq n+1$ ; 
\item pour tout $i$,
l'ouvert $Y_i\setminus (Y_i\times_{Z_i}S_i)$
de $Y_i$
est contenu dans l'image réciproque de $Y\setminus
\qlns {\mathscr F}X{n+1}$ ; 
\item pour tout $i$,
le fermé $Y_i\times_{Z_i}S_i$
de $Y_i$
est contenu dans l'image réciproque de $\qlns {\mathscr F}X{n+1}$ ; 
\item la réunion des images des morphismes $Y_i\to Y$ contient $Y\setminus\qlns {\mathscr F}X{n+1}$ ; 
\item la réunion des images des morphismes $Y_i
\times_{Z_i}S_i\to Y$
contient tout point de $\qln {\mathscr F}X{n+1}$
dont l'image sur $X$ n'est pas adhérente à $f(\qln {\mathscr F}X{n+1})
\cap \mathsf A(X)$. 

\end{enumerate}

D'après les résultats de \ref{platif-coeur}, il existe pour tout $i$
une famille finie $[(Z_i^\ell, S_i^\ell, Y_i^\ell)]_{\ell}$ où :

\begin{itemize}
\item [$(\alpha)$] $Z_i^\ell$ est pour tout $\ell$ un espace
$k$-analytique compact
et $\Gamma$-strict
muni d'un morphisme quasi-étale vers $Z_i$, et $S_i^\ell$
est un sous-espace analytique fermé de $Z_i^\ell$ majorant $Z_i^\ell\times_{Z_i}S_i$ ; 
\item [$(\beta)$] $Y_i^\ell$ est pour tout $\ell$ un domaine analytique compact
et $\Gamma$-strict
de $Y_i\times_{Z_i}\ecl{Z_i^\ell}{S_i^\ell}$, tel que
\[\mathscr F_{Y_i^\ell}\oslash_{Z_i^\ell} {S_i^\ell}=(\mathscr F_{Y_i}\oslash_{Z_i}S_i)_{Y_i^\ell}
\oslash_{Y_i^\ell}{S_i^\ell}\]
soit $\ecl {Z_i^\ell}{S_i^\ell}$-plat en dimensions $\geq n$ ; 
\item [$(\gamma)$] pour tout indice $\ell$, l'image du morphisme $Y_i^\ell\times_{Z_i^\ell}S_i^\ell\to Y_i$ est contenue dans
\[\qlns {\mathscr F_{Y_i}\oslash_{Z_i}S_i}{Z_i}n\cup (Y_i\times_{Z_i}S_i).\]
\item [$(\delta)$] la réunion pour $\ell$ variable des images des morphismes $Y^\ell_i\to Y_i$ contient
$Y_i\setminus (\qlns{\mathscr F_{Y_i}\oslash_{Z_i}S_i}{Z_i}n\cup (Y_i\times_{Z_i}S_i))$ ; 
\item[($\epsilon$)] si l'on note
$f_i$ le morphisme $Y_i\to Z_i$, la réunion pour $\ell$
variable des images des morphismes
$Y_i^\ell\times_{Z_i^\ell}S_i^\ell\to Y_i$
contient tout point de 
\[(Y_i\times_{Z_i}S_i)\cup \qln {\mathscr F_{Y_i}\oslash_{Z_i}S_i}{Z_i}n\]
dont l'image
par $f_i$ n'est pas adhérente
à
\[
[(f_i(Y_i)\cap S_i)\cup f_i(\qln {\mathscr F_{Y_i}\oslash_{Z_i}S_i}{Z_i}n)]\cap \mathsf A(Z_i).\]
\end{itemize}

La famille
$[(\dex{Z_i^\ell}
{S_i^\ell}\hookrightarrow
\ecl {Z_i^\ell}
{S_i^\ell}\to X, Y_i^\ell)]_{i,\ell}$
satisfait alors les conclusions du théorème. 
C'est en effet évident pour (1)
qui découle de ($\beta$), et on sait que (2) est alors automatiquement
vérifiée (\ref{sss-auto2}). 
En ce qui concerne (3), soit $y$ un point de $Y\setminus \qlns {\mathscr F}Xn$ ; il appartient \emph{a fortiori}
à $Y\setminus \qlns {\mathscr F}X{n+1}$ et est donc égal d'après (4$^\ast$)
à l'image d'un point $y'$ de $Y_i$ pour un certain $i$. Comme $y$ n'appartient pas 
à $\qlns{\mathscr F}X{n+1}$, la condition (3$^\ast$)
entraîne que $y'$ 
n'appartient pas à $Y_i\times_{Z_i}S_i$  ; il 
est donc situé au-dessus
du lieu de platitude de $Z_i\to X$.
Par conséquent, comme $y$ n'appartient pas à $\qlns{\mathscr F}Xn$, le point $y'$
n'appartient
pas à $\qlns{\mathscr F_{Y_i}}{Z_i}n$, et donc pas 
non plus à $\qlns{\mathscr F_{Y_i}\oslash_{Z_i}S_i}{Z_i}n$
car $\mathscr F_{Y_i}\oslash_{Z_i}S_i$
coïncide avec $\mathscr F_{Y_i}$ au-dessus
de $Z_i\setminus S_i$.
Il s'ensuit que $y'$
appartient en vertu de $(\delta)$ à l'image de $Y_i^\ell\to Y_i$
pour un certain $\ell$, et $y$ appartient
dès lors à l'image de $Y_i^\ell\to Y$.

Vérifions (4). Fixons un couple $(i,\ell)$, soit $y$ un point de $Y_i^\ell\times_{Z_i^\ell}S_i^\ell$ et
soit $\eta$ son
image sur $Y_i$. 
Il résulte de $(\gamma)$ que
\[\eta \in \qlns {\mathscr F_{Y_i}\oslash_{Z_i}S_{i}}{Z_i}n\cup  (Y_i\times_{Z_i}S_i).\]
On distingue maintenant deux cas : si 
$\eta$ appartient à $Y_i\times_{Z_i}S_i$ il est alors situé d'après
la propriété (3$^\ast$) au-dessus de $\qlns {\mathscr F}X{n+1}$, et \emph{a fortiori}
au-dessus de $\qlns{\mathscr F}Xn$ ; sinon $\eta$ appartient à 
$\qlns {\mathscr F_{Y_i}\oslash_{Z_i}S_i}{Z_i}n$
et le même raisonnement que 
celui suivi pour la preuve de (3) montre alors
que $\eta$
est situé au-dessus de $\qlns {\mathscr F}Xn$.

Montrons enfin (5). Soit $y$ un point de
$f(\qln{\mathscr F}Xn)$
dont l'image $x$ sur $X$  n'est pas
adhérente à $f(\qln{\mathscr F}Xn)\cap \mathsf A(X)$ ; nous
allons prouver que
le point $y$ appartient à l'image de l'un des
morphismes $Y_i^\ell\times
_{Z_i^\ell}S_i^\ell\to Y$.

Nous allons tout d'abord montrer qu'il existe $i$ tel que $y$
possède un antécédent $y'$ sur
$\Lambda_i:=(Y_i\times_{Z_i}S_i)\cup \qln {\mathscr F_{Y_i}\oslash_{Z_i}S_i}{Z_i}n.$
On distingue pour ce faire deux cas. 

\emph{Supposons que $x$ n'appartient pas à $f(\qln{\mathscr F}X{n+1})$}. 
Le point $y$ n'appartient alors
pas à $\qlns{\mathscr F}X{n+1}$.
Il s'ensuit en vertu de (4$^\ast$)
qu'il existe
un indice $i$ 
et un point $y'$ de $Y_i$ situé au-dessus de $y$. D'après
(3$^\ast$) 
le point $y'$ n'appartient pas à $Y_i\times_{Z_i}S_i$ ; il
est situé par conséquent 
au-dessus
du lieu de platitude de la
flèche $Z_i\to X$, si bien qu'il appartient
à $\qln{\mathscr F_{Y_i}}{Z_i}n$, et partant
à $\qln{\mathscr F_{Y_i}\oslash_{Z_i}S_i}{Z_i}n$
puisqu'il n'est pas situé au-dessus de $S_i$. 
Il appartient donc bien à 
$\Lambda_i$. 

\emph{Supposons maintenant que $x$ appartient à $f(\qln{\mathscr F}X{n+1})$}. Puisque
$x$ n'est, par
hypothèse, pas adhérent 
à $f(\qln{\mathscr F}Xn)\cap \mathsf A(X)$, il n'est 
\emph{a fortiori}
pas adhérent à
$f(\qln{\mathscr F}X{n+1})\cap \mathsf A(X)$ ; 
il résulte alors de (5$^\ast$)
qu'il existe un indice $i$ et un antécédent $y'$ de $y$
sur $Y_i\times_{Z_i}S_i\subset \Lambda_i$.

\emph{Fin
de la démonstration.}
Soit $z$ l'image 
de $y'$ sur $Z_i$.
Montrons que $z$ n'est pas 
adhérent à $f_i(\Lambda_i) \cap \mathsf A(Z_i)$. On raisonne par l'absurde, en supposant qu'il l'est. 
Comme
$S_i$ est un diviseur de Cartier de $Z_i$, 
il ne rencontre pas $\mathsf A(Z_i)$, et $z$ est donc adhérent 
à
$f_i(\qln {\mathscr F_{Y_i}\oslash_{Z_i}S_i}{Z_i}n)\cap \mathsf A(Z_i\setminus S_i).$
Soit $\zeta$ un point de $f_i(\qln {\mathscr F_{Y_i}\oslash_{Z_i}S_i}{Z_i}n)\cap
\mathsf A(Z_i\setminus S_i)$. Par définition, le point $\zeta$
possède un antécédent $\eta$ par $f_i$ situé sur $\qln {\mathscr F_{Y_i}\oslash_{Z_i}S_i}{Z_i}n$ ; puisque $\zeta$ n'appartient pas à $S_i$, le point $\eta$ appartient en fait à 
$\qln {\mathscr F_{Y_i}}{Z_i}n$ ; et comme $Z_i\to X$ est quasi-étale et en particulier plat en dehors de $S_i$, ceci entraîne que l'image de $\eta$ sur $Y$
appartient à $\qln {\mathscr F}Xn.$ En conséquence, l'image $\xi$ de $\zeta$ sur $X$ appartient à $f(\qln {\mathscr F}Xn)$. Par ailleurs le point $\zeta$ appartient à $\mathsf A(Z_i\setminus S_i)$, ce qui
entraîne que $\xi\in \mathsf A(X)$
(lemme
\ref{dz-ecl-qet}). 
Il résulte de ce qui précède que $x$ appartient à l'adhérence
de $f(\qln {\mathscr F}Xn)\cap \mathsf A(X)$, ce qui est absurde. Par conséquent, $z$ n'appartient pas à l'adhérence de
$f_i(\Lambda_i)\cap \mathsf A(X)$. L'assertion $(\epsilon)$ assure alors
que $y'$ appartient à l'image de $Y_i^\ell\times_{Z_i^\ell}
S_i^\ell\to Y_i$ pour un certain $\ell$, 
et $y$ appartient de ce fait
à l'image de $Y_i^\ell\times_{Z_i^\ell}
S_i^\ell\to Y$.\qed

\section{Applications}
\label{s-appli}

\begin{lemm}\label{lem-plat-dimension}
Soit $Y\to X$ un morphisme entre espaces $k$-analytiques, 
soit $y$ appartenant à $\mathsf A(Y)$
et soit $x$ son image sur $X$. 

\begin{enumerate}[1]
\item 
On a 
$d_k(x)=d_k(y)-\dim_x Y_x$. 

\item Supposons de
plus $X$ réduit.
Les assertions suivantes sont équivalentes : 

\begin{enumerate}[j]
\item l'espace $Y$ est $X$-plat en $y$ ; 
\item
on a 
$\dim_y Y_x+\dim_x X=\dim_y Y$ ; 
\item le point $x$ appartient à $\mathsf A(X)$ ; 
\item l'espace $Y_{\mathrm{red}}$ est $X$-plat en $y$. 

\end{enumerate}
\end{enumerate}
\end{lemm}

\begin{proof}
Prouvons tout d'abord (1).
Soit $U$ l'ouvert de Zariski de $Y$ formé des points $z$ 
tels que $\dim_z Y_{f(z)}\leq \dim_y Y_x$ et $\dim_z Y\leq \dim_y Y$ ; il contient $y$. On a
$\dim_y Y=d_k(y)=d_{\hr x}(y)+d_k(x)$,
et si $z$ est un point de $U_x$ on a 
\[\dim_y Y\geq \dim_z Y\geq  d_k(z)=d_{\hr x}(z)+d_k(x),\]
ce qui montre que le maximum de $d_{\hr x}$
sur $U_x$ est atteint en $y$. Puisque $\dim U_x$ est égal à $\dim_y Y_x$ par définition 
de $U$, 
il vient $d_{\hr x}(y)=\dim_y Y_x$ ; par conséquent
$d_k(x)=d_k(y)-d_{\hr x}(y)=d_k(y)-\dim_y Y_x$.

Faisons l'hypothèse que $X$ est réduit, et montrons (2).
 Les implications (i)$\Rightarrow$(ii) et (iii)$\Rightarrow$(i) font 
 partie des propriétés de base de la platitude
 rappelées en \ref{rapp-plat}
 \emph{et sq.}

Supposons que (ii)
est vraie. On a d'après (1) l'égalité  
$d_k(x)=d_k(y)-\dim_y Y_x$ et au vu de l'hypothèse (ii), elle entraîne
que $d_k(x)=\dim_xX$, c'est-à-dire que $x\in \mathsf A(X)$, d'où (iii). 

Ainsi les assertions (i), (ii) et (iii) sont-elles équivalentes.
Comme la valeur de vérité de (iii) ne change pas si l'on remplace $Y$ par $Y_{\mathrm{red}}$
on a également (i)$\iff$(iv). 
\end{proof}

Le but de ce qui suit est de montrer un «théorème d'équidimensionalisation», fondé sur notre théorème
principal mais ne nécessitant aucune hypothèse de platitude générique. 

\begin{lemm}\label{lem-plat-dimconst}
Soit $f\colon Y\to X$ un morphisme d'espaces $k$-analytiques. 
Supposons qu'il existe $\delta$ tel que
$\Omega:=\{y\in Y, \dim_y f=\delta\}$
soit un ouvert de Zariski dense de $Y$.

\begin{enumerate}[1]
\item On a $\pl YX\subset \Omega$. 
\item On a $Y\setminus \Omega =\qln YX{\delta+1}$.
\item Si $Y$ est équidimensionnel, l'ensemble $\mathsf A(Y)$ ne rencontre pas
$\qln YX{\delta+1}^{\mathrm{sat}}$. 
\end{enumerate}
\end{lemm}

\begin{proof}
Commençons par prouver (1).
On peut remplacer $Y$ par $\pl YX$, et partant supposer $f$
plat. Posons
$\Omega=\{y\in Y, \dim_y f=\delta\}$ ; c'est par hypothèse un ouvert de Zariski dense
de $Y$ et il s'agit de démontrer que 
$\Omega=Y$. Soit $y\in Y$ et soit $x$ son image sur $X$. 
Soit $Z$ une composante irréductible de $X$ passant par $x$.
Il suffit de montrer que 
$\dim_y (Y\times_X Z)_x=\delta$.

Soit $Z'$ l'ouvert de Zariski de $X$
constitué des points de $Z$ qui n'appartiennent à aucune autre composante irréductible de $X$.
L'intersection de $Y\times_Z Z'$ avec $\Omega$ est Zariski-dense dans $Y\times_Z Z'$
(car ce dernier est un ouvert de Zariski de $Y$).
Par ailleurs
$Y\times_X Z'$ est Zariski-dense dans $Y\times_X Z$ ; en effet $Y\times_X Z\to Z$ est plat (pour donner un sens à cette affirmation il faut choisir une structure analytique sur $Z$, par exemple sa structure réduite) et l'image sur $Z$ de toute composante irréductible de $Y\times_X Z$ est donc Zariski-dense  dans $Z$ et de ce fait rencontre $Z'$, d'où notre
assertion. Il s'ensuit que $\Omega\cap (Y\times_Z Z')$ est dense dans $Y\times_X Z$.

On peut donc remplacer $X$
par $Z$, c'est-à-dire supposer que $X$ est irréductible. Soit $n$ sa dimension. Soit $T$ une composante irréductible de $Y$. Choisissons un point $t$ de $Y$ situé sur $T\cap \Omega$
et n'appartenant à aucune autre composante irréductible de $Y$
(c'est possible car $\Omega$ est Zariski-dense dans $Y$). On a alors $\dim_t T_{f(t)}=\dim_t Y_{f(t)}=\delta$. Par platitude il vient $\dim T=\dim_t Y=n+\delta$. Ceci valant pour tout $T$, l'espace
$Y$ est purement de dimension $n+\delta$. On en déduit à nouveau par platitude que la dimension relative de $Y$
sur $X$ est partout égale à $\delta$, ce qui achève de montrer (1). 

Montrons maintenant (2). Soit $y\in Y$. Si 
$y$
appartient à $\qln YX{\delta+1}$ on a \emph{a fortiori}
$\dim_y f\geq \delta+1$, et $y$ n'appartient donc
pas à $\Omega$. 
Réciproquement, supposons que $y$ n'appartient pas 
à $\Omega$. Le point $y$ vit alors sur une composante irréductible $T$ de $Y_{f(y)}$
de dimension $\geq \delta+1$. 
L'assertion (1) déjà démontrée assure que
$T\subset \ql YX$, ce qui entraîne que $y\in \ql YX_{\geq \delta+1}$. 

Montrons enfin (3). On suppose donc que $Y$ est purement de dimension $m$ pour un certain $m\geq 0$.
Soit $y\in \mathsf A(Y)$ et soit $x$ son image sur $X$.
On a $d_k(y)=m$, et par densité de $\Omega$ dans $Y$ on a
$\dim_y Y_x=\delta$ ; ceci entraîne que $d_{\hr x}(y)\leq \delta$ et qu'il existe $z\in Y_x$ tel que $d_{\hr x}(z)=\delta$. 

On a $d_k(x)=d_k(y)-\d_{\hr x}(y)\geq m-\delta$, et 
$d_k(x)=d_k(z)-\d_{\hr x}(z)\leq m-\delta$ car $d_k(z)\leq m$ 
puisque $\dim Y=m$. Par conséquent, $d_k(x)=m-\delta$. Pour 
tout $t\in Y_x$ on a donc 
\[m\geq d_k(t)=d_{\hr x}(t)+d_k(x)=d_{\hr x}(t)+m-\delta.\]
Ceci entraîne que $\dim Y_x\leq \delta$, et finalement que $\dim Y_x=\delta$. En particulier, $Y_x$
ne rencontre pas $\qln YX{\delta+1}$, et $y$ n'appartient
ainsi
pas
à $\qln YX{\delta+1}^{\mathrm{sat}}$. 
\end{proof}

\begin{theo}\label{thm-equidim}
Soit $f\colon Y\to X$ un morphisme entre espaces
$k$-analytiques $\Gamma$-stricts compacts. 
On suppose que 
l'application $y\mapsto \dim_y f$ prend des valeurs comprises
entre deux entiers $\delta$ et $d$ avec $\delta\leq d$, et
que la réunion des fibres de $f$ de dimension $\delta$ est dense dans $Y$. 
Il existe
une famille finie de
couples $[(S_i\hookrightarrow Z_i\to X, Y_i)]_i$
où $S_i\hookrightarrow Z_i\to X$ est un diagramme
appartenant à $\bkg k\Gamma {d-\delta}$
et où $Y_i$ est un domaine analytique compact et $\Gamma$-strict de $Y\times_X Z_i$, qui satisfait les propriétés suivantes : 

\begin{enumerate}[1]
\item pour tout $i$, l'adhérence analytique $Y'_i$ de $Y_i\setminus Y_i\times_{Z_i}S_i$ dans $Y_i$
est purement de dimension relative $\delta$ sur $Z_i$ ; 
\item la réunion des images des morphismes $Y'_i\to Y$ est égale à $Y$ ; 
\item la réunion des images des morphismes $Y_i\times_{Z_i}S_i\to Y$ est contenue dans la réunion des fibres de $f$
de dimension $>\delta$, et contient $\{y\in Y, \dim_y f>\delta\}$.
\end{enumerate}
\end{theo}

\begin{rema}
Les hypothèses du théorème sont notamment vérifiées lorsque les conditions suivantes sont satisfaites: $y\mapsto \dim_y f$ prend des valeurs comprises
entre deux entiers $\delta$ et $d$, l'ouvert $\{y\in Y, \dim_y f=\delta\}$
est dense dans $Y$, et $Y$ est équidimensionnel. En effet, si c'est le cas il résulte alors de l'assertion (3) du
lemme \ref{lem-plat-dimconst} 
que $\mathsf A(Y)$ ne rencontre pas 
$\qln YX{\delta+1}^{\mathrm{sat}}$, ce qui signifie que pour tout $y\in \mathsf A(Y)$, la fibre $f\inv(f(y))$ évite $\qln YX{\delta+1}$ ; en vertu de l'assertion (2) de \emph{loc. cit.}, ceci revient à dire que $f\inv(fy))$ est entièrement contenu dans $\{y\in Y, \dim_y f=\delta\}$, soit encore que la fibre $f\inv(f(y))$ est de dimension $\delta$. La réunion des fibres de $f$ qui sont de dimension $\delta$ contient donc $\mathsf A(Y)$, et est par conséquent dense dans $Y$. 
\end{rema}

\begin{proof}[Démontration du théorème \ref{thm-equidim}]
Le théorème \ref{flat-eclat}
appliqué lorsque $\mathscr F=\mathscr O_Y$
et $n=\delta+1$
assure l'existence d'une famille finie de
couples $[(S_i\hookrightarrow Z_i\to X, Y_i)]_i$
où $S_i\hookrightarrow
Z_i\to X$ est un diagramme
appartenant à $\bkg k\Gamma {d-\delta}$ 
et
où $Y_i$ est un domaine analytique compact et $\Gamma$-strict de $Y\times_X Z_i$, 
qui satisfait les propriétés suivantes : 
\begin{enumerate}[a]
\item pour tout $i$, l'adhérence analytique $Y'_i$ de $Y_i\setminus Y_i\times_{Z_i}S_i$ dans $Y_i$
est plate en dimensions $\geq \delta+1$ sur $Z_i$ ; 
\item pour tout $i$, l'ouvert $Y_i\setminus Y_i\times_{Z_i}S_i$ de $Y_i$ est contenu dans l'image réciproque de 
$Y\setminus \qln YX{\delta+1}$ ; 
\item pour tout $i$, le fermé $Y_i\times_{Z_i}S_i$ de  $Y_i$ est contenu dans l'image réciproque de 
$\qln YX{\delta+1}^{\mathrm{sat}}$ ; 
\item la réunion des images des morphismes $Y_i\to Y$ contient $Y\setminus \qln YX{\delta+1}^{\mathrm{sat}}$ ; 
\item la réunion des images des morphismes $Y_i\times_{Z_i}S_i\to Y$ contient tout point de $\qln YX{\delta+1}$
dont l'image sur $X$ n'est pas adhérente à $\mathsf A(X)\cap f(\qln YX{\delta+1})$. 
\end{enumerate}
Il suffit maintenant d'expliquer pourquoi ces propriétés entraînent (1), (2) et (3).
Posons $\Omega=\{y\in Y,\dim_y f=\delta\}$. 
Sous nos hypothèses, $\Omega$ est dense dans $Y$
(il contient même un ensemble dense de fibres de $f$) ;
le lemme \ref{lem-plat-dimconst} (2)
assure dès lors que $\Omega$ est le complémentaire de $\qln YX{\delta+1}$ 
dans $Y$ ; il s'ensuit que $\qln YX{\delta+1}^{\mathrm{sat}}$ est exactement la réunion des fibres de $f$ de dimension $>\delta$ ; compte-tenu de nos hypothèses, ceci entraîne la densité de $Y\setminus 
\qln YX{\delta+1}^{\mathrm{sat}}$ dans $Y$.

\subsubsection{Preuve de \textnormal{(1)}}
Comme $\Omega$
est le complémentaire de $\qln YX{\delta+1}$ 
dans $Y$, la propriété (b)
assure que $Y_i\setminus Y_i\times_{Z_i}S_i$ est situé 
au-dessus de $\Omega$ ; par conséquent, 
$Y_i\setminus Y_i\times_{Z_i}S_i$ est purement de dimension relative $\delta$ sur $Z_i$.

Puisque l'ouvert $(Y_i\setminus Y_i\times_{S_i}Z_i)$
de $Y'_i$ est analytiquement dense (et en particulier Zariski-dense)
dans ce dernier, on déduit de (a)
et du lemme \ref{lem-plat-dimconst} (2)
que $Y'_i$ est purement de dimension relative $\delta$ sur $Z_i$, et (1)
est établie. 

\subsubsection{Preuve de \textnormal{(2)}}
La conjonction de (c)
et de (d) assure que la réunion des images des morphismes $Y_i\setminus (Y_i\times_{Z_i}S_i)\to Y$
contient $Y\setminus \qln YX{\delta+1}^{\mathrm{sat}}$, dont on a vu plus
haut que c'était un ouvert dense de $Y$.  Comme cette réunion est par ailleurs compacte, l'assertion (2) en résulte. 

\subsubsection{Preuve de \textnormal{(3)}}
D'après la condition (c), la réunion des images des 
morphismes $Y_i\times_{Z_i}S_i\to Y_i$ est contenue dans
$\qln YX{\delta+1}^{\mathrm{sat}}$, 
c'est-dire dans la réunion des fibres de $f$ de dimension $\delta$. 

Soit $y$ un point de $Y$ tel que $\dim_y f>\delta$, c'est-à-dire tel que 
$y\in \qln YX{\delta+1}$. On sait d'après la propriété (2) déjà établie que $y$ possède un antécédent $z$ sur $Y'_i\subset Y_i$ pour un certain $i$ ; et la propriété (b) implique que $z\in Y_i\times_{Z_i}S_i$. 
\end{proof}

Si $f\colon Y\to X$ est
un morphisme
de type fini
entre schémas
noethériens avec $X$ intègre
et si l'ouvert de platitude de $f$ est vide, $f(Y)$ ne
contient pas le point générique de $X$, et est donc contenu dans un fermé de Zariski strict de $X$. 

En raison du «mauvais» comportement de la topologie de Zariski en géométrie analytique,
le résultat
analogue dans le monde des espaces
de Berkovich est grossièrement faux (penser à l'immersion dans un gros bidisque $D$ d'une courbe $C$ tracée sur un bidisque plus petit telle que $\adhz CD=D$, \emph{cf.} par exemple \cite{ducros2018}, 4.4 pour un exemple un peu plus
explicite). Mais on peut en un certain sens
le «rendre vrai par éclatements
et morphismes quasi-étales», comme en
atteste le théorème
\ref{thm-dominant-factor} ci-dessous. Et cela nous permettra
lors de la preuve du théorème \ref{thm-lipshitz}
de procéder à une réduction au cas génériquement plat par récurrence sur la dimension du but, 
comme on le fait couramment en théorie des schémas. 

\begin{theo}\label{thm-dominant-factor}
Soit $f\colon Y\to X$ un morphisme entre espaces
$k$-analytiques $\Gamma$-stricts compacts, avec $X$ réduit. 
On suppose que
$y\mapsto \dim_y f$ prend des valeurs comprises
entre deux entiers $\delta$ et $d$ avec $\delta\leq d$, 
et on fait l'hypothèse que $\pl YX=\emptyset$. 
Il existe alors 
une famille finie de couples
$[(S_i\hookrightarrow Z_i\to X, Y_i)]_i$
où $S_i\hookrightarrow
Z_i\to X$ est un
diagramme appartenant à $\bkg k\Gamma {d-\delta+1}$
et où 
$Y_i$ est un domaine analytique compact et $\Gamma$-strict de $Y\times_X Z_i$, qui satisfait les propriétés suivantes : 

\begin{enumerate}[1]
\item pour tout $i$, l'image du morphisme $Y_i\to Z_i$ est contenue dans $S_i$ ; 
\item la réunion des images des morphismes $Y_i\to Y$ est égale à $Y$. 
\end{enumerate}
\end{theo}

\begin{proof}
Le théorème \ref{flat-eclat}
appliqué avec $\mathscr F=\mathscr O_Y$
et $n=\delta$
assure l'existence d'une famille finie de
couples $[(S_i\hookrightarrow Z_i\to X, Y_i)]_i$
où $S_i\hookrightarrow Z_i\to X$ est
un diagramme appartenant à $\bkg k\Gamma {d-\delta+1}$ et où
$Y_i$ est un domaine analytique compact et $\Gamma$-strict de $Y\times_X Z_i$, 
qui satisfait entre autres les propriétés suivantes : 
\begin{enumerate}[a]
\item pour tout $i$, l'ouvert $Y_i\setminus Y_i\times_{Z_i}S_i$ de $Y_i$ est contenu dans l'image réciproque de 
$Y\setminus \qln YX\delta$ ; 
\item la réunion des images des morphismes $Y_i\times_{Z_i}S_i\to Y$ contient tout point de $\qln YX\delta$
dont l'image sur $X$ n'est pas adhérente à $\mathsf A(X)\cap f(\qln YX\delta)$. 
\end{enumerate}

Or comme $\pl YX$ est vide et comme $Y\to X$ est partout
de dimension relative supérieure ou égale à $\delta$, 
le fermé $\qln YX\delta$ de $Y$ est égal à $Y$ tout entier. Par conséquent, la propriété (a) équivaut
à l'assertion (1) 
et la propriété (b) revient à dire que la réunion des images des  morphismes
$Y_i\times_{Z_i}S_i\to Y$ contient tout point de $Y$ dont l'image sur $X$ n'est pas
adhérente à $f(Y)\cap \mathsf A(X)$. Mais comme $X$ est réduit ,
$f\inv(\mathsf A(X))$ est contenu dans
$ \pl YX$
et ce dernier est vide par hypothèse. 
Il vient $f(Y)\cap \mathsf A(X)=\emptyset$. L'assertion (2) s'en déduit. 
\end{proof}

Notre but est maintenant 
de donner une description générale des images de morphismes entre espaces 
$k$-analytiques compacts. Nous allons commencer par le cas génériquement plat. 

\begin{theo}\label{thm-images-genplat}
Soit $d$ un entier 
et soit $f\colon Y\to X$ un morphisme entre espaces 
$k$-analytiques compacts et $\Gamma$-stricts, à fibres de dimension $\leq d$. On suppose
que $X$ est réduit 
et que $\pl YX$ est dense dans $Y$. 
Il existe alors une famille finie
$(f_i\colon V_i\to X)$ de flèches de $\akg 
k\Gamma {d+2}$ possédant les propriétés suivantes : 
\begin{itemize}[label=$\diamond$]
\item chacune des $f_i$ admet une factorisation $V_i\hookrightarrow Z_i\to X$ où $Z_i\to X$
est
une flèche de $\akg k\Gamma{d+1}$ 
et où $V_i$ est un domaine
analytique compact et $\Gamma$-strict de $Z_i$ ; 
\item l'image $f(Y)$est la réunion des $f_i(V_i)$. 
\end{itemize}
De plus $f(Y)\cap \mathsf A(X)$ est dense dans $f(Y)$. 
\end{theo}

\begin{proof}
En appliquant le théorème \ref{flat-eclat}
avec $\mathscr F=\mathscr O_Y$ et $n=0$, on obtient  
l'existence d'une famille finie de
couples $[(S_i\hookrightarrow
Z_i\to X, Y_i)]_i$ où $S_i\hookrightarrow
Z_i\to X$ est un diagramme appartenant à
$\bkg k\Gamma {d+1}$ et où $Y_i$ est un domaine analytique compact et $\Gamma$-strict de $Y\times_X Z_i$, telle que les propriétés suivantes (entre autres) soient satisfaites : 
\begin{enumerate}[1]
\item pour tout $i$, l'adhérence analytique $Y'_i$ de
$Y_i\setminus (Y_i\times_{Z_i}S_i)$ dans $Y_i$ est plate sur $Z_i$ ; 
\item pour tout $i$, le fermé
$Y_i\times_{Z_i}S_i$ de $Y_i$
est contenu dans l'image réciproque de
$\mathsf Q(Y/X)^{\mathrm{sat}}$ ; 
\item la réunion des images des morphismes $Y_i\to Y$ contient $\mathsf P(Y/X)$.
\end{enumerate}

\subsubsection{}
Puisque
$\pl YX$ est dense dans $Y$, il contient
$\mathsf A(Y)$. Le
lemme  \ref{lem-plat-dimension} assure alors
que $f(\mathsf A(Y))\subset \mathsf A(X)$. Comme $\mathsf A(Y)$ est dense 
dans $Y$, ceci entraîne que l'intersection
$f(Y)\cap \mathsf A(X)$ est dense dans $f(Y)$. 

\subsubsection{}\label{sss-qsat-intvide}
Soit $y\in \mathsf A(Y)$. Par ce qui précède $f(y)\in \mathsf A(X)$ ; comme $X$
est réduit, il en résulte que $f$ est plat en tout point de $f\inv (f(y))$, si bien que
 $y\notin \ql YX^{\mathrm{sat}}$. Il s'ensuit que $Y\setminus \mathsf Q(Y/X)^{\mathrm{sat}}$
 est dense dans $Y$.

\subsubsection{}
Soit $y$ un point de $Y$ n'appartenant pas à $\ql YX^{\mathrm{sat}}$. Le point $y$ est en particulier
situé sur $\pl YX$, et est par conséquent d'après (3)
égal à l'image d'un point $y'$ de $Y_i$ pour un certain $i$. 
Et puisque $y$ n'appartient pas à $\ql YX^{\mathrm{sat}}$, la propriété (2)
assure que $y'$ n'est pas situé sur $Y_i\times_{Z_i}S_i$ ; il appartient
dès lors à $Y'_i$. 

Par ce qui précède, le sous-ensemble $Y\setminus \ql YX^{\mathrm{sat}}$
de $Y$ est contenu dans la réunion des images des morphismes $Y'_i\to Y$. Cette réunion est compacte,
et $Y\setminus \ql YX^{\mathrm{sat}}$ est dense dans $Y$ en vertu de \ref{sss-qsat-intvide}. Il s'ensuit que la réunion des images
des morphismes $Y'_i\to Y$ est égale à $Y$ tout entier ; de ce fait, $f(Y)$ est la réunion des images des morphismes
$Y'_i\to X$. 

\subsubsection{}
Pour tout $i$, la flèche $Y'_i\to X$ admet une factorisation $Y'_i\to Z_i\to X$ avec $Y'_i\to Z_i$ plat ; 
l'image $V_i$ de $Y'_i\to Z_i$ est un domaine analytique $\Gamma$-strict de $Z_i$, et le morphisme
composé $V_i\hookrightarrow Z_i\to X$ est donc une flèche de $\mathfrak A_{k,\Gamma}$, que l'on note $f_i$. Par ailleurs, puisque 
$S_i\hookrightarrow Z_i\to X$ appartient
à $\bkg k\Gamma{d+1}$, le morphisme
$Z_i\to X$ appartient à
$\akg k\Gamma {d+1}$, et $V_i\to Z_i\to X$
appartient à $\akg k\Gamma {d+2}$.
Comme 
$f(Y)$ est la réunion des images des morphismes $Y'_i\to X$, c'est aussi la réunion des $f_i(V_i)$. 
\end{proof}

\subsection{}Nous désignons par $\mathfrak B_{k,\Gamma}$ la classe des morphismes $Y\to X$
d'espaces $k$-analytiques admettant une factorisation 
\[Y=X_n\to X_{n-1}\to \ldots \to X_0=X
\]
où les $X_i$ sont tous compacts, $\Gamma$-stricts et réduits et où $X_i\to X_{i-1}$ est pour tout $i$ ou bien un éclatement, ou bien un morphisme quasi-étale, ou bien une immersion fermée.

\subsection{}
Soit $f\colon Y\to X$ une flèche de $\mathfrak A_{k,\Gamma}$ avec $X$ réduit. C'est alors une flèche de $\mathfrak B_{k,\Gamma}$. En effet, par définition de $\mathfrak A_{k,\Gamma}$ la flèche $f\colon Y\to X$ admet
 une factorisation 
$Y=X_n\to X_{n-1}\to \ldots \to X_0=X$ où les $X_i$ sont tous compacts et  $\Gamma$-stricts et où $X_i\to X_{i-1}$ est pour tout $i$ ou bien un éclatement, ou bien un morphisme quasi-étale. Il suffit pour conclure de s'assurer que chacun des $X_i$ est réduit. Or c'est évident car la source d'un éclatement (resp. d'un morphisme quasi-étale) de but réduit est encore un espace réduit. 

\begin{theo}\label{thm-lipshitz}
Soit $f\colon Y\to X$ un morphisme entre espaces $k$-analytiques compacts et $\Gamma$-stricts, 
avec $X$ réduit. 
Soit $Z$ l'adhérence de $\pl YX$ dans $Y$. 

\begin{enumerate}[1]
\item Il existe une famille finie $(f_i\colon V_i\to X)$ de flèches de $\mathfrak B_{k,\Gamma}$ telles que
$f(Y)$ soit égale à $\bigcup_i f_i(V_i)$. 
\item L'adhérence de $f(Y)\cap \mathsf A(X)$
dans $X$ est égale à $f(Z)$. 
\end{enumerate}
\end{theo}

\begin{rema}\label{rem-implat}
Seul intervient dans l'énoncé le fermé ensembliste $Z$ ; il n'est donc pas nécessaire de spécifier
une structure analytique sur celui-ci. Remarquons toutefois que si l'on en choisit une (par exemple la structure
réduite, ou celle d'adhérence analytique de   $\pl YX$) alors
$\pl ZX$ est dense dans $Z$ (et $f(Z)$ est donc de la forme décrite par
le théorème \ref{thm-images-genplat}). Pour le voir, on commence
par remarquer que l'ouvert
$\pl ZX$ est dense dans $Z$ si et seulement s'il contient $\mathsf A(Z)$, ce qui grâce au caractère réduit de $X$
et en vertu du
le lemme \ref{lem-plat-dimension}
est indépendant de la structure analytique choisie sur $Z$. 
Munissons donc celui-ci 
de la structure qui en fait l'adhérence analytique de $\pl YX$. Dans ce cas $\pl YX$ s'identifie
à un ouvert dense de l'espace $Z$ qui est plat sur $X$, d'où notre assertion. 

\end{rema}
\begin{proof}[Démonstration du théorème \ref{thm-lipshitz}]
On montre les énoncés (1) et (2) séparément. 

\subsubsection{Preuve de \textnormal{(1)}}
Soit $(X_i)$ une famille finie de sous-espaces
analytiques fermés réduits de $X$ recouvrant $X$. Il suffit de montrer le théorème
pour chacun des morphismes $Y\times_X X_i\to X_i$, ce qui permet de se ramener au cas où
$X$ est purement de dimension $n$ pour un certain $n\geq 0$. 
On raisonne maintenant par récurrence sur $n$. 

Si $n=0$ alors $X$ consiste en un ensemble fini de points rigides, et $f(Y)$ aussi. Par conséquent $f(Y)$ est un fermé de Zariski de $X$ (et également un domaine analytique compact et $\Gamma$-strict de $X$), et le théorème
est donc démontré dans ce cas. 

On suppose maintenant que $n>0$. On peut raisonner composante par composante sur $Y$ (en munissant chaque composante 
d'une structure analytique arbitraire) et donc se ramener au cas où celui-ci est irréductible.
On distingue maintenant deux cas.

\paragraph{Supposons que $\pl YX$ est non vide}
Dans ce cas $\pl YX$ est dense dans $Y$ et le théorème est alors une conséquence du théorème
\ref{thm-images-genplat} (et on peut même demander que les $f_i$ soient des flèches
de $\mathfrak A_{k,\Gamma}$). 

\paragraph{Supposons que $\pl YX$ est vide} Nous pouvons alors  appliquer le théorème
\ref{thm-dominant-factor}. Il assure en particulier l'existence d'une famille finie $(Z_i\to X)_i$ de flèches de $\mathfrak A_{k,\Gamma}$ et, pour chaque $i$, d'un domaine analytique
compact et $\Gamma$-strict $Y_i$ de $Y\times_X Z_i$
et d'un diviseur de Cartier 
$S_i$ de $Z_i$ tels que les deux propriétés suivantes soient satisfaites : 

\begin{itemize}[label=$\diamond$] 
\item pour tout $j$, l'image du morphisme $Y_i\to Z_i$ est contenue dans $S_i$ ; 
\item la réunion des images des morphismes $Y_i\to Y$ est égale à $Y$. 
\end{itemize}

L'ensemble $f(Y)$ est alors égal à la réunion des images des morphismes composés $Y_i\to Y \to X$. 
Fixons $i$. Par construction, $Y_{i,\mathrm{red}}\to X$ se factorise
par $S_{i,\mathrm{red}}$. L'espace $Z_i$ est purement de dimension $n$
(parce que c'est le cas de $X$, et en vertu de \ref{sss-composantes-eclate}).  En tant que diviseur de Cartier de l'espace $Z_i$, l'espace $S_i$ est alors
purement de
dimension $n-1$. L'hypothèse de récurrence assure donc l'existence
d'une famille finie $(T_{ij}\to S_{i,\mathrm{red}})_j$
de flèches de $\mathfrak B_{k,\Gamma}$ telles que l'image de 
$Y_{i,\mathrm{red}}\to S_{i,\mathrm{red}}$ soit la réunion des
images des morphismes $T_{ij}\to S_{i,\mathrm{red}}$ pour $j$ variable. 

Il résulte de ce qui précède que $f(Y)$ est la réunion des images des morphismes composés
\[T_{ij}\to S_{i,\mathrm{red}}\to X\] pour $i$ et $j$ variables.
Or ces morphismes composés appartiennent à 
$\mathfrak B_{k,\Gamma}$ par construction, ce qui achève de démontrer (1). 

\subsubsection{}
Montrons maintenant (2).
Soit $x\in  f(Y)\cap \mathsf A(X)$. Comme $X$ est réduit, la fibre $f\inv(x)$ est contenue dans $\pl YX\subset Z$ ; 
puisque cette fibre est non vide par hypothèse, $x\in f(Z)$. Ainsi, $f(Y)\cap \mathsf A(X)=f(Z)\cap \mathsf A(X)$. 

D'autre part, soit $y$ un point de $\mathsf A(Y)$ situé sur $Z$. Il appartient à 
$\pl YX$ par définition de $Z$, et $f(y)$ appartient donc à $\mathsf A(X)$ par le lemme \ref{lem-plat-dimension}. 
Comme $\mathsf A(Y)\cap Z$ est dense
dans $Z$ il en résulte que $f(Z)\cap \mathsf A(X)=f(Y)\cap \mathsf A(X)$ est dense dans $f(Z)$, ce qui achève de
démontrer (2). 
\end{proof}

\subsection{Commentaires}
\label{ss-comment-images}
L'intérêt d'écrire $f(Y)$ comme réunion des $f_i(V_i)$ est que
chacun des morphismes $f_i\colon V_i\to X$ est \emph{a priori} nettement plus simple que $f$, étant composé d'éclatements, morphismes quasi-étales et immersions fermées (notons que ceci entraîne que $\dim V_i\leq \dim X$, indépendamment de la dimension de $Y$).
On peut certes  juger que le gain reste limité, les $f_i(V_i)$
ne semblant malgré tout pas si aisés à appréhender. Mais il y a néanmoins des chances que ce résultat soit à peu près optimal : un certain nombre de travaux de théorie des modèles, dont nous discutons ci-dessous, laissent en effet penser qu'il y a peu d'espoir d'obtenir 
une description beaucoup plus tangible des images de morphismes.

\subsection{Liens avec les travaux de Cluckers,  Lipshitz, Robinson\ldots}
\label{comment-final}
Les images de morphismes arbitraires entre espaces strictement $k$-analytiques compacts ont
déjà fait l'objet de nombreux travaux du point de vue de la théorie des modèles
(\cite{lipshitz1993}, \cite{schoutens1994},
\cite{lipshitz-rob2000}, \cite{cluckers-l2017}\ldots), qui abordent le
problème
sous l'angle de l'élimination des
quantificateurs. 

Plus précisément, faisons l'hypothèse que $k$ est algébriquement clos
non trivialement valué, et considérons
le langage $\mathcal L_k\an$ obtenu à partir du  langage $\mathcal L:=(+,\times, 0,1, x\mapsto x\inv, |)$
des corps valués en lui adjoignant
un symbole
pour chaque série convergente appartenant à $k\{T_1,\ldots, T_n\}$ (pour $n$ variable). 
Le corps $k$ s'interprète naturellement comme une structure pour $\mathcal L_k\an$, et il est
naturel de se demander si, à l'instar de {\sc acvf}, la théorie de $k$ admet l'élimination des quantificateurs
dans le langage $\mathcal L_k\an$. La réponse est négative ; on trouvera un contre-exemple
à la section 4 de \cite{cluckers-l2017}. Mais on peut montrer l'existence
d'une famille $\mathcal F=(\mathcal F_{nm})_{(n,m)}$
où $\mathcal F_{nm}$ est pour tout $(n,m)$ un ensemble de fonctions $k$-analytiques bornées
sur le produit du polydisque unité fermé de dimension $n$ par le polydisque unité ouvert
de dimension $m$, telle que : 

\begin{itemize}[label=$\diamond$] 
\item $\mathcal F_{n0}=k\{T_1,\ldots, T_n\}$ pour tout $n$ ; 
\item dans le langage $\mathcal L^{\mathcal F}\supset \mathcal L_k\an$
obtenu en adjoignant à $\mathcal L$ un symbole par élément de $\mathcal F_{nm}$ (pour $(n,m)$ variable),
la théorie de $k$ admet l'élimination des quantificateurs. 
\end{itemize}

La preuve initiale de ce fait est due à Lipshitz (\cite{lipshitz1993}, thm. 3.8.2), mais la famille
$\mathcal F=(\mathcal F_{nm})$ dont il montre l'existence n'est ni canonique ni vraiment explicite.
Dans le
travail plus récent \cite{cluckers-l2017}, Cluckers et Lipshitz construisent une autre famille 
$(\mathcal F_{nm})$  satisfaisant aux exigences ci-dessus
bien plus aisée à appréhender :
en gros, $\mathcal F_{nm}$ s'obtient en adjoignant à $k\{T_1,\ldots, T_n, S_1,\ldots, S_m\}$ les solutions 
dans $k\{T_1,\ldots, T_n\}[\![S_1,\ldots, S_m]\!]$ de
certains «systèmes polynomiaux henséliens»
à coeffcients dans $k\{T_1,\ldots, T_n,S_1,\ldots, S_m\}$ (voir la section 3 de \cite{cluckers-l2017}
pour davantage de détails) ; désignons par $\mathcal L^{\mathrm h}_k$
le langage $\mathcal L^{\mathcal F}$ lorsque $\mathcal F$ est la famille de Cluckers et Lipshitz. 

Nous pensons (mais nous ne l'avons pas vérifié) que le théorème \ref{thm-lipshitz}
ci-dessus est essentiellement (lorsque $k$ est algébriquement clos non trivialement valué et lorsque $\Gamma=\{1\}$)
une reformulation géométrique de l'élimination des quantificateurs dans $\mathcal L_k^{\mathrm h}$, les
systèmes polynomiaux henséliens de Cluckers et Lipshitz
devant peu ou prou correspondre à nos morphismes quasi-étales. Le contre-exemple de
\cite{cluckers-l2017} évoqué
plus haut suggère qu'on ne peut pas se passer des systèmes polynomiaux henséliens pour avoir élimination des
quantificateurs, et donc qu'on ne peut sans doute pas
se passer des morphismes
quasi-étales pour décrire les images de morphismes et ni sans doute, par voie de conséquence, pour aplatir les faisceaux cohérents.

Signalons pour terminer que si l'on note $\mathcal L^\dagger_k$ le langage obtenu en adjoignant à $\mathcal L$
un symbole pour chaque série \emph{surconvergente}
de $k\{T_1,\ldots, T_n\}$ (pour $n$ variable) alors la théorie de $k$ dans le langage $\mathcal L^\dagger_k$
admet l'élimination des quantificateurs. C'est un résultat établi par Schoutens
dans \cite{schoutens1994}, dont  Florent Martin a donné une formulation et une preuve géométriques 
(\cite{martin2016}, thm. 1.35)  ; notre théorème \ref{thm-lipshitz} est en quelque sorte
au langage $\mathcal L_k^{\mathrm h}$ ce que celui de Martin
est au langage $\mathcal L_k^\dagger$.

\bibliographystyle{smfalpha}
\bibliography{aducros}
 
\end{document}